\documentclass[11pt]{article}
\usepackage[margin=1in]{geometry}
\usepackage[utf8]{inputenc}
\usepackage[T1]{fontenc} 
\usepackage{mathtools}
\usepackage{subcaption}

\usepackage{xcolor} 
\usepackage{color} 
\definecolor{niceRed}{RGB}{190,38,38} 
\definecolor{niceYellow}{HTML}{f5b400}
\definecolor{blueGrotto}{HTML}{059DC0}
\definecolor{royalBlue}{HTML}{057DCD}
\definecolor{navyBlue}{HTML}{0B579C}
\definecolor{limeGreen}{HTML}{81B622}
\definecolor{nicePurple}{HTML}{9c27b0}
\definecolor{lightRoyalBlue}{HTML}{def2ff}  
\definecolor{gold}{HTML}{ffa300}

\newcommand*{\white}[1]{\textcolor{white}{#1}}

\usepackage{color-edits}  
 
\usepackage{hyperref}
\hypersetup{
  colorlinks = true, 
  urlcolor = {blueGrotto},
  linkcolor = {royalBlue},
  citecolor = {navyBlue}
}
\usepackage[numbers,sort&compress]{natbib}
\makeatletter
\providecommand{\abx@aux@refcontext}[1]{}
\providecommand{\abx@aux@cite}[2]{}
\providecommand{\abx@aux@segm}[3]{}
\providecommand{\abx@aux@defaultrefcontext}[2]{}
\makeatother

\usepackage{booktabs} 
\usepackage{multicol} 
\usepackage{multirow} 
\usepackage{makecell} 
\usepackage{longtable}

\usepackage{graphicx}
\usepackage{float} 
\usepackage{tikz}
\usepackage{pgfplots}
\pgfplotsset{compat=1.17}
\usepgfplotslibrary{fillbetween}

\usepackage{physics} 
\usepackage{amsmath}
\numberwithin{equation}{section}
\allowdisplaybreaks
\usepackage{amssymb}
\usepackage{amsthm}
\usepackage{bbm}
\usepackage{blkarray}
\usepackage{nameref}
\usepackage{nicefrac}
\usepackage{dsfont}
\usepackage{pifont} 
\usepackage{thm-restate} 
\usepackage[capitalise,noabbrev,nameinlink]{cleveref}

\usepackage{enumerate} 
\usepackage[inline,shortlabels]{enumitem} 
\usepackage{tcolorbox}
\usepackage[framemethod=TikZ]{mdframed}
\mdfsetup{%
backgroundcolor=yellow!00, , 
roundcorner=4pt,
linewidth=1pt}
\usepackage{changepage} 

\usepackage{algorithm}
\usepackage{algpseudocode}[1]

\usepackage{mathrsfs} 
\usepackage{soul} 

\theoremstyle{plain} 
\newtheorem{theorem}{Theorem}[section]
\newtheorem{corollary}[theorem]{Corollary}

\newtheorem{proposition}[theorem]{Proposition}
\newtheorem{lemma}[theorem]{Lemma}

\newtheorem{assumption}{Assumption}[section]

\newtheorem{inftheorem}{Informal Theorem}

\newtheorem{definition}{Definition}[section]

\newtheorem*{definition*}{Definition}

\theoremstyle{definition} 

\newtheorem{remark}[theorem]{Remark}

\theoremstyle{remark}

\AfterEndEnvironment{definition}{\noindent\ignorespaces}
\AfterEndEnvironment{claim}{\noindent\ignorespaces}
\AfterEndEnvironment{assumption}{\noindent\ignorespaces}
\AfterEndEnvironment{lemma}{\noindent\ignorespaces}
\AfterEndEnvironment{theorem}{\noindent\ignorespaces}
\AfterEndEnvironment{proposition}{\noindent\ignorespaces}
\AfterEndEnvironment{fact}{\noindent\ignorespaces}
\AfterEndEnvironment{question}{\noindent\ignorespaces}
\AfterEndEnvironment{corollary}{\noindent\ignorespaces}
\AfterEndEnvironment{model}{\noindent\ignorespaces}
\AfterEndEnvironment{remark}{\noindent\ignorespaces}
\AfterEndEnvironment{proof}{\noindent\ignorespaces}
\AfterEndEnvironment{fact}{\noindent\ignorespaces}
\AfterEndEnvironment{minftheorem}{\noindent\ignorespaces}
\AfterEndEnvironment{inftheorem}{\noindent\ignorespaces}
\AfterEndEnvironment{maintheorem}{\noindent\ignorespaces}
\AfterEndEnvironment{restatable}{\noindent\ignorespaces}
\AfterEndEnvironment{infassumption}{\noindent\ignorespaces}
\AfterEndEnvironment{conjecture}{\noindent\ignorespaces}

\crefname{section}{Section}{Sections}
\crefname{theorem}{Theorem}{Theorems}
\crefname{lemma}{Lemma}{Lemmas}
\crefname{definition}{Definition}{Definitions}
\crefname{conjecture}{Conjecture}{Conjectures}
\crefname{corollary}{Corollary}{Corollaries}
\crefname{construction}{Construction}{Constructions}
\crefname{conjecture}{Conjecture}{Conjectures}
\crefname{claim}{Claim}{Claims}
\crefname{observation}{Observation}{Observations}
\crefname{proposition}{Proposition}{Propositions}
\crefname{fact}{Fact}{Facts}
\crefname{question}{Question}{Questions}
\crefname{problem}{Problem}{Problems}
\crefname{remark}{Remark}{Remarks}
\crefname{example}{Example}{Examples}
\crefname{equation}{Equation}{Equations}
\crefname{appendix}{Appendix}{Appendices}
\crefname{algorithm}{Algorithm}{Algorithms}
\crefname{model}{Model}{Models}
\crefname{figure}{Figure}{Figures}
\crefname{infthm}{Informal Theorem}{Informal Theorems}
\crefname{infassumption}{Informal Assumption}{Informal Assumptions}
\crefname{minftheorem}{Main Informal Theorem}{Main Informal Theorems}
\crefname{maintheorem}{Main Theorem}{Main Theorems}
\crefname{assumption}{Assumption}{Assumptions}
\crefname{case}{Case}{Cases}

\newlist{asmpenum}{enumerate}{1} 
\setlist[asmpenum]{label={\arabic*.},ref=\theassumption.{\arabic*}}
\crefname{asmpenumi}{Assumption}{Assumptions}

\newlist{infasmpenum}{enumerate}{1} 
\setlist[infasmpenum]{label={\arabic*.},ref=\theassumption.{\arabic*},leftmargin=20pt}
\crefname{infasmpenumi}{Informal Assumption}{Informal Assumptions}

\usepackage{etoolbox}

\BeforeBeginEnvironment{subenvironment}{\white{.}\\ \vspace{-10mm}}
\AfterEndEnvironment{subenvironment}{\vspace{4mm}}

\makeatletter
\newcommand{\tagnum}[2]{%
    \refstepcounter{equation}%
    \tag{#1) \ (\theequation}%
    \protected@write \@auxout {}{%
        \string \newlabel {#2}{{\theequation}{\thepage}{}{equation.\theequation}{}}%
    }%
}
\makeatother

\makeatletter
\let\orgdescriptionlabel\descriptionlabel
\renewcommand*{\descriptionlabel}[1]{%
  \let\orglabel\label
  \let\label\@gobble
  \phantomsection
  \edef\@currentlabel{#1\unskip}%
  \let\label\orglabel
  \orgdescriptionlabel{#1}%
}
\makeatother

\def\abs#1{|#1|}

\newcommand{\inangle}[1]{\left\langle#1\right\rangle}

\let\norm\relax
\newcommand{\norm}[1]{\ensuremath{\lVert #1 \rVert}}

\newcommand{\R}{\mathbb{R}}
\newcommand{\PP}{\mathbb{P}}

\newcommand{\E}{\operatorname{\mathbb{E}}} 
\newcommand{\Ex}{\E}

\newcommand{\cov}{\ensuremath{\operatornamewithlimits{\rm Cov}}}
\let\var\relax
\newcommand{\var}{\ensuremath{\operatornamewithlimits{\rm Var}}}
\newcommand{\Var}{\var}

\newcommand{\argmax}{\operatornamewithlimits{arg\,max}}

\newcommand{\one}{\mathbbm{1}}

\newcommand\simiid{\overset{\text{i.i.d.}}{\sim}}
\newcommand{\CPI}{C_{\mathsf{PI}}}
\newcommand{\mugood}{\mu|_{\hat\Theta_{\delta}}}
\DeclareMathOperator{\osc}{osc}

\newcommand{\deq}{\coloneqq}

\renewcommand{\epsilon}{\varepsilon}
\newcommand{\T}{\mathsf{T}}
\makeatletter
\newcommand*{\tran}{{\mathpalette\@tran{}}}
\newcommand*{\@tran}[2]{\raisebox{\depth}{$\m@th#1\intercal$}}
\makeatother

\mathchardef\NABLA"272
\newcommand*{\Nabla}{\boldsymbol\NABLA}
\let\nabla\Nabla

\renewcommand{\hat}{\widehat}

\renewcommand{\bar}{\overline}

\renewcommand{\tilde}{\widetilde}

\newcommand{\sN}{\mathsf{N}}

\newcommand{\Cgr}{C^{\mathsf{LG}}}
\newcommand{\error}{\mathtt{Err}}
\newcommand{\Pois}{\mathsf{Poisson}}

\title{Sampling from Constrained Gibbs Measures: with Applications to High-Dimensional Bayesian Inference} 
\author{Ruixiao Wang \\ \texttt{\small ruixiao.wang@yale.edu} \and Xiaohong Chen \\ \texttt{\small xiaohong.chen@yale.edu} \and Sinho Chewi \\ \texttt{\small sinho.chewi@yale.edu}}
\date{Yale University \\[1em] \today}

\begin{document}

\maketitle
\begin{abstract}
    This paper considers a non-standard problem of generating samples from a low-temperature Gibbs distribution with \emph{constrained} support, when some of the coordinates of the mode lie on the boundary. These coordinates are referred to as the non-regular part of the model. We show that in a ``pre-asymptotic'' regime in which the limiting Laplace approximation is not yet valid, the low-temperature Gibbs distribution concentrates on a neighborhood of its mode. Within this region, the distribution is a bounded perturbation of a product measure: a strongly log-concave distribution in the regular part and a one-dimensional exponential-type distribution in each coordinate of the non-regular part. Leveraging this structure, we provide a non-asymptotic sampling guarantee by analyzing the spectral gap of Langevin dynamics. Key examples of low-temperature Gibbs distributions include Bayesian posteriors, and we demonstrate our results on three canonical examples: a high-dimensional logistic regression model, a Poisson linear model, and a Gaussian mixture model.
\end{abstract}
%

\section{Introduction}\label{sec:intro}
We consider sampling from a high-dimensional low-temperature Gibbs distribution that is defined over a constrained domain. More specifically, we assume that the distribution $\mu$ of interest admits a Lebesgue density of the following form:
    \begin{equation}
        \mu(d\theta) = \frac{\pi(\theta)\exp{n \ell(\theta)} \one_{[0,\infty)^d}(\theta)}{\int_{[0,\infty)^d}\pi(\theta)\exp{n \ell(\theta)}\, d\theta}\,d\theta\,, \label{eq:constrained_gibbs_density}
    \end{equation}
    where $\pi : [0,\infty)^d \to [0,\infty)$, and $\ell$ is $C^2$ on the interior of $[0, \infty)^d$ and admits a unique global maximizer $\hat \theta$.  
    Given access to evaluations of $\pi$ and $\ell$ and their gradients, our aim is to generate samples whose distribution is close to $\mu$  in total variation distance.
    Here, $n \ge 0$ can be interpreted as an inverse temperature parameter, but we use the notation $n$ since it will take on the interpretation of a \emph{sample size} in our primary application. 

    Indeed, a first motivation for studying such distributions comes from the field of Bayesian inference
    in which $\mu$ is the posterior distribution\footnote{Our results generally extend to pseudo-posterior distributions, in which the log-likelihood is replaced by a general criterion function. For ease of discussion, we stick with the terminology ``posterior''.} and is thus random, depending on the observed data.
    In this case, $\pi$ is the prior, and $\ell = \ell_n$ is the averaged log-likelihood over $n$ observations.
    
 Bayesian inference provides a principled framework for uncertainty quantification, but it also raises natural computational questions, as the distribution~\eqref{eq:constrained_gibbs_density} can be quite complicated.
 Toward this end, the predominant approach is to apply sampling algorithms based on the Markov Chain Monte Carlo (MCMC) paradigm.
 From a theoretical standpoint, the justification of these methods hinges on an analysis of their convergence behavior.

A second motivation comes from the study of rare events. 
Specifically, if $E$ denotes the rare event of interest, and the mode of the distribution lies outside of $E$, then the distribution conditioned on $E$ can be modelled in the form~\eqref{eq:constrained_gibbs_density} where $n$ is a parameter controlling the rarity.
Although we do not discuss this and further applications in this work, the reason why we work at the level of generality of~\eqref{eq:constrained_gibbs_density} is because such distributions are ubiquitous in applications.

When $n$ tends to infinity with $d$ fixed, the classical Laplace approximation and non-regular variants thereof provide accurate asymptotic approximations to the distribution $\mu$, and hence sampling becomes easy.
Recent works (see \cref{sec:related_work}) have also established non-asymptotic versions of these results, although they generally require at least $n \gg d^2$.
In contrast, we are particularly interested in a ``pre-asymptotic'' regime in which the limiting approximation is not yet accurate, e.g., $d^2 \gg n \gg d$.
    In this setting, the sampling problem is rich, yet---as we will show---remains surprisingly tractable.

\paragraph{Non-regular models.}
In this paper, we study the setting in which the distribution $\mu$ is supported on a constraint set $\mathcal C$.
For simplicity, we confine ourselves to the model case $\mathcal C = [0,\infty)^d$.
We expect that our arguments, being local in nature, extend to the case when $\mathcal C$ is, e.g., an open domain with a smooth boundary via the standard technique of ``straightening out the boundary'', although we leave this for future work.
In fact, $[0,\infty)^d$ is in a sense more challenging because multiple coordinates of the mode can equal $0$, corresponding to multiple non-regular directions (whereas the boundary of an open domain has but one); one says that $[0,\infty)^d$ is ``stratified''.

From the standpoint of practical applications, constraint sets arise frequently, either because of the natural interpretation of the parameters (e.g., the mean of a Poisson distribution is always non-negative), or because of problem-specific constraints.

For example, in economics, constraints are present in auction models \cite{PAARSCH1992191, DONALD2002305, 4284864d-4807-33db-bf07-abcba0681499, 7a7a3cca-73b5-3d1e-89e3-f917af691716}, and some structural models in labor economics \cite{FLINN1982115}, where the support of the model depends on the unknown parameter. Simple examples include a one-sided uniform distribution $\mathsf{Unif}([0, \theta^\star])$, and Pareto distributions which are used to model latent cost distributions.
Such models are referred to as \emph{non-regular} and underscore the need to develop theory encompassing these settings.

A key challenge, however, is that non-regular models can exhibit starkly different phenomena from the regular (smooth, unconstrained) case.
In asymptotic statistics, the maximum likelihood estimator is often inefficient~\cite{Vaart_1998}.
Moreover, the classical asymptotic approximation to~\eqref{eq:constrained_gibbs_density} is no longer Gaussian, but includes truncated Gaussian and gamma components~\cite{bochkinaBernsteinMisesTheoremNonregular2014}.
What, then, is the behavior of $\mu$ in the pre-asymptotic phase?
Here, the distribution $\mu$ does not admit a simple approximation; instead, our approach will be to establish structural properties of $\mu$ which imply that MCMC algorithms are rapidly mixing, at least locally.

\paragraph{Pre-asymptotic guarantees for MCMC via Poincar\'e inequalities.}

Given access to evaluations of the log-density $\log \mu$ and its gradients, an MCMC algorithm forms a Markov chain whose stationary distribution approximates $\mu$. Ergodicity ensures that the law of the iterate of the Markov chain converges to the stationary distribution as the number of iterations tends to infinity. However, MCMC is computationally expensive, and practical performance can severely degrade due to slow mixing or poor scaling with the dimension. It is therefore important to provide theoretical foundations for the convergence of MCMC algorithms in high-dimensional settings.

Convergence guarantees are often provided under the assumption of strong log-concavity~\cite{Chewi25Book}, which however is too restrictive to capture many distributions of interest.
However, in the regular case,~\cite{nickl2022} showed that for many high-dimensional Bayesian inverse problems, the posterior is \emph{locally} strongly log-concave; see the related works section (\Cref{sec:related_work}) for further details. Furthermore, the onset of local strong log-concavity typically occurs in the pre-asymptotic regime $n \gg d$.
Their result implies that standard MCMC methods, such as the Langevin diffusion, mix rapidly---in the sense of producing a sample from the posterior in time which is polynomial in both $d$ and $n$---\emph{when initialized in a local region around the mode}.

Actually, the initialization assumption is fundamental.
As shown in~\cite{Ban+23BarrierMCMC}, even a unimodal posterior can exhibit a ``free energy barrier'', resulting in exponentially slow mixing when initialized far away from the mode.
Nevertheless, these results paint a hopeful picture for MCMC\@, as they suggest that sampling is tractable as soon as point estimation is.
Since uncertainty quantification is generally more difficult than point estimation, and since the latter is often solvable both via powerful heuristics and via a growing body of rigorous results, the takeaway message is that high-dimensional Bayesian inference appears to be surprisingly tractable.

In our work, we build upon these insights and aim to provide a theory for non-regular models.
However, in this setting, local strong log-concavity fails, and we must seek a different approach.
Indeed, even in the asymptotic picture, the limiting distribution contains exponential-type distributions which are not strongly log-concave.
Despite this, it is well-known that the Langevin dynamics mixes rapidly for an exponential target, because it satisfies a functional inequality known as a \emph{Poincar\'e inequality}.
Intuitively, a Poincar\'e inequality corresponds to a
a spectral gap for the generator of the Langevin dynamics; see~\cite{bakry2014analysis, Chewi25Book} for a detailed treatment.

Taking inspiration from this picture, we establish that the distribution~\eqref{eq:constrained_gibbs_density}, under appropriate assumptions and when localized to a neighborhood around the mode, satisfies a Poincar\'e inequality in the pre-asymptotic phase.
In turn, this extends the story in the regular setting to the non-regular one, by showing that standard MCMC methods mix rapidly when suitably initialized.
In summary:
\begin{itemize}
    \item In the regular case, local strong log-concavity kicks in before the limiting Gaussian approximation.
    \item In the non-regular case, a \emph{local Poincar\'e inequality} kicks in before the limiting Gaussian-exponential approximation.
\end{itemize}

\paragraph{Our contributions.}
Our contributions are twofold: first, we obtain deterministic conditions under which we can sample from the Gibbs distribution~\eqref{eq:constrained_gibbs_density} with non-asymptotic convergence guarantees; second, we provide a suite of conditions under which the random posterior distribution in a Bayesian setting satisfies our deterministic conditions.
Taken together, our results yield various sampling corollaries for concrete non-regular Bayesian problems, which we explore in~\Cref{section:examples}. 

In the first part, our main result (see~\Cref{infthm:poincare}) is a Poincar\'e inequality for the Gibbs density~\eqref{eq:constrained_gibbs_density} when restricted to a ``good set'', consisting of a ball of radius $\asymp n^{-1/2}$ in the regular coordinates (coordinates in which $\hat\theta_i > 0$), and a box with side length $\asymp n^{-1}$ in the non-regular coordinates (coordinates in which $\hat\theta_i = 0$).
Our proof decomposes the Gibbs density on the good set into three components: $(i)$ a strongly log-concave distribution on the regular coordinates, $(ii)$ a one-dimensional perturbed exponential distribution for each non-regular coordinate, and $(iii)$ a perturbation term coupling the two parts. We show that the distributions associated with $(i)$ and $(ii)$ satisfy a Poincar\'e inequality with dimension-free constants. By the tensorization property of the Poincar\'e inequality and its stability under bounded perturbations, this yields a dimension-free Poincar\'e constant for the Gibbs density restricted to the good set, which implies fast mixing guarantees for MCMC algorithms truncated to the good set.
We then show that the good set has overwhelming probability under~\eqref{eq:constrained_gibbs_density}, justifying the truncation.

In the second part, we perform a frequentist analysis of a Bayesian
setup in which $\ell = \ell_n$ is an average of $n$ i.i.d.\ realizations of a random function.
Letting $\ell^\star \deq \E \ell_n$ denote the population version of the criterion function, we place assumptions on $\ell^\star$ ensuring that the assumptions in the first part hold with high probability.
The main assumption is that $\ell^\star$ admits a unique global minimizer, together with standard regularity and smoothness conditions.

We also note here two caveats of our approach.
First, as discussed above, our results only imply the existence of fast samplers when the good set is known, which amounts to approximate knowledge of the mode.
When the target density is highly non-log-concave, identifying the mode can also be challenging, but this assumption is natural in light of the lower bound of~\cite{Ban+23BarrierMCMC} and is justified by the widespread prevalence of practically successful methods for optimization/point estimation, e.g., \cite{doi:10.1073/pnas.2519845123}, where they deal with the challenge of highly non-log-concave settings and boundary modes. 

Second, if $d_0$, $d_1$ denote the number of regular and non-regular coordinates respectively, our polynomial-time sampling guarantees and concentration result for the good set only hold\footnote{Here, and subsequently, we assume that $d_1 \ge 1$ to avoid trivialities; otherwise, our bounds should be stated in terms of $1\vee d_1$ instead of $d_1$.} when $n \gg d_0 d_1$.
In this paper, our goal is to focus on settings in which the number of non-regular coordinates is small; recall, for instance, that when the constraint set is an open domain with a smooth boundary, this morally corresponds to $d_1 = 1$. When $d_1$ is constant or growing slowly with $n$, e.g., $d_1 \lesssim \log n$, then our condition reduces to $n \gg d_0$ up to logarithmic factors. In contrast, prior quantitative works in the non-regular setting only consider $d_1 = 1$~\cite{katsevich2025asymptoticanalysisrareevents}.

\subsection{Related works}\label{sec:related_work}

\paragraph{Laplace approximation.}
Approximation of low-temperature Gibbs measures by Gaussians is classically known as the Laplace approximation, and it remains an active area of research~\cite{33dbad4f-5b73-3bc9-ac01-477e2b81fa2a, doi:10.1137/1.9780898719260, doi:10.1137/S0036144504446175, kirwin2010higherasymptoticslaplacesapproximation,Nemes_2013, BarDrtTan16Laplace, Ogden21Laplace, tang2023laplacesaddlepointapproximationshigh, katsevich2024laplaceapproximationaccuracyhigh}.
In high-dimensional settings, a central question is to determine the rate at which the dimension $d$ can grow with the parameter $n$ while ensuring that the Laplace approximation remains valid.
Recently, \cite{kasprzak2025goodlaplaceapproximationbayesian} established the sufficient condition $n \gg C_d\, d^2$ for some model-specific constant $C_d > 0$ which could possibly depend on the dimension $d$. This result was further refined and complemented with necessary conditions in~\cite{katsevich2024laplaceasymptoticexpansionhigh}, thereby arriving at a tight characterization of its validity.

On the other hand, there has been limited work on the Laplace expansion in non-regular models (e.g., constrained measures). A notable exception is~\cite{math8040479}, which studies the convergence rate of a Laplace-type integral approximation near the boundary of the constraint set.
Recently, in the context of rare event sampling, the work of \cite{katsevich2025asymptoticanalysisrareevents} established the validity of a Laplace-type approximation in the regime where $n \gg d^2$. The rare event $E$ they consider has a smooth boundary, which is locally homeomorphic to a half-space.
  Our work complements theirs by providing non-asymptotic guarantees for sampling algorithms in the regime where $n$ is comparable to $d$, and our results apply to the case with more than one non-regular direction.

\paragraph{Bernstein--von Mises (BvM) theorems.}
When the Gibbs measure arises as a posterior distribution formed from $n$ samples, the celebrated Bernstein--von Mises (BvM) theorem performs a frequentist analysis of its limiting behavior assuming that the data is generated from a ground truth parameter.
It states that
under certain regularity conditions, the posterior distribution converges in total variance (TV) distance to a normal distribution as $n \to \infty$ centered around the posterior mode $\hat\theta$ and with a covariance matrix given by the inverse of Fisher information at $\hat\theta$. 
Perhaps the most important implication of the BvM theorem is that Bayesian inference is asymptotically correct from a frequentist point of view, which cements the Bayesian approach as both a computationally tractable and theoretically principled approach.

Similarly to work on the Laplace approximation, recent research interest centers on the high-dimensional setting. Recently,~\cite{Kat25BvM} showed that the BvM theorem holds when $n \gg d^2$ for logistic regression models with random design, whereas the pre-existing general theory required $n \gg d^3$, e.g., for semi-parametric models in \cite{panovFiniteSampleBernstein2015} and Bayesian inverse problems in \cite{lu2017bernsteinvonmisestheoremhigh}.
There are also many works on the BvM theorem in non-parametric settings, which we do not survey here.

We turn to work in non-regular settings. 
For Bayesian inference on set identified models, \cite{chenMonteCarloConfidence2018} establishes Bayesian Wilks-type theorems in non-regular models under certain priors, and their resulting Bayesian credible sets provide valid frequentist coverage, which applies to models with parameter-dependent support. In regular set-identified models, they show that the limiting quasi-posterior distribution of the quasi-likelihood ratio statistic has a chi-squared distribution, and a gamma distribution when the model has parameter-dependent support.
Moreover, \cite{bochkinaBernsteinMisesTheoremNonregular2014} establishes a Bernstein--von Mises (BvM) theorem for non-regular models, showing that the posterior distribution converges to a product of (truncated) Gaussian distributions and gamma distributions, and~\cite{brunel2026bernsteinvonmisestheoremlogconcave} showed that non-regular BvM theorems hold under the sole assumption of log-concavity.

In our work, we focus on a regime of sample sizes in which these asymptotic results do not apply (e.g., $n \gg d$), and thus cannot be used for inferential purposes.
Instead, we establish structural properties which imply computational tractability of the posterior in this regime.

\paragraph{Non-asymptotic posterior sampling guarantees.}

The closest related work to ours in the pre-asymptotic regime is \cite{nickl2022}, which studied the sampling problem for high-dimensional \emph{regular} posterior distributions. They showed local strongly log-concavity of the posterior, which implies fast mixing for Langevin-type sampling algorithms when initialized at a warm start. Specifically, they studied a Bayesian inverse problem with i.i.d.\ pairs $(X_i, Y_i)$, where $X_i$ is the covariate and $Y_i$ is the response variable,
$$Y_i = \mathcal{G}(\theta)(X_i) + \epsilon_i\,,\qquad i=1,\dotsc, n\,.$$
The regression function $\mathcal{G}$ maps the parameter $\theta\in\Theta \subseteq \R^d$ to a continuous real-valued function over a bounded subset $\mathcal O$ of $\R^m$. The distributional assumptions are that $\epsilon_i \simiid \sN(0, \sigma^2)$ and $X_i \simiid P_X$, are independent, where $P_X$ is a probability measure over $\mathcal{O}$. Their analysis showed that when $n \gg d$, Langevin MCMC with a warm start produces samples whose law is close to the posterior in the 2-Wasserstein distance, with an iteration complexity polynomial in both $d$ and $n$.

We stress that the result of \cite{nickl2022} does not apply to non-regular models. 
Indeed,
their proof relies on a local strong log-concavity around the mode of the posterior distribution, which fails when the mode is attained on the boundary. To our knowledge, our work is the first to analyze sampling properties of non-regular models in a pre-asymptotic regime.

We also mention the recent work~\cite{chak2025complexitymarkovchainmonte}, which studies the complexity of MCMC algorithms on log-concave measures. In particular, they establish for linear, logistic, and Poisson regression that with $n \gg d$, MCMC attains the same complexity scaling in $n, d$ as first-order optimization algorithms, given a warm start inside the locally strongly log-concave region.

\paragraph{Functional inequalities for low-temperature Gibbs measures.}
Prior works have also studied functional inequalities, such as the Poincar\'e inequality, for low-temperature Gibbs measures.
When the energy function (the negative of the log-density) has an energy barrier, the precise exponential blow-up of the Poincar\'e constant as $n\to\infty$ is known as the Eyring--Kramers formula~\cite{Eyr1935, Kra1940}.
On the other hand, when the energy function has a benign landscape, the Poincar\'e constant remains bounded or even decays at the rate $1/n$ as $n\to\infty$~\cite{LiErd23RiemLangevin, CheSri25Lyapunov, chewi2024ballisticlimitlogsobolevconstant, GonHeShe25LocalPL}.
By truncating to the good set, our analysis falls within the benign landscape setting; however, to our knowledge, prior works did not consider the non-regular case in which the mode lies on the boundary of a constraint set.

\subsection{Notation}

We abuse notation by using the same symbol $\mu$ to denote both the probability measure of interest, as well as its Lebesgue density.
Let $d$ denote the ambient dimension and $n$ be a parameter, interpreted as the inverse temperature in general, and as the sample size in a Bayesian context. We write $\hat \theta$ for the mode of $p$, and $\theta^{\star}$ as the ground truth parameter in our Bayesian analysis.
The constraint set is $\Theta \deq [0, \infty)^d$ unless otherwise specified. 
 For any set of indices $S$, we denote by $\theta_{S}$ the vector of coordinates of $\theta$ indexed by $S$, and $\ell(\theta_S) \deq \ell(\theta_S, \hat \theta_{[d]\setminus S})$, where the coordinates at the omitted indices are fixed to be at the mode. 

 We denote the usual Euclidean norm as $\norm{\cdot}_2$, and $\norm{\cdot}_{\rm op}$ as the operator norm. We denote the largest and smallest eigenvalues by $\lambda_{\max}(\cdot)$ and $\lambda_{\min}(\cdot)$ respectively. For a matrix $A$, $\norm{A}_\infty$ is the maximum absolute coordinate of $A$, i.e., $\norm{A}_\infty \deq \max_{i,j} |A_{ij}|$.

 The notation $\vee$ denotes taking maximum of two numbers, and $\wedge$ denotes taking minimum of two numbers. We denote a function $f$ as $o(n)$ if $f(n)/n \to 0$ as $n \to \infty$, $O(n)$ if $f(n)/n$ is bounded by a constant, and we write $a \lesssim b$, $a \asymp b$, and $a\gtrsim b$ to denote $a = O(b)$, $a = \Theta(b)$, and $a = \Omega(b)$. The condition $n \gg A$ means that there exists a sufficiently large universal constant $C$ such that if $n \geq CA$, then the subsequent statement holds.
 We denote a ball of radius $r$ centered at $\theta$ in with respect to the $\ell_2$ and $\ell_\infty$ norms as $B_2(\theta, r)$ and $B_\infty(\theta,r)$ respectively, and we sometimes abbreviate $B(\theta,r) \deq B_2(\theta, r)$.

\subsection{Main results}

As discussed in~\Cref{sec:related_work}, unlike the work of~\cite{nickl2022}, the constrained low-temperature Gibbs distribution~\eqref{eq:constrained_gibbs_density} is not locally strong log-concave.
We propose to decompose the negative log-likelihood into three components: a regular part, one-dimensional non-regular parts, and perturbation terms that capture interactions between two parts. We analyze the regular and non-regular components separately.
In the regular part of the model, we can apply the Laplace expansion to second order, where the first-order term involves the gradient of the log-likelihood function and therefore vanishes at the mode. Therefore, the regular part corresponds to a locally strongly log-concave distribution.
 
On the other hand, we show that the non-regular part has an exponential-type density, which therefore satisfies a Poincar\'e inequality. By applying tensorization (\cref{thm:tensorization}) and the bounded perturbation principle (\cref{thm:holley-stroock}), we can then establish a Poincar\'e inequality for the overall distribution $\mu$, truncated to the good set. The first main result is the following (see the formal statement in~\cref{theorem:poincare}): 
\begin{inftheorem}\label{infthm:poincare}
    Consider the density $\mu$ defined in \eqref{eq:constrained_gibbs_density} and let $\hat\theta$ be its unique mode. 
    Define the index sets
    \begin{align*}
        S_0 \deq \{j\in [d]: \hat\theta_j > 0\}\,, \qquad S_1 \deq \{j\in [d]: \hat\theta_j = 0\}\,,
    \end{align*}
    corresponding to the regular and non-regular parts respectively, and write $d_0 \deq \abs{S_0}$, $d_1 \deq \abs{S_1}$.
    Define a ``good set'' $\hat\Theta \deq B_2(\hat\theta_{S_0}, r_0) \times B_\infty(\hat\theta_{S_1}, r_1)$, where $r_0 \asymp \sqrt{d_0/n}$ and $r_1 \asymp (\log d_1)/n$.
    Adopt the following assumptions:
    \begin{itemize}
        \item The Hessian of $\ell$ in the regular part, and the gradient of $\ell$ in the non-regular part, are strictly negative at $\hat\theta$. Namely, $\nabla^2_{S_0,S_0} \ell(\hat\theta) \preceq -C_{S_0} I \prec 0$ and $\partial_j \ell(\hat\theta) \le -C_{S_1} < 0$ for all $j\in S_1$.
        \item On $\hat\Theta$, the Hessian $\nabla^2 \ell$ is bounded in operator norm by $s_2$.
        \item The ``prior'' $\pi = \bigotimes_{j\in [d]} \pi_j$ is a log-concave product measure, such that each $\log \pi_j$ is $o(n)$-Lipschitz.
    \end{itemize}
    If $n \gg d_0$, then the conditional distribution $\mu|_{\hat\Theta}$ satisfies a Poincar\'e inequality with constant
    $$\CPI \lesssim \Bigl(\frac{1}{C_{S_0} n} \vee \frac{1}{C_{S_1}^2 n^2}\Bigr)\exp\Bigl\{O\bigl(s_2\,\bigl(\sqrt{\frac{d_0 d_1}{n} } + \frac{d_1}{n}\bigr)\bigr)\Bigr\}\,.$$ 
\end{inftheorem}

In order for this result to be meaningful, it is also important to check that $\mu$ places most of its mass on $\hat\Theta$.
Toward that end, we establish sufficient conditions for this to hold in \cref{sec:laplace_sufficient}.
In the application to Bayesian inference, these results capture the concentration of the posterior distribution around its mode which, to our knowledge, has not previously been quantitatively and systematically investigated in the non-regular case.
Although our conditions are somewhat complicated to state, in the case when $\mu$ is log-concave, our result holds as soon as
\begin{align*}
    n \gg d_0 d_1 \,.
\end{align*}
Hence, our results are strongest when $d_1$, the number of non-regular coordinates, is small.
We expect this to be true for many applications, e.g., the case of an open domain with smooth boundary corresponding to $d_1 = 1$.
It is worth noting that even when $d_1$ is allowed to grow slowly with $n$, e.g., $d_1 \lesssim \log n$, this is still well before the asymptotic regime $n \gg d^2$. 
We conjecture that the condition could be improved to $n \gg d_0 + d_1$ with more refined techniques, and we leave this as future work.

Next, we turn toward the application of our result to Bayesian inference in \cref{sec:random_laplace_density}.
Here, the measure $\mu$ is now random, interpreted as the posterior, and we aim to place suitable assumptions such that the conditions of \cref{infthm:poincare} hold with high probability.

\begin{inftheorem}\label{infthm:random}
    Consider the density $\mu$ in~\eqref{eq:constrained_gibbs_density}, where now $\ell \deq \ell_n$ is a random function and $\ell^\star \deq \E \ell_n$.
    Let $\hat\theta$, $\theta^\star$ denote the maximizers of $\ell_n$ and $\ell^\star$ over $[0,\infty)^d$ respectively, and
    define the index sets
    \begin{align*}
        S_0 \deq \{j\in [d]: \theta_j^\star > 0\}\,, \qquad S_1 \deq \{j\in [d]: \theta_j^\star = 0\}\,.
    \end{align*}
    Write $d_0 \deq \abs{S_0}$, $d_1 \deq \abs{S_1}$, define the ``good set'' $\hat\Theta \deq B_2(\hat\theta_{S_0}, c_0\sqrt{d_0/n}) \times B_\infty(\hat\theta_{S_1}, (c_1 \log d_1)/n)$, as well as the region $R \deq B_2(\theta^\star_{S_0}, r_0) \times B_\infty(\theta^\star_{S_1}, r_1)$.
    
    Suppose that the following assumptions hold with high probability:
    \begin{itemize}
        \item Over $R$, the empirical Hessian $\nabla^2\ell_n$ satisfies $-\nabla^2_{S_0,S_0} \ell_n \succeq c_{S_0} I$ and $\norm{\nabla^2 \ell_n}_{\rm op} \le s_2$.
        \item For all $j\in S_1$, $\partial_j \ell^\star(\theta^\star) \le -c_{S_1} < 0$.
        \item The random function $\ell_n$ is an i.i.d.\ sum $\ell_n = n^{-1} \sum_{i=1}^n \ell(\cdot; X_i)$, and the one-sample score is subexponential: $\sup_{\norm v_2 \le 1}{\norm{\langle v,\nabla_\theta \ell(\theta; X_1) - \E\nabla_\theta \ell(\theta; X_1)\rangle}_{\psi_1}} < \infty$ for all $\theta \in R$.
    \end{itemize}
    Then, if $n \gg d_0$, with high probability, the conditional distribution $\mu|_{\hat\Theta}$ 
    satisfies a Poincar\'e inequality with constant
    $$\CPI \lesssim \Bigl(\frac{1}{c_{S_0} n} \vee \frac{1}{c_{S_1}^2 n^2}\Bigr)\exp\Bigl\{O\bigl(s_2\,\bigl(\sqrt{\frac{d_0 d_1}{n} } + \frac{d_1}{n}\bigr)\bigr)\Bigr\}\,.$$
\end{inftheorem}

Our assumptions are chosen with an eye toward checking them in examples.
To illustrate our theory, we present three concrete applications to posterior sampling: logistic regression (\cref{section:logistic_regression}), the Poisson linear model (\cref{section:poisson_linear_model}), and Gaussian mixture models (\cref{section:gaussian_mixture_model}).

In the logistic regression and Poisson linear models, we show that when $n \gg d_0d_1$, the posterior distribution concentrates on the good set, and conditioned on that set the posterior satisfies a Poincar\'e inequality with constant $O(1/n)$. For the Gaussian mixture model, we show that the same conclusion holds when $n \gg d^2$.
We caution that these results are contingent on assumptions on the model and that the implicit constants depend on model-specific parameters, which could also scale with the dimension in certain situations.
Here, our goal is not to study these examples in-depth, but rather to demonstrate the applicability of the general theory.

However, in the well-studied logistic regression case, it is possible to directly compare our results to Laplace/BvM-type results from~\cite{katsevich2024laplaceasymptoticexpansionhigh, kasprzak2025goodlaplaceapproximationbayesian, Kat25BvM} and deduce that our sample size requirement $n \gg d_0 d_1 $ is indeed much less stringent than the requirement $n \gg d_0^2$ for the validity of the asymptotic approximation (and they consider the regular case with $d_1 = 0$; recall that we assume $d_1 \ge 1$, so in the case $d_1 = 0$ our requirement should be interpreted as $n \gg d_0$).
More broadly, although we do not claim tightness of our bounds for our other specific examples, we generally expect our approach to furnish sampling guarantees at a significantly smaller sample size than asymptotic approaches.

We conduct two simulation studies on three statistical models to assess how our theory performs in practice. We implemented the projected Langevin Monte Carlo algorithm with warm start in both pre-asymptotic and asymptotic regimes. The results support our theory and classical large-sample Bayesian behavior in non-regular models.

\paragraph{Organization of the paper.}
The paper is organized as follows. \cref{sec:preliminary} provides some technical background. In \cref{sec:laplace_density}, we provide a detailed analysis of the constrained low-temperature Gibbs distribution and establish the Poincar\'e inequality (\cref{infthm:poincare}), as well as our results on concentration on the good set. In \cref{sec:random_laplace_density}, we study Bayesian inference setting in which the Gibbs distribution is random, and we establish \cref{infthm:random}. In \cref{section:examples}, we illustrate our analysis by three examples of posterior sampling in logistic regression, Poisson linear models, and Gaussian mixture models. Two sets of simulation results on the three statistical models are then presented in \cref{sec:simulation} to complement our theory with practical experiments. Proofs deferred from the main text can be found in the Appendix.

\section{Preliminaries}\label{sec:preliminary}

We recall the definition of the Poincar\'e inequality.

\begin{definition}[Poincar\'e constant]\label{def:poincare}
Given a probability measure $\mu \in \mathcal{P}(\R^d)$, its Poincar\'e constant, denoted $\CPI(\mu)$, is the least constant $C$ such that for all smooth, compactly supported functions $f: \R^d \to \R$,
\begin{equation*}
\int f^2\, d\mu - \bigl(\int f\, d\mu\bigr)^2 \leq C \int \norm{\grad f}^2_2\, d\mu\,.
\end{equation*}
\end{definition}

We recall that a Poincar\'e inequality can be interpreted as a spectral gap for the Langevin dynamics, and in particular the following lemma holds~\cite{bakry2014analysis}.

\begin{lemma}
Consider a probability distribution $\mu$ with a smooth positive density, and the associated Langevin dynamics \begin{equation}
d X_t = \grad \log \mu(X_t)\,dt + \sqrt{2}\, d B_t\,, \qquad X_0 \sim \mu_0\,,
\end{equation}
where $\{B_t\}_{t\geq 0}$ is a standard Brownian motion.
Then, $\mu$ satisfies a Poincar\'e inequality with constant $\CPI$ if and only if
\begin{equation}
\chi^2(\mu_t, \mu) \leq \exp\bigl(-\frac{2t}{\CPI}\bigr)\,\chi^2(\mu_0, \mu) \qquad\text{for all}~t\ge 0\,,
\end{equation}
where $X_t \sim \mu_t$.
\end{lemma}

A sufficient condition for $\mu$ to satisfy a Poincar\'e inequality is for $\mu$ to be $\alpha$-strongly log-concave, i.e., $-\log \mu$ is $\alpha$-strongly convex, in which case $\CPI(\mu) \le 1/\alpha$.
However, in general, a Poincar\'e inequality is a much weaker condition than strong log-concavity.
We also record the following important properties of the Poincar\'e constant.

\begin{lemma}[Tensorization]\label{thm:tensorization}
Suppose that $\mu_1, \dots, \mu_N \in \mathcal{P}(\R^d)$ satisfy a Poincar\'e inequality with Poincar\'e constants $\CPI(\mu_1), \dots, \CPI(\mu_N)$ respectively. Then, for any $N\in \mathbb{N}$, the product measure $\mu \coloneqq \bigotimes_{i=1}^{N}\mu_i$ satisfies the Poincar\'e inequality with constant $\CPI(\mu) = \max_{i \in [N]} \CPI(\mu_i)$.
\end{lemma}

\begin{lemma}[Holley--Stroock perturbation \cite{Holley1987LogarithmicSI}]\label{thm:holley-stroock}
Suppose that a probability measure $\pi$ satisfies a Poincar\'e inequality with constant $\CPI(\pi)$. If we write $\frac{d\mu}{d\pi} = \exp f$, then $\mu$ also satisfies a Poincar\'e inequality with constant $\CPI(\mu) \leq \CPI(\pi) \exp(\osc f)$, where $\osc f \deq \sup f - \inf f$.
\end{lemma}

These facts, as well as further detailed information on Poincar\'e inequalities in the general context of Markov semigroup theory, can be found in the comprehensive monograph~\cite{bakry2014analysis}.

\section{Sampling from low-temperature Gibbs distributions}\label{sec:laplace_density}

In this section, we consider the problem of sampling from a high-dimensional low-temperature Gibbs distribution on a constrained parameter space. We consider the following density:
\begin{equation}
    \mu(\theta) \deq \frac{\pi(\theta)\exp{n \ell(\theta)}}{\int_{[0,\infty)^d}\pi(\theta)\exp{n \ell(\theta)} \,d\theta}\,, \qquad\theta \in [0, \infty)^d\,,\label{eq:gibbs_constrained}
\end{equation}
where the function $\ell$ has a unique global maximizer $\hat \theta$ in the truncated parameter space $\Theta \deq [0, \infty)^d$.
Let $S_0 \deq \{j \in [d] : \hat\theta_j > 0\}$ and $S_1 \deq \{j \in [d] : \hat\theta_j = 0\}$. We denote the dimensions of the regular and non-regular parts as $d_0$ and $d_1$ respectively, i.e., $\abs{S_0} = d_0$ and $\abs{S_1}=d_1$.
We assume throughout that $d_1 \ge 1$ to avoid unnecessary cases.

Throughout the entire paper, we impose the following assumption on $\pi$.

\begin{assumption}[Log-concavity of $\pi$]\label{integrand} 
    The measure $\pi$ is a log-concave product measure, i.e., $\pi = \bigotimes_{j\in [d]} \pi_j$ and for all $j \in [d]$, $\pi_j$ is log-concave.
\end{assumption}

\begin{remark}
The log-concavity assumption is not essential for our analysis and is only used for simple illustration. Our analysis extends to a more general $\pi$, e.g., it would suffice for $\log \pi$ to be a smooth function (admitting Lipschitz continuous gradients) in each coordinate.
\end{remark}

Later, in~\cref{sec:random_laplace_density}, we will take $\ell = \ell_n$ to be a random function, interpreted as the averaged (quasi-)log-likelihood function for applications to Bayesian inference.
The analysis there will build upon the deterministic analysis developed in this section.
As discussed in \cref{sec:intro}, we are interested in the pre-asymptotic regime, meaning that our analysis should apply to the regime $n\gg d$ (under suitable regularity conditions on the model).
In contrast,
validity of the Laplace approximation typically requires $n \gg d^2$ by \citep{katsevich2024laplaceasymptoticexpansionhigh} and \citep{kasprzak2025goodlaplaceapproximationbayesian}.

\subsection{Assumptions}\label{sec:laplace_assumptions}

In this section, we state the assumptions needed to establish our main result (\cref{theorem:poincare}) for constrained Gibbs measures. To check concentration on the good set (\cref{assumption:posterior contraction}), we provide sufficient conditions which are easier to verify in~\Cref{sec:laplace_sufficient}.

\begin{assumption}[Unique maximizer]\label{ass:mode}
    The maximizer $\hat\theta \deq \argmax_{\theta \in \Theta}\ell(\theta)$ exists and is unique.
\end{assumption}

\begin{assumption}[Regular part in the strict interior] \label{assumption:dimension}
The regular part of the global maximizer lies strictly in the interior of the parameter space, i.e., for some $C_0 > 0$, $$\inf_{i \in S_0}\abs{\hat \theta_{i}} \geq C_0 > 0\,. $$
\end{assumption}

Before proceeding, we introduce the notion of a good set, on which $\ell$ admits derivative bounds and which contains most of the mass of $\mu$, as formalized below.
  
\begin{definition}[Good set]\label{def:good_set}
  The good set $\hat\Theta_{\delta} \subseteq \Theta$ is defined as follows: $$\hat\Theta_\delta \coloneqq \Bigl\{\theta \in \Theta : \norm{\hat \theta_{S_0} - \theta_{S_0}}_{2} \leq \delta_0  \sqrt{\frac{d_0}{n}}\,,\; \norm{\hat \theta_{S_1} - \theta_{S_1}}_{\infty} \leq  \frac{\delta_1}{n}\Bigr\}\,.$$
  We assume that $0 < \delta_0 \sqrt{d_0/n} < C_0$ so that the condition $\norm{\hat\theta_{S_0} - \theta_{S_0}}_2 \le \delta \sqrt{d_0/n}$ implies $\theta_{S_0} \in [0,\infty)^{d_0}$.
  In an abuse of notation, we also write
  \begin{align*}
      \hat \Theta_{\delta_0} \deq \Bigl\{\theta_{S_0} \in \R^{d_0}: \norm{\hat \theta_{S_0} - \theta_{S_0}}_{2} \leq \delta_0\sqrt{\frac{d_0}{n}}\Bigr\}\,, \qquad \hat\Theta_{\delta_1} \deq \Bigl\{\theta_{S_1} \in \R^{d_1}: \norm{\hat \theta_{S_1} - \theta_{S_1}}_{\infty} \leq \frac{\delta_1}{n}\Bigr\}\,.
  \end{align*}
\end{definition}

Next, we impose assumptions on the local behavior of $\ell$ around the maximizer $\hat\theta$.
The second-order optimality condition already implies that the regular part of the Hessian of $\ell$ is non-positive at $\hat \theta$. We further assume that $\ell$ is in fact locally \emph{strongly} concave around $\hat \theta$ in the regular directions.

\begin{assumption}[Local strong concavity in the interior]\label{ass:local_cvxty}
    There exists $C_{S_0} > 0$ such that for all $\theta \in \hat\Theta_\delta$, we have $-\nabla_{S_0}^2 \ell(\theta) \succeq C_{S_0} I \succ 0$.
\end{assumption}

For the non-regular part, the first-order optimality condition implies that the non-regular part of the gradient of $\ell$ at $\hat\theta$ is coordinate-wise non-positive.
We further assume that non-regular part of the gradient is actually \emph{strictly} negative at the boundary.
This assumption rules out cases in which $\partial_j \ell(\hat\theta) = 0$ for some $j\in S_1$ and is made for the sake of simplicity.

\begin{assumption}[Locally negative gradient at the boundary] \label{assumption:negative_gradient_l}
  There exists $C_{S_1} > 0$ such that $\partial_{j} \ell(\hat \theta) \leq -C_{S_1} < 0$ for all $j \in S_1$.
\end{assumption}

\begin{remark}[One-sided derivative at the boundary]
    Here and below, derivatives (and higher-order derivatives) at the boundary are interpreted as one-sided derivatives, i.e., if $\theta_j = 0$, then
    \begin{align*}
        \partial_j \ell(\theta) \deq \lim_{h\searrow 0} \frac{\ell(\theta+he_j) - \ell(\theta)}{h}\,.
    \end{align*}
\end{remark}

Next, we require derivative bounds over $\hat\Theta_\delta$.

\begin{assumption}[Boundedness of derivatives around the mode] \label{assumption:bounded_mixed_derivatives}
  %
  We assume $\ell$ is $C^2$ on $\hat\Theta_\delta$ and admits the following Hessian bounds on the good set: $$\sup_{\theta \in \hat\Theta_{\delta}} \norm{\grad^2 \ell(\theta)}_{\rm op} \leq s_2 $$  for some $s_2 < \infty$.
\end{assumption}

Finally, we require that the good set contains most of the mass of $\mu$.

\begin{assumption}[Concentration] \label{assumption:posterior contraction}
    Let $\mugood$ be the measure $\mu$ conditional on the good set $\hat \Theta_{\delta}$, i.e., $\mugood \deq \mu\one_{\hat\Theta_\delta}/\mu(\hat\Theta_\delta)$. We assume that by judiciously choosing the radii $\delta_0$ and $\delta_1$ of the good set, $\mu$ concentrates on the good set:
    \begin{equation*}
  \norm{\mu - \mugood}_{\rm TV} \leq \epsilon\,,
\end{equation*}
where $\epsilon \in (0, 1)$ is the desired error tolerance.

\end{assumption}

\subsection{Likelihood decomposition}

Our main structural result is the following decomposition of $\ell$ on the good set.

\begin{theorem}[Likelihood decomposition]\label{thm:likelihood_decomposition}
    Adopt~\cref{integrand} through~\cref{assumption:bounded_mixed_derivatives}.
    We have the decomposition 
$$\ell(\theta) = B(\theta) + f(\theta_{S_0}) + \sum_{j\in S_1} g_j(\theta_j)\,, \qquad\theta \in \hat\Theta_\delta\,,$$
with the following properties:
\begin{itemize}
    \item $f$ is $C_{S_0}$-strongly concave.
    \item Each $g_j$ is linear with slope $\partial_j \ell(\hat\theta) \le -C_{S_1}$.
    \item The coupling term $B$ satisfies the bound
    \begin{align*}
        \osc B
        \deq \sup B - \inf B
        \le 2s_2 \,\Bigl(\frac{\delta_0 \delta_1 \,(d_0 d_1)^{1/2}}{n^{3/2}} + \frac{\delta_1^2 d_1}{n^2}\Bigr)\,.
    \end{align*}
\end{itemize}
\end{theorem}

Using \cref{thm:likelihood_decomposition}, tensorization of the Poincar\'e inequality (\cref{thm:tensorization}), and the Holley--Stroock perturbation theorem (\cref{thm:holley-stroock}), we can then obtain a dimension free Poincar\'e constant for $\mugood$ (see~\cref{theorem:poincare} below).
%

%

\begin{proof}
    We show in turn the likelihood decomposition, log-concavity in the regular part, and control of the bounded perturbation term.
    \paragraph{Likelihood decomposition.}
    To show this separable structure, we first separate the regular and non-regular parts by Taylor expanding $\ell$ around $\hat \theta_{S_0}$ while fixing $\hat\theta_{S_1}$ and around $\hat \theta_{S_1}$ while fixing $\hat\theta_{S_0}$. Then, we apply the same idea within the non-regular part. 
    We use the notation $\ell(\theta_{S_0}) \deq \ell(\theta_{S_0}, \hat\theta_{S_1})$ and $\ell(\theta_{S_1}) \deq \ell(\hat\theta_{S_0}, \theta_{S_1})$.
\begin{align*}
  \ell(\theta) &= \ell(\theta_{S_0},\hat \theta_{S_1}) + \ell(\hat \theta_{S_0}, \theta_{S_1}) + \ell(\theta) - \ell(\theta_{S_0},\hat \theta_{S_1}) - \ell(\hat \theta_{S_0}, \theta_{S_1}) \\
    &= \ell(\theta_{S_0}) + \ell(\theta_{S_1}) + B_{S_0, S_1}(\theta)\,,
\end{align*}
where $B_{S_0, S_1}(\theta) \deq \ell(\theta) - \ell(\theta_{S_0}) - \ell(\theta_{S_1})$.
Further, we have
\begin{align*}
    B_{S_0, S_1}(\theta)
    &= \ell(\theta) - \ell(\theta_{S_0}) - \ell(\theta_{S_1}) + \ell(\hat \theta) - \ell(\hat \theta) \\ 
    &= \int_{0}^{1} (\theta_{S_1} -\hat\theta_{S_1})^\T \grad_{S_1} \ell(\theta_{S_0}, t\theta_{S_1} + (1-t)\hat \theta_{S_1})\, dt \\ 
    &\qquad{} - \int_{0}^{1} (\theta_{S_1} -\hat \theta_{S_1})^\T\grad_{S_1} \ell(\hat \theta_{S_0}, t\theta_{S_1} + (1-t)\hat \theta_{S_1})\, dt - \ell(\hat \theta) \\
    &= \int_{0}^{1} \int_{0}^{1} (\theta_{S_1} -\hat \theta_{S_1})^\T\grad_{S_1,S_0}^2 \ell(s\theta_{S_0} + (1-s)\hat \theta_{S_0}, t\theta_{S_1})(\theta_{S_0} -\hat \theta_{S_0})\, ds\, dt - \ell(\hat \theta) \,,
\end{align*}
where we used $\hat \theta_{S_1}=0$ to lighten notation.


Note that by \cref{assumption:bounded_mixed_derivatives}, $\osc B_{S_0,S_1} \le 2s_2\delta_0 \delta_1 \sqrt{\frac{d_0}{n}}\,\frac{\sqrt{d_1}}{n}$.

\paragraph{Exponential-type distribution in non-regular part.}
Now, let
\begin{align*}
    B_{S_1}(\theta) \deq \ell(\theta_{S_1}) - \ell(\hat\theta) - \sum_{j \in S_1} \partial_j \ell(\hat\theta)\, \theta_j\,.
\end{align*}
Then, by Taylor expansion to second order, we have
\begin{align*}
    B_{S_1}(\theta)
    &= \int_{0}^{1} (1-t)\,\theta_{S_1}^\T \grad^2_{S_1} \ell(\hat \theta_{S_0}, t\theta_{S_1})\theta_{S_1}\, dt \\ 
    &\leq \int_{0}^{1} (1-t)\,s_2\, \norm{\theta_{S_1}}_2^2\, dt
    \leq \frac{s_2 \delta_1^2 d_1}{2n^2}\,.
\end{align*}

Therefore, we have the desired decomposition, with $f(\theta_{S_0}) \deq \ell(\theta_{S_0})$, $g_j(\theta_j) \deq \partial_j \ell(\hat\theta)\,\theta_j$, and $B(\theta) \deq B_{S_0,S_1}(\theta) + B_{S_1}(\theta) + \ell(\hat\theta)$.
By~\cref{ass:local_cvxty}, we know that $f$ is $C_{S_0}$-strongly concave on $\hat\Theta_\delta$.
Finally, the perturbation term satisfies
\begin{align*}
    \osc B
    &\le 2s_2 \delta_0 \delta_1 \sqrt{\frac{d_0}{n}}\,\frac{\sqrt{d_1}}{n} + \frac{s_2 \delta_1^2 d_1}{n^2}\,. \qedhere
\end{align*}
\end{proof}

\subsection{Main result}\label{section:main_results}

Under the assumptions in \cref{sec:laplace_assumptions}, we are able to show good sampling properties of \eqref{eq:gibbs_constrained} on the good set $\hat\Theta_\delta$ via a dimension-free Poincar\'e inequality. We present the main result for constrained low-temperature Gibbs distributions in the following theorem.
\begin{theorem}[Poincar\'e inequality for the constrained Gibbs measure]\label{theorem:poincare}
    Adopt~\cref{integrand} through~\cref{assumption:bounded_mixed_derivatives}, and assume that $n \ge \delta_0^2 d_0/C_0^2$.
    Then, the low-temperature Gibbs distribution in \eqref{eq:gibbs_constrained} conditioned on the good set $\hat\Theta_\delta$ satisfies the Poincar\'e inequality with
    \begin{align*}
        \CPI(\mugood) \le \Bigl(\frac{1}{nC_{S_0}} \vee \frac{4 \max_{j\in S_1} \exp(\osc_{[0,\delta_1/n]} \log \pi_j)}{n^2 C_{S_1}^2}\Bigr) \exp\Bigl(2s_2 \,\Bigl(\frac{\delta_0 \delta_1 \,(d_0 d_1)^{1/2}}{n^{1/2}} + \frac{\delta_1^2 d_1}{n}\Bigr)\Bigr)\,.
    \end{align*}
\end{theorem}

As noted in the introduction, in order to avoid exponential dependence, the result requires $n\gg d_0 d_1$.
We have in mind the case in which the non-regular part has much smaller dimension than the regular part, i.e., $d_1 \ll d_0$.
In particular, if $d_1 = O(1)$, the requirement simply becomes $n \gg d_0$, which covers the regular case ($d_1 = 0$) treated in~\cite{nickl2022}, and is expected to cover constraint sets consisting of an open domain with smooth boundary ($d_1 = 1$).

If $\log\pi_j$ is $O(1)$-Lipschitz, then $\osc_{[0,\delta_1/n]} \log \pi_j = O(1/n)$. 
Thus, under these conditions on $n$, the Poincar\'e constant of $\mugood$ scales as $O(1/n)$.
This is in line with recent works on functional inequalities for low-temperature Gibbs measures, as discussed in the related works (\cref{sec:related_work}), except that our analysis is local (restricted to the good set $\hat\Theta_\delta$) and covers the non-regular case.

\begin{proof}
First, define the distribution $\tilde\mu$ on $\hat\Theta_\delta$ with density
\begin{align*}
    \tilde\mu(\theta) \propto \pi(\theta) \exp\Bigl(nf(\theta_{S_0}) + n\sum_{j\in S_1} g_j(\theta_j)\Bigr)\,,
\end{align*}
where the functions $f$ and $g_j$, $j\in S_1$ are from~\cref{thm:likelihood_decomposition}.
By the Holley--Stroock perturbation theorem (\cref{thm:holley-stroock}) and the bound on the oscillation of $B$,
\begin{align*}
    \CPI(\mugood) \le \CPI(\tilde\mu) \exp\Bigl(2s_2 \,\Bigl(\frac{\delta_0 \delta_1 \,(d_0 d_1)^{1/2}}{n^{1/2}} + \frac{\delta_1^2 d_1}{n}\Bigr)\Bigr)\,.
\end{align*}
Next, since $\tilde\mu$ is a product measure (using the structure of $\hat\Theta_\delta$), tensorization of the Poincar\'e inequality (\cref{thm:tensorization}) yields
\begin{align*}
    \CPI(\tilde\mu) \le \CPI(\tilde\mu_{S_0}) \vee \max_{j\in S_1} \CPI(\tilde\mu_j)\,.
\end{align*}
Since $\tilde\mu_{S_0}$ is $nC_{S_0}$-strongly log-concave, the well-known Bakry{--}\'Emery criterion implies that
\begin{align*}
    \CPI(\tilde\mu_{S_0}) \le \frac{1}{nC_{S_0}}\,.
\end{align*}
For a version of this theorem that applies to constrained domains, see~\cite[Theorem 3.3.2]{Wang14Diffusion}.

Finally, for the non-regular part, we again apply the Holley--Stroock perturbation principle:
\begin{align*}
    \CPI(\tilde\mu_j) \le \exp\bigl(\osc_{[0, \delta_1/n]} \log \pi_j\bigr)\,\CPI(\breve\mu_j)\,, \qquad \breve\mu_j \propto \exp(ng_j) \one_{[0,\delta_1/n]}\,.
\end{align*}
The distribution $\breve\mu_j$ is an exponential distribution with parameter $-n\partial_j\ell(\hat\theta) \ge nC_{S_1}$, restricted to the interval $[0,\delta_1/n]$.
Since restriction to an interval does not increase the Poincar\'e constant---see~\cite[Lemma 4]{RouBarIoo17PI}---it follows from a standard computation that $\CPI(\breve\mu_j) \le 4/(n^2 C_{S_1}^2)$.
See, e.g.,~\cite[Propositions 4.4.1 and 4.4.4]{bakry2014analysis}.

The final result follows from putting together these bounds.
\end{proof}

\begin{remark}\label{remark:generalize nonregular poincare}
From the proof, it can be seen that our analysis can be extended to the case $\partial_j \ell(\hat\theta) = 0$ for some $j\in S_1$, provided that the measure $\tilde\mu_j(\theta_j) \propto \exp\bigl(n\ell(\hat\theta_{-j},\theta_j)\bigr) \one_{0 \le \theta_j \le \delta_1/n}$ satisfies a Poincar\'e inequality.
The same proof carries through with minor modifications.
\end{remark}
It is standard that a Poincar\'e inequality implies convergence of the Langevin diffusion.

\begin{corollary}[Convergence of the Langevin diffusion]\label{corollary:langevin_convergence}
    Adopt the assumptions of~\cref{sec:laplace_assumptions}.
    Let $\mu_t$ denote the law of the Langevin diffusion
    \begin{align*}
        dX_t = \nabla \log \mu(X_t)\,dt + \sqrt 2\,d B_t\,,
    \end{align*}
    initialized at $\mu_0$ supported on $\hat\Theta_\delta$ and reflected at the boundary of $\hat\Theta_\delta$.
    Then,    

    %
  \begin{equation*}
    \norm{\mu - \mu_t}_{\rm TV} \leq \epsilon + \sqrt{\frac{1}{2}\,\chi^2(\mu_0,  \mugood)}\exp\Bigl(-\frac{t}{\CPI(\mugood)}\Bigr)
  \end{equation*}
\end{corollary}
\begin{proof}
    By the triangle inequality, Pinsker's inequality, and the fact that the KL divergence is bounded by the $\chi^2$ divergence,
    \begin{align*}
        \norm{\mu - \mu_t}_{\rm TV} &\leq \norm{\mu - \mugood}_{\rm TV} + \norm{\mugood - \mu_t}_{\rm TV} \\ 
        &\leq \epsilon + \sqrt{\frac{1}{2}\,\chi^2(\mu_0, \mugood)\exp\Bigl(-\frac{2t}{\CPI(\mugood)}\Bigr)}\,,
      \end{align*}
      where the last line uses~\cref{assumption:posterior contraction} and the classical equivalence between a Poincar\'e inequality and exponential ergodicity of the Langevin diffusion; see~\cite[Theorem 4.2.5]{bakry2014analysis}.
\end{proof}

Although \cref{corollary:langevin_convergence} only pertains to the idealized diffusion in continuous time, discrete-time algorithms for sampling under a Poincar\'e inequality are also well-studied.
For example, a simple approach could be to use projected Langevin Monte Carlo~\citep{BubeckEldanLehec2018}:
\begin{equation}
    X_{k+1} = \mathcal{P}_{\hat\Theta_{\delta}}\big(X_k + h \nabla \log \mu(X_k) + \sN(0, 2hI_d)\big)\,, \label{eq:projected_langevin}
\end{equation}
where $\mathcal{P}_{\hat\Theta_{\delta}}$ is the Euclidean projection onto the set $\hat\Theta_{\delta}$ and $h$ is the step size. In the non-negative orthant case, the projection simply takes the maximum with $0$ coordinate-wise. This is implemented in our simulation experiments in \cref{sec:simulation}.

For completeness, we also show how sampling guarantees immediately translate into guarantees for estimating bounded test functions against the posterior, e.g., indicator functions.

\begin{corollary}[Estimating expectations]
     Let $f:\R^d \to [0,1]$ be bounded and $\theta_1,\dots,\theta_N \stackrel{\mathrm{i.i.d.}}{\sim}\hat\mu$, where $\hat\mu$ is the output distribution from a sampling algorithm, e.g., the law of Langevin diffusion in \cref{corollary:langevin_convergence}. Then,
    \begin{align*}
        &\E_{\hat\mu}\Bigl|\frac{1}{N}\sum_{i=1}^N f(\theta_i) - \E_{\mu} f(\theta)\Bigr| \leq \sqrt{\frac{\Var_{\mu}(f(\theta)) + \norm{\hat\mu-\mu}_{\rm TV}}{N} } + \norm{\hat\mu-\mu}_{\rm TV} \,. 
    \end{align*}
\end{corollary}
\begin{proof}
    Since $f$ is bounded, $\abs{\E_{\hat \mu} f(\theta) - \E_{\mu} f(\theta)}\le \norm{\hat\mu-\mu}_{\rm TV}$, and
\begin{align*} 
  \E_{\hat\mu}\Bigl|\frac{1}{N}\sum_{i=1}^N f(\theta_i) - \E_{\hat\mu} f(\theta)\Bigr|
  &\le \sqrt{\frac{\Var_{\hat\mu}(f(\theta))}{N}}\,.
\end{align*}
Then, we can use the following inequality:
\begin{align*}
  \Var_{\hat\mu} (f(\theta))
   = \inf_{m \in [0,1]} \E_{\hat\mu} (f(\theta) - m)^2 &\leq \inf_{m \in [0,1]} \E_{\mu} (f(\theta) - m)^2 + \norm{\hat\mu-\mu}_{\rm TV} \\ 
  &= \Var_{\mu}(f(\theta)) + \norm{\hat\mu-\mu}_{\rm TV}\,. \qedhere
\end{align*}
\end{proof}

\subsection{Sufficient conditions for concentration on the good set}\label{sec:laplace_sufficient}

To show that the good set $\hat\Theta_{\delta}$ contains most of the mass of the distribution, we make two different sets of assumptions to control the tail of the distribution. 
First, we show that concentration holds for log-concave distributions. Next, for non-log-concave distributions, we assume that the constraint set is compact and that the mode is well-separated.
To our knowledge, quantitative concentration results under various structural conditions in the non-regular case have not been extensively investigated in the literature, and to do so would be out of scope for the present work.
Instead, our goal is simply to provide some sufficient conditions which can be checked for our examples in \cref{section:examples}.
The proofs are given in~\cref{app:concentration}.

\paragraph{Log-concave case.}

We impose an assumption on the prior to ensure that it is relatively ``flat'' and therefore does not significantly affect the distribution $\mu$.

\begin{assumption}[Gradient bound on the prior]\label{assumption:Gradient bound on the prior} 
We assume that each $\log \pi_j$ is $L_\pi$-Lipschitz. 
\end{assumption}

We also assume the following condition which we call ``consistency'', because in the context of~\cref{sec:random_laplace_density} it is implied by the statement that a random draw from the posterior is a consistent estimator.
This is a global assumption and is typically checked separately.

\begin{assumption}[Consistency]\label{posterior consistency}
     For some $r_0' > 0$, $\mu(\|\theta_{S_0}-\hat\theta_{S_0}\|_2\le r_0') \ge 2/3$. 
\end{assumption}

We can now state our main concentration bound under log-concavity.
Note that below, we impose our assumptions on a larger region with radii $r_0$, $r_1$, where we treat $r_0$ and $r_1$ as being of constant order (rather than only imposing the assumptions on the good set).
For simplicity, we state a simplified result which suppresses the dependence on the constants $C_{S_0}$, $C_{S_1}$, $s_2$, etc., but the full dependencies are given in the detailed proof.

\begin{theorem}[Concentration under log-concavity]\label{thm:concentration_log_concave}
Adopt~\cref{integrand} through~\cref{assumption:bounded_mixed_derivatives}, \cref{assumption:Gradient bound on the prior}, and \cref{posterior consistency}.
Furthermore, assume the following:
\begin{itemize}
    
    \item $\mu$ is log-concave.
    \item For some constants $r_0 > 0$ and $\frac{C_{S_0}C_{S_1}}{2s_2^2} \ge r_1 > 0$, $\grad_{S_0}^2 \ell(\theta) \preceq -C_{S_0} I$ for all $\theta \in B(\hat\theta, r_0, 0)$, and $\partial_j\ell(\theta) \leq -C_{S_1}$ for all $j \in S_1$ and $\theta \in B(\hat\theta, 0, r_1)$.
    \item  $\sup_{\theta \in B(\hat\theta, r_0, r_1)} \norm{\grad^2 \ell(\theta)}_{\rm op} \leq s_2$.
\end{itemize}

There are constants $\bar c_0$, $\bar c_1$, $\bar c_2$, $\bar c_3$, $\bar c_4$, depending only on $C_{S_0}$, $C_{S_1}$, $r_0$, $r_1$, $s_2$, such that if the following conditions hold:
\begin{align*}
    \delta_0 = \bar c_0 \log\frac{1}{\varepsilon}\,, \quad \delta_1 = \bar c_1 \log\frac{d_1}{\varepsilon}\,, \quad L_\pi \le \bar c_2 \,\bigl(\sqrt{\frac{n}{d_0 d_1}} \wedge \sqrt{\frac{d_0}{d_1}}\bigr)\,, \quad r_0' \le \frac{\bar c_3}{\log(1/\varepsilon)}\,, \quad d_1 \le d_0\,,
\end{align*}
and
\begin{align}\label{eq:log_concave_concentration_n}
    n \ge \bar c_4 d_0 d_1 \log^2\bigl(\frac{d}{\varepsilon}\bigr)\,,
\end{align}
then~\cref{assumption:posterior contraction} holds: $\mu(\Theta\setminus\hat\Theta_\delta)\le \varepsilon$.
\end{theorem}

\begin{remark}
    In the above theorem, we take $r_1$ to be a constant.
    However, for our results in \cref{sec:random_laplace_density}, we will also require $r_1 \le C_{S_1}/(3s_2 \sqrt{d_1})$, which depends on $d_1$.
    If we take $r_1 \asymp 1/\sqrt{d_1}$, then the detailed bound in the proof of~\cref{thm:concentration_log_concave} shows that the condition~\eqref{eq:log_concave_concentration_n} still suffices, where $\bar c_4$ now only depends on $C_{S_0}$, $C_{S_1}$, $r_0$, and $s_2$.
    Similarly, when these other parameters depend on the dimension, one can refer to the full bound given in~\cref{appendix:proof_lemma: log concave}.
\end{remark}

\paragraph{Well-separated mode.}\label{contraction:compact_region}
We next look at a compact parameter space and assume that the log-likelihood admits a well-separated maximizer, which is significantly weaker than assuming log-concavity. In this case, we can still show concentration of the posterior $\mu$ on the good set $\hat\Theta$.

\begin{assumption}[Compact parameter space]\label{assumption_contraction:compact_parameter_space}
    The parameter space is defined as $\Theta \deq \mathcal C \cap [0,\infty)^d$, where $\mathcal{C}\subset\mathbb{R}^d$ is a compact set with non-empty interior.
Assume that $B(\hat\theta,r_0,r_1)\subseteq \Theta$.

Let $R_\Theta$ denote the radius of the parameter space $\Theta$, i.e., $R_\Theta \deq \max_{\theta \in \Theta}\norm{\theta}_2$.
\end{assumption}

\begin{assumption}[Well-separated mode]\label{assumption:exponential_decay}
    For some $\zeta > 0$ and any $\theta \in \Theta \setminus B(\hat\theta, r_0, r_1)$, we have $$\ell(\theta) - \ell(\hat\theta) \leq - \zeta\,.$$
\end{assumption}

Our analysis extends the arguments in \cref{thm:concentration_log_concave} to more complex models, in particular allowing us to treat Gaussian mixture models in \cref{section:gaussian_mixture_model}.

\begin{theorem}[Concentration under the well-separated mode condition]\label{proof_lemma: exponential_decay}
Adopt~\cref{integrand} through~\cref{assumption:bounded_mixed_derivatives}, \cref{assumption:Gradient bound on the prior}, and \cref{assumption_contraction:compact_parameter_space}.
Furthermore, assume:
\begin{itemize}
    \item For some constant $r_0 > 0$ and $0 < r_1 \leq \frac{C_{S_1}C_{S_0}}{2s_2^2}$, $\grad_{S_0}^2 \ell(\theta) \preceq -C_{S_0} I$ for all $\theta \in B(\hat\theta, r_0, 0)$, and $\partial_j\ell(\theta) < -C_{S_1}$ for all $j \in S_1$ and $\theta \in B(\hat\theta, 0, r_1)$. Moreover, $\ell$ is concave in $B(\hat\theta, r_0, r_1)$.
    \item \cref{assumption:exponential_decay} is satisfied with the same $r_0$ and $r_1$ above.
    \item  $\sup_{B(\hat\theta, r_0, r_1)} \norm{\grad_{S_0,S_1}^2 \ell(\theta)}_{\rm op} \leq s_2$.
\end{itemize}
Then, there are constants $\bar c_0$, $\bar c_1$, $\bar c_2$, depending only on $C_{S_0}$, $C_{S_1}$, $r_0$, $r_1$, $s_2$, $R_\Theta$, $\zeta$, such that if the following conditions hold:
\begin{align*}
    \delta_0 = \bar c_0 \log\frac{1}{\varepsilon}\,, \quad \delta_1 = \bar c_1 \log\frac{d_1}{\varepsilon}\,, \quad d_1 \le d_0\,, \quad n \ge \bar c_2\,d \log^2\bigl(\frac{d}{\varepsilon}\bigr)\,,
\end{align*}
and if $L_\pi$ is sufficiently small,
then~\cref{assumption:posterior contraction} holds: $\mu(\Theta\setminus\hat\Theta_\delta)\le \varepsilon$.

\end{theorem}

\begin{remark}
    \cite{katsevich2025asymptoticanalysisrareevents} also studies posterior concentration in a similar setting (e.g., see Proposition 3.11 therein), but focuses on the asymptotic regime $n \gg d^2$ and $d_1 = 1$. Our result is stronger in the following aspects: (i) we prove concentration of the good set in which the non-regular part is a box with width shrinking as $\log(d_1)/n$, while the analysis in \cite{katsevich2025asymptoticanalysisrareevents} requires $d/n$ width in the non-regular part; (ii) we allow for $d_1 > 1$, explicitly handling corners where multiple parameters lie on the boundary, whereas \cite{katsevich2025asymptoticanalysisrareevents} assumes a smooth boundary of the parameter space, which transforms to a single non-regular parameter after straightening the boundary.
\end{remark}

\section{Sampling from random low-temperature Gibbs distributions}\label{sec:random_laplace_density}

Let \( \mathcal{P} \) denote a family of distributions with density \( p_\theta \), parameterized by \( \theta \in \Theta = [0, \infty)^d \). Suppose that there exists a ground truth parameter \( \theta^{\star} \in \Theta \) that uniquely maximizes the population log-likelihood \( \ell^{\star}(\theta)  \deq  \mathbb{E} \ell_n(\theta) \), where  
  \[
  \ell_n(\theta) = \ell_n(\theta; X^{(n)}) \deq \frac{1}{n} \sum_{i=1}^n \ell(\theta; X_i) = \frac{1}{n} \sum_{i=1}^n \log p_\theta(X_i)
  \]  
  is the empirical log-likelihood of $n$ i.i.d.\ samples. 
  We consider the problem of sampling from a random, constrained low-temperature Gibbs distribution $\mu$ with the following density:
\begin{equation}
    \mu(\theta) = \frac{\pi(\theta)\exp{n \ell_n(\theta; X^{(n)})}}{\int_{[0,\infty)^d}\pi(\theta)\exp{n \ell_n(\theta;X^{(n)})}\, d\theta}\,, \qquad \theta \in \Theta\,. \label{eq:random_laplace_density}
\end{equation}
  
Define the index sets $$S_0 \deq \{j \in [d] : \theta^\star_j > 0\} \qquad \text{and} \qquad S_1 \deq \{j\in [d]: \theta_j^\star = 0\}\,.$$  We denote the dimension of regular part as $d_0$ and the non-regular part as $d_1$, i.e., $\abs{S_0}=d_0$ and $\abs{S_1}=d_1$. The prior $\pi$ is assumed to satisfy~\cref{integrand}.

The key difference from the previous section is that we now consider the averaged empirical likelihood or quasi-likelihood $\ell_n$, which introduces randomness from the data, in contrast to the deterministic function $\ell$ in \eqref{eq:gibbs_constrained}. Consequently, we replace the assumptions made in the previous section with ones stated in terms of the population function $\ell^\star$. We will establish how to transfer these assumptions from $\ell^\star$ to $\ell_n$ using finite-sample concentration, and then extend our earlier analysis to derive sampling guarantees for \eqref{eq:random_laplace_density}. At the end of this section, we will establish the posterior or quasi-posterior contraction rate to the good set given in \cref{def:good_set} around the MLE $\hat\theta$ under two different sets of assumptions.

\subsection{Assumptions} \label{sec:assumptions}

Here, we state the analogue of the assumptions from~\cref{sec:laplace_assumptions}.

\begin{assumption}[Unique maximizers]\label{assumption:existence of mle}
    The maximizer $\theta^\star \deq \argmax_{\theta\in\Theta} \ell^\star(\theta)$ exists and is unique.
    The MLE $\hat\theta \deq \argmax_{\theta\in\Theta}\ell_n(\theta)$ exists and is unique with probability at least $1-\eta$.
\end{assumption}

\begin{assumption}[Regular and non-regular parts] \label{assumption:dimension_random}
  We assume that $\theta_{S_0}^{\star}$ lies strictly in the interior of the parameter space, i.e.,$$\inf_{i \in S_0}\abs{\theta_{i}^\star} \geq c_0 > 0\,.$$
\end{assumption}

\begin{assumption}[Consistency]\label{ass:consistency}
    Given $r_0, r_1 > 1$, we define the set
    \begin{align*}
        B(\theta^\star,r_0,r_1) \deq (B_2(\theta_{S_0}^\star, r_0) \times B_\infty(\theta_{S_1}^\star, r_1))\cap [0,\infty)^d\,.
    \end{align*}
   To make sure that $B_2(\theta^\star,r_0)$ is contained in $\Theta_{S_0}$, we assume
   $r_0 \ge c_0$.
    We assume that with probability at least $1-\eta$, $\hat\theta \in B(\theta^\star, r_0, r_1)$.
    Also, we let $R \deq r_0 + r_1\sqrt{d_1}$ denote the radius.
\end{assumption}

\begin{assumption}[Local strong concavity in the interior]\label{assumption: random convexity of regular part at true parameter}
  There exist constants $c_{S_0}, r_0, r_1 > 0$ such that with probability at least $1-\eta$, for all $\theta \in B(\theta^{\star}, r_0, r_1)$, $$-\grad_{S_0}^2 \ell_n(\theta) \succeq c_{S_0}I \succ 0\,.$$ We informally refer to this as the $\Theta(1)$-region of strong concavity around the true parameter in the regular part.
\end{assumption}

\begin{assumption}[Locally negative gradient at the boundary]\label{ass:negative gradient l_n}
    There exists a constant $c_{S_1} > 0$ such that $\partial_j \ell^\star(\theta^\star) \le -c_{S_1}$ for all $j\in S_1$.
\end{assumption}

\begin{assumption}[Subexponential score]\label{ass:subG}
    There exists $\sigma > 0$ such that for any $\theta \in B(\theta^\star, r_0, r_1)$ and any unit vector $v \in S^{d-1}$, $\langle v, \nabla \ell(\theta; X_1)\rangle$ is $\sigma$-subexponential.
    That is,
    \begin{align*}
        \|\langle v, \nabla \ell(\theta; X_1) - \E\nabla \ell(\theta; X_1)\rangle\|_{\psi_1} \le \sigma\,,
    \end{align*}
    where $\|\cdot\|_{\psi_1}$ is the subexponential Orlicz norm.
\end{assumption}

\begin{assumption}[Boundedness of derivatives]\label{ass:bounded_mixed_derivatives_ell_n}
   
    We assume that with probability at least $1-\eta$
    \begin{align*}
        \sup_{\theta \in B(\theta^\star,r_0,r_1)}\;{\|\nabla^2 \ell_n(\theta)\|_{\rm op}} \le s_2\,.
    \end{align*}
\end{assumption}

Finally, we assume that posterior concentration (\cref{assumption:posterior contraction}) holds.

\begin{remark}[Remarks on the assumptions]
    Our assumptions are a mix of assumptions on the empirical likelihood $\ell_n$, as well as assumptions on the population likelihood $\ell^\star$.
    In principle, the assumptions could be imposed solely at the level of the population quantities. For example, in lieu of~\cref{assumption: random convexity of regular part at true parameter}, we could assume that $-\nabla^2_{S_0} \ell^\star$ is positive definite in a $\Theta(1)$ region, as well as to assume a concentration inequality for $\|\nabla^2_{S_0} \ell_n - \nabla^2_{S_0} \ell^\star\|_{\rm op}$ (or sufficient conditions for such an inequality to hold); in fact, we provide an approach to checking~\cref{assumption: random convexity of regular part at true parameter} via this approach in~\cref{general proof of local convexity} below.
    However, for concrete examples, it is sometimes easier (and leads to sharper bounds) to verify the positive definiteness of the empirical Hessian $-\nabla_{S_0}^2 \ell_n$ directly, rather than passing through concentration.
    In stating our set of assumptions, we opted for conditions which cover a range of examples of interest; see~\cref{section:examples}.

    Another notable difference is that we now require many of our assumptions to hold on the larger set $B(\theta^\star, r_0,r_1)$, rather than the good set $\hat\Theta_\delta$.
    This can be justified as follows, e.g., for~\cref{assumption: random convexity of regular part at true parameter}.
    If the population likelihood is strongly concave at the maximizer, $-\nabla_{S_0}^2 \ell^\star(\theta^\star) \succeq 2c_{S_0} I$, and if $\nabla^2\ell^\star$ is $s_3$-Lipschitz, then we would have $-\nabla_{S_0}^2 \ell^\star(\theta) \succeq c_{S_0} I$ for all $\theta$ in an $\ell_2$ ball around $\theta^\star$ with radius $r_0 = c_{S_0}/s_3$, and then we could expect the empirical Hessian to satisfy a similar inequality.
    Similar considerations apply for the other assumptions.
    Thus, it is reasonable to suppose that these assumptions hold on a region of size $\Theta(1)$ (rather than just on the good set $\hat\Theta_\delta$), and we will verify the assumptions for our examples in~\cref{section:examples}.

    Finally,~\cref{ass:consistency} is new.
    If $\hat\theta$ (the MLE) is consistent, then~\cref{ass:consistency} is satisfied for all  sufficiently large sample sizes.
    In general, we note that existence and consistency are often verified separately; see, e.g.,~\cite[Section 5.2]{Vaart_1998}.
\end{remark}

\subsection{Intermediate results}

In this section, we gather together intermediate results needed to prove our main results in the subsequent section.

\paragraph{Gradient concentration.}
We first establish convergence of the empirical gradient to the population gradient.

\begin{lemma}[Concentration of the gradient]\label{lem:grad_bd}
Adopt~\cref{assumption:existence of mle} through~\cref{ass:consistency}, \cref{ass:subG}, and \cref{ass:bounded_mixed_derivatives_ell_n}.
Then, with probability at least $1-\eta$,
\begin{align*}
    \|\nabla\ell_n(\theta^\star) - \nabla \ell^\star(\theta^\star)\|_2 \lesssim \sigma\,\Bigl(\frac{\sqrt{d+\log(1/\eta)}}{\sqrt n} + \frac{d+\log(1/\eta)}{n}\Bigr)\,.
\end{align*}
\end{lemma}
\begin{proof}
    For any unit vector $v \in S^{d-1}$, $\langle v, \nabla \ell_n(\theta^\star) - \nabla \ell^\star(\theta^\star)\rangle$ satisfies a Bernstein inequality by \cref{ass:subG}~\cite[Example 2.2.12]{vdVWel23Empirical}. Take a union bound over a covering of $S^{d-1}$.
\end{proof}

\paragraph{Rate of convergence.}
Recall that by~\cref{ass:consistency}, the MLE $\hat\theta$ is $O(1)$-close to $\theta^\star$, which generally holds by consistency.
We now show that under our assumptions, the consistency can be upgraded to a finite-sample rate of convergence.

We first show that the MLE exactly recovers the true parameter in the non-regular part.

\begin{lemma}[Negative gradient in the non-regular part]\label{lemma: negative gradient at the boundary ln random}
  Adopt~\cref{assumption:existence of mle} through~\cref{ass:consistency}, and \cref{ass:negative gradient l_n} through \cref{ass:bounded_mixed_derivatives_ell_n}.
  Assume that
  \begin{align*}
      R \le \frac{c_{S_1}}{3s_2} \qquad \text{and}\qquad n \gg \Bigl(\frac{\sigma}{c_{S_1}} \vee \frac{\sigma^2}{c_{S_1}^2}\Bigr)\, \bigl(d + \log(1/\eta)\bigr)\,.
  \end{align*}

  Then, with probability at least $1-\eta$, $$\partial_{j} \ell_n(\theta) \leq -\frac{c_{S_1}}{2} < 0\qquad\text{for all}~\theta\in B(\theta^\star,r_0,r_1)~\text{and}~j\in S_1\,.$$
\end{lemma} 
\begin{proof}
    By~\cref{ass:negative gradient l_n}, \cref{ass:bounded_mixed_derivatives_ell_n}, and~\cref{lem:grad_bd}, for all $\theta \in B(\theta^\star, r_0, r_1)$ and $j\in S_1$,
    \begin{align*}
        \partial_j \ell_n (\theta)
        &\le \partial_j \ell_n(\theta^\star) + s_2\,\|\theta-\theta^\star\|_2 \le \partial_j \ell^\star(\theta^\star) + \sigma \,\Bigl(\frac{\sqrt{d+\log(1/\eta)}}{\sqrt n} + \frac{d+\log(1/\eta)}{n}\Bigr) + s_2 R \\
        &\le -c_{S_1} + \sigma \,\Bigl(\frac{\sqrt{d+\log(1/\eta)}}{\sqrt n} + \frac{d+\log(1/\eta)}{n}\Bigr) + s_2 R
        \le -\frac{c_{S_1}}{2}\,,
    \end{align*}
    provided $R \le c_{S_1}/(3s_2)$ and $n$ is sufficiently large.
\end{proof}

\begin{corollary}[Non-regular coordinates]\label{cor:non_reg_rate}
    Assume that the conclusion of~\cref{lemma: negative gradient at the boundary ln random} holds.
    Then, $\hat\theta_{S_1} = \theta^\star_{S_1}$.
\end{corollary}
\begin{proof}
    By~\cref{ass:consistency}, $\hat\theta \in B(\theta^\star,r_0,r_1)$.
    By~\cref{lemma: negative gradient at the boundary ln random}, we know that for $\theta$ belonging to this set and all $j\in S_1$, $\partial_j \ell_n(\theta) < 0$.
    Hence, the maximizer $\hat\theta$ of $\ell_n$ must have $\hat\theta_{S_1} = 0 = \theta_{S_1}^\star$.
\end{proof}

Building on this result, we can now obtain the rate of convergence in the regular part by strong concavity of $\ell_n$.

\begin{lemma}[Rate of convergence in the regular part]\label{lemma:uniform convergence regular part}
    Suppose that the conditions of~\cref{lemma: negative gradient at the boundary ln random} hold.
    Furthermore, adopt~\cref{assumption: random convexity of regular part at true parameter}.
    Then, with probability at least $1-2\eta$, it holds that
    \begin{align*}
        \|\hat\theta_{S_0} - \theta_{S_0}^\star\|_2
        &\;\lesssim\; \frac{\sigma}{c_{S_0}}\,\Bigl( \frac{\sqrt{d_0 + \log(1/\eta)}}{\sqrt n} + \frac{d_0+\log(1/\eta)}{n}\Bigr)\,.
    \end{align*}

\end{lemma}
\begin{proof}
    By \cref{assumption: random convexity of regular part at true parameter}, $\ell_n$ is strongly concave over $B(\theta^\star,r_0,r_1)$, and $\hat\theta \in B(\theta^\star,r_0,r_1)$ by~\cref{ass:consistency}.
    Also, by~\cref{cor:non_reg_rate}, $\hat\theta_{S_1} = \theta_{S_1}^\star$, so that $\hat\theta_{S_0}$ maximizes $\theta_{S_0} \mapsto \ell_n(\theta_{S_0}, \theta_{S_1}^\star)$.
    Then,
    \begin{align*}
        c_{S_0}\,\|\hat\theta_{S_0} - \theta_{S_0}^\star\|_2
        &\le \|\nabla_{S_0} \ell_n(\theta_{S_0}^\star, \theta_{S_1}^\star)\|_2
        = \|\nabla_{S_0} \ell_n(\theta_{S_0}^\star,\theta_{S_1}^\star) - \nabla_{S_0} \ell^\star(\theta_{S_0}^\star,\theta_{S_1}^\star)\|_2\,,
    \end{align*}
    where we use $\nabla_{S_0} \ell^\star(\theta_{S_0}^\star,\theta_{S_1}^\star) = 0$.
    Applying~\cref{lem:grad_bd} (restricting to the coordinates in $S_0$) yields the result.
\end{proof}

Recall that $\hat\Theta_{\delta}\coloneqq \{\theta \in \Theta: \norm{\theta_{S_0} - \hat \theta_{S_0}}_{2} < \delta_0  \sqrt{\frac{d_0}{n}},\, \norm{\theta_{S_1} - \hat \theta_{S_1}}_{\infty} <  \frac{\delta_1}{n}\}$, where $\hat\Theta_\delta$ is the good set (\cref{def:good_set}). Note that $\hat\Theta_\delta$ is now a random set.
We need the following lemma.
\begin{lemma}[Concentration region]\label{lemma:contraction region} 
    Suppose that both~\cref{cor:non_reg_rate} and \cref{lemma:uniform convergence regular part} hold.
    If
    \begin{align*}
        n \gg \max\Bigl\{\Bigl( \frac{\sigma}{c_{S_0} r_0} \vee \frac{\sigma^2}{c_{S_0}^2 r_0^2}\Bigr)\,\bigl(d+\log(1/\eta)\bigr)\,, \; \bigl(\frac{\delta_0}{r_0}\bigr)^2\,d_0\,,\; \frac{\delta_1}{r_1}\Bigr\}\,,
    \end{align*}
    it holds that
  $$\hat \Theta_{\delta} \subseteq B(\theta^\star, r_0, r_1)\,.$$
\end{lemma}
\begin{proof}
    The proof is immediate.
\end{proof}

\subsection{Main results}\label{sec:main_results_random}

We now establish our main result.
\begin{theorem}[Poincar\'e inequality for the random constrained Gibbs measure]\label{theorem:random_poincare}
    Suppose that~\cref{integrand}, as well as \cref{assumption:existence of mle} through \cref{ass:bounded_mixed_derivatives_ell_n}, hold.
    With probability at least $1-O(\eta)$, for $3R \le c_{S_1}/s_2$, 
    and
    \begin{align*}
        n \gg \max\Bigl\{\phi\Bigl(\frac{\sigma}{c_{S_0} r_0 \wedge c_{S_1}}\Bigr)\, \bigl(d + \log(1/\eta)\bigr)\,,\; \bigl(\frac{\delta_0}{c_0}\bigr)^2\,d_0\,, \; \frac{\delta_1}{r_1}\Bigr\}\,,
    \end{align*}
    where $\phi(x) \deq x \vee x^2$,
    the posterior distribution $\mu$ in \eqref{eq:random_laplace_density} conditioned on the good set $\hat\Theta_\delta$ satisfies a Poincar\'e inequality with Poincar\'e constant 
    at most
    \begin{align*}
        \CPI(\mugood) \le \Bigl(\frac{1}{nc_{S_0}} \vee \frac{16 \max_{j\in S_1} \exp(\osc_{[0,\delta_1/n]} \log \pi_j)}{n^2 c_{S_1}^2}\Bigr) \exp\Bigl(2s_2 \,\Bigl(\frac{\delta_0 \delta_1 \,(d_0 d_1)^{1/2}}{n^{1/2}} + \frac{\delta_1^2 d_1}{n}\Bigr)\Bigr)\,.
    \end{align*}

\end{theorem}
\begin{proof}
    We apply \cref{theorem:poincare} to $\mu|_{\hat\Theta_\delta}$, and so we must check the assumptions of that theorem.
    \begin{itemize}
        \item \cref{ass:mode} holds with probability at least $1-\eta$ by \cref{assumption:existence of mle}.
        \item \cref{assumption:dimension} holds with $C_0 = c_0/2$ by \cref{assumption:dimension_random} and \cref{lemma:uniform convergence regular part}, provided that $n\gg (\sigma + s_2 R)^2\,(d+\log^2(1/\eta))/(c_0 c_{S_0})^2$.
        \item \cref{ass:local_cvxty} holds with $C_{S_0} = c_{S_0}$ by \cref{assumption: random convexity of regular part at true parameter} and \cref{lemma:contraction region}.
        \item \cref{assumption:negative_gradient_l} holds with $C_{S_1} = c_{S_1}/2$ by \cref{lemma: negative gradient at the boundary ln random}.
        \item \cref{assumption:bounded_mixed_derivatives} follows from \cref{ass:bounded_mixed_derivatives_ell_n} and \cref{lemma:contraction region}.
    \end{itemize}
\end{proof}

As discussed in~\cref{section:main_results}, the Poincar\'e inequality implies rapid mixing for Langevin-type MCMC algorithms for sampling from the posterior $\mu$. 

\subsection{Checking local concavity}

By the second-order optimality condition for the population log-likelihood at the true parameter $\theta^\star$, we have $\grad^2_{S_0} \ell^{\star}(\theta^\star) \preceq 0$.
One way to check the strong concavity of the empirical log-likelihood (\cref{assumption: random convexity of regular part at true parameter}) is to assume a constant region of strong concavity for the population log-likelihood around $\theta^\star$ in the regular part of the parameter space, and then to transfer this property to the empirical log-likelihood via concentration.

\begin{assumption}[Local strong concavity in the interior]\label{assumption: l_star convexity of regular part at true parameter}
  There exists a constant $c_{S_0}^\star >0$ such that for all $\theta \in B(\theta^{\star}, r_0, r_1)$, $$-\grad_{S_0}^2 \ell^\star(\theta) \succeq c_{S_0}^\star I \succ 0\,.$$ Here, the radii $r_0, r_1 > 0$ are as in \cref{ass:consistency}.
\end{assumption}

\begin{assumption}[H\"older continuity of the Hessian]\label{assumption:holder_hessian}
    With probability at least $1-\eta$, the Hessian of the empirical log-likelihood is H\"older-continuous on $B(\theta^\star,r_0,r_1)$, i.e., for any $\theta, \theta' \in B(\theta^\star,r_0,r_1)$, $$\norm{\grad^2 \ell_n(\theta) - \grad^2 \ell_n(\theta')}_{\rm op} \leq s_{2+\gamma}\, \norm{\theta - \theta'}_2^\gamma$$
    for some $s_{2+\gamma} > 0$ and $\gamma \in (0,1]$.
\end{assumption}

\begin{lemma} \label{lemma: l_n convexity}
  Adopt \cref{assumption:existence of mle,assumption:dimension_random,ass:consistency}, \cref{assumption: l_star convexity of regular part at true parameter}, and \cref{assumption:holder_hessian}.
    We have $$\inf_{\theta \in B(\theta^\star, r_0, r_1)}\lambda_{\min}(-\grad^2_{S_0} \ell_n(\theta)) \geq \frac{c_{S_0}^\star}{2} - \max_{\theta \in N_\epsilon(\theta^\star, r_0, r_1)}\norm{\grad^2_{S_0}\ell_n(\theta) - \grad^2_{S_0}\ell^\star(\theta)}_{\rm op}$$ with probability at least $1-\eta$, provided that $N_\varepsilon(\theta^\star,r_0,r_1)$ is an $\varepsilon$-covering net of $B(\theta^\star,r_0,r_1)$ with $\varepsilon \le (\frac{c_{S_0}^\star}{2s_{2+\gamma}})^{1/\gamma}$.
\end{lemma}

\begin{proof}\label{general proof of local convexity}
    This lemma combines the strong concavity region of $\ell^\star$ in \cref{assumption: random convexity of regular part at true parameter} and the Hessian continuity in \cref{assumption:holder_hessian} via a covering net argument.

    Let $N_\epsilon(\theta^\star, r_0, r_1)$ denote an $\epsilon$-covering net of $B(\theta^\star, r_0, r_1)$,
     and $B(\theta, \epsilon)$ denote the ball of radius $\epsilon$ around $\theta$. By Weyl's inequality, on the event that~\cref{assumption:holder_hessian} holds,
     \begin{align*}
        &\inf_{\theta \in B(\theta^\star, r_0, r_1)}\lambda_{\min}(-\grad^2_{S_0} \ell_n(\theta)) \\
        &\qquad \ge \inf_{\theta \in N_\varepsilon(\theta^\star,r_0,r_1)} \lambda_{\min}(-\grad^2_{S_0} \ell_n(\theta)) - \sup_{\substack{\theta \in N_\varepsilon(\theta^\star,r_0,r_1) \\ \theta_\varepsilon \in B(\theta,\varepsilon)}}\|\nabla_{S_0}^2 \ell_n(\theta) - \nabla_{S_0}^2 \ell_n(\theta_\varepsilon)\|_{\rm op} \\
        &\qquad \ge \inf_{\theta \in N_\varepsilon(\theta^\star,r_0,r_1)}\bigl\{\lambda_{\min}(-\nabla_{S_0}^2\ell^\star(\theta)) - \|\nabla^2_{S_0} \ell_n(\theta) - \nabla^2_{S_0} \ell^\star(\theta)\|_{\rm op}\bigr\} - s_{2+\gamma}\, \varepsilon^\gamma \\
        &\qquad \ge c_{S_0}^\star - \max_{\theta \in N_\varepsilon(\theta^\star,r_0,r_1)} \|\nabla^2_{S_0} \ell_n(\theta) - \nabla^2_{S_0} \ell^\star(\theta)\|_{\rm op} - s_{2+\gamma}\, \varepsilon^\gamma \\
        &\qquad \ge \frac{c_{S_0}^\star}{2} - \max_{\theta \in N_\varepsilon(\theta^\star,r_0,r_1)} \|\nabla^2_{S_0} \ell_n(\theta) - \nabla^2_{S_0} \ell^\star(\theta)\|_{\rm op}\,,
     \end{align*}
     provided $\varepsilon \le (\frac{c_{S_0}^\star}{2s_{2+\gamma}})^{1/\gamma}$.
\end{proof}

    Using the result from \cref{lemma: l_n convexity}, we see that to verify the local concavity of the empirical log-likelihood in \cref{ass:local_cvxty}, it suffices to verify~\cref{assumption: l_star convexity of regular part at true parameter,assumption:holder_hessian}, and the uniform convergence of the Hessian difference over the covering net $N_\epsilon(\theta^\star, r_0, r_1)$.
    This can be done by applying a union bound over the finite covering net and using the pointwise concentration. We will demonstrate this in specific examples in \cref{section:glm,section:gaussian_mixture_model}.

\section{Applications to Bayesian inference}
\label{section:examples}

We are now ready to apply our main result in \cref{theorem:random_poincare} to specific statistical models. We will consider the following models: Generalized Linear Models (GLMs) in \cref{section:glm}, including logistic regression with random design in \cref{section:logistic_regression} and a Poisson linear model in \cref{section:poisson_linear_model}, and Gaussian mixture models in \cref{section:gaussian_mixture_model}.

In order to apply~\cref{theorem:random_poincare}, we must check~\cref{integrand}, \cref{assumption:posterior contraction}, as well as~\cref{assumption:existence of mle} through~\cref{ass:bounded_mixed_derivatives_ell_n}.
We now discuss these assumptions in turn.
\begin{itemize}
    \item \cref{integrand} states that the prior is log-concave and of product form.
    In order to check posterior concentration, we will also assume that the prior is log-Lipschitz (\cref{assumption:Gradient bound on the prior}) with a sufficiently small constant.
    For simplicity, one can think of our results as holding for a flat (uninformative) prior.
    \item \cref{assumption:existence of mle} through \cref{ass:consistency} concern the existence and uniqueness of the MLE\@, and that the regular part of $\theta^\star$ should lie strictly in the interior.
    These conditions are typically verified using arguments orthogonal to the ones in this paper, and will simply be assumed here.
    \item Similarly,~\cref{ass:negative gradient l_n}, which states that the gradient in the non-regular part at $\theta^\star$ is strictly negative, should be viewed as specifying the situation under consideration and will simply be assumed throughout.
    \item \cref{assumption: random convexity of regular part at true parameter} asserts strong concavity of the empirical Hessian of the log-likelihood.
    This will either be checked explicitly, or we will assume strong concavity of the population counterpart (\cref{assumption: l_star convexity of regular part at true parameter}) and invoke~\cref{lemma: l_n convexity}.
    \item \cref{ass:subG} (subexponential score) and~\cref{ass:bounded_mixed_derivatives_ell_n} (boundedness of the Hessian) will be checked explicitly.
    \item Finally,~\cref{assumption:posterior contraction} (posterior concentration on the good set) will be checked either via~\cref{thm:concentration_log_concave} or~\cref{proof_lemma: exponential_decay}.
    These theorems introduce additional assumptions:
    \begin{itemize}
        \item \cref{posterior consistency} (consistency of a posterior draw) will simply be assumed, similarly to consistency of the MLE\@.
        \item The theorems assume some form of log-concavity, either globally (\cref{thm:concentration_log_concave}) or locally (\cref{proof_lemma: exponential_decay}).
        These will be checked for the specific models.
        \item \cref{proof_lemma: exponential_decay} additionally assumes a compact parameter space (\cref{assumption_contraction:compact_parameter_space}) and a well-separated mode condition (\cref{assumption:exponential_decay}). Both conditions will simply be assumed when invoking~\cref{proof_lemma: exponential_decay}.
        \item We employ the notation $B(\theta^\star, r_0, r_1)$ and $B(\hat\theta, r_0, r_1)$ using the same symbols for the radii. However, implicitly, the radii defining the neighborhood around the estimator $\hat\theta$ are chosen to be smaller (e.g., scaled by a factor of one-half) than those around the true parameter $\theta^\star$. This adjustment is necessary to ensure the inclusion $B(\hat\theta, r_0, r_1) \subseteq B(\theta^\star, 2r_0, 2r_1)$, which follows from \cref{lemma:uniform convergence regular part} and \cref{cor:non_reg_rate}. Consequently, the values of $r_0$ and $r_1$ selected here effectively correspond to half the values used in \cref{sec:random_laplace_density}. For simplicity, we will not explicitly distinguish between these scaled radii in the subsequent examples.
        \item Finally, both theorems assume some conditions on derivatives of the log-likelihood, but these are implied by \cref{assumption: random convexity of regular part at true parameter}, \cref{ass:bounded_mixed_derivatives_ell_n}, and \cref{lemma: negative gradient at the boundary ln random}.
    \end{itemize}
\end{itemize}

To summarize, we will always invoke the following suite of basic assumptions.

\begin{description}
    \item[Assumption (B)\label{ass:basic}] We assume that \cref{integrand}, \cref{assumption:Gradient bound on the prior} (with sufficiently small $L_\pi$), \cref{posterior consistency}, \cref{assumption:existence of mle} through \cref{ass:consistency}, and \cref{ass:negative gradient l_n} hold.
\end{description}

We will focus on checking \cref{assumption: random convexity of regular part at true parameter} (strong concavity of the empirical Hessian), \cref{ass:subG} (subexponential score), \cref{ass:bounded_mixed_derivatives_ell_n} (boundedness of the Hessian), and the conditions for posterior concentration.

\begin{remark}[Choice of $r_0$, $r_1$]
    \cref{thm:concentration_log_concave} and \cref{theorem:random_poincare} impose the following conditions on $r_0$ and $r_1$:
    \begin{align*}
        r_1 \le \frac{c_{S_0} c_{S_1}}{4s_2^2}\,, \qquad r_0 \le c_0\,, \qquad r_0 + \sqrt{d_1}\,r_1 \le \frac{c_{S_1}}{3s_2}\,.
    \end{align*}
    The chain of reasoning is summarized as follows: first, we obtain bounds on $c_{S_0}$ and $s_2$ (which hold for $n$ sufficiently large).
    Then, we choose $r_0$, $r_1$ to be the largest possible values satisfying the above constraints.
    Then, we assume that \cref{ass:consistency} holds for these choices of $r_0$, $r_1$, and that \cref{posterior consistency} holds for $r_0' \asymp r_0/\log(1/\varepsilon)$.
    The logic is not circular since the derivation of bounds on $c_{S_0}$ and $s_2$ do not depend on \cref{ass:consistency}.
\end{remark}

\subsection{Generalized linear models}\label{section:glm}

We consider Generalized Linear Models (GLMs) in the following setting. We observe feature-label pairs $(X_i, Y_i)$ for $i = 1, \dotsc, n$, where the feature $X_i \in \R^{d}$ is a $d$-dimensional vector and the label $Y_i \in \R$. Given parameters $\theta \in \R^{d}$, the model for the distribution of $Y_i$ given $X_i$ is
\begin{equation}
Y_i \mid X_i \sim p(y \mid X_i^\T \theta)\,d \lambda(y)\,,\label{model:glm}
\end{equation}
for some reference measure $\lambda$. Let $\{p(\cdot \mid \eta) : \eta \in \Xi \subseteq \R\}$ be an exponential family:
\begin{equation} 
    p(y \mid \eta) = \exp{ b(\eta)\, y - c(\eta)}\,,  \label{model:glm_formula}
\end{equation}
where $c(\eta) = \log \int e^{b(\eta)\,y} \,d\lambda(y)$ is the log-partition function. Based on the model \eqref{model:glm_formula}, we have the normalized log-likelihood $$\ell_n(\theta) = \frac{1}{n}\sum_{i=1}^{n} \bigl(Y_i\, b(X_i^\T\theta) - c(X_i^\T\theta)\bigr)\,.$$ 

The derivatives of the log-likelihood are
\begin{align} 
    \grad \ell_n(\theta) &= \frac{1}{n}\sum_{i=1}^{n} \bigl(Y_i\, b^\prime(X_i^T\theta) - c'(X_i^\T\theta)\bigr)\,X_i\,, \label{eq:glm_gradient} \\
    \grad^2 \ell_n(\theta) &= \frac{1}{n}\sum_{i=1}^{n} \bigl(Y_i\, b^{\prime\prime}(X_i^\T\theta)-  c''(X_i^\T\theta)\bigr)\,X_i X_i^\T\,. \label{eq:glm_hessian} 
\end{align}

\begin{remark}
    The two examples we consider below possess concave log-likelihoods, so that guarantees for sampling from the posterior follow from standard log-concave theory~\cite{Chewi25Book}.
    We consider them in order to illustrate how to check the assumptions of our general theory for concrete examples.
    We note, however, that our framework only requires \emph{local} log-concavity, and therefore could be applied to non-convex GLMs as well.
    We consider a non-log-concave example in~\cref{section:gaussian_mixture_model}.
\end{remark}

\subsubsection{Logistic regression with random design}\label{section:logistic_regression}

We observe feature-label pairs $(X_i, Y_i)$ for $i = 1, \dotsc, n$, where $X_i \in \R^{d}$ is a feature vector and $Y_i \in \{0, 1\}$ is the label of $X_i$. Given parameter $\theta \in \Theta \deq [0,\infty)^{d}$, the logistic regression model for the distribution of $Y_i$ given $X_i$ is
\begin{equation}
Y_i \mid X_i \sim \mathsf{Bernoulli}(c^{\prime}(X_i^\T\theta))\,, \qquad i = 1, \dotsc, n\,,\label{model:logistic_regression}
\end{equation}
where the log-partition function is given by $c(\eta) = \log(1 + e^{\eta})$ and its derivative is the logistic function $c^{\prime}(\eta) = e^\eta/(1 + e^{\eta})$. 

This is a special case of the GLM in \eqref{model:glm_formula} with $b(\eta) = \eta$ and $c(\eta) = \log(1 + e^{\eta})$. 
The gradient and Hessian are given by
\begin{align} 
    \grad \ell_n(\theta) &= \frac{1}{n}\sum_{i=1}^{n} (Y_i - c^{\prime}(X_i^T\theta))\,X_i\,, \label{eq:logistic_regression_gradient} \\
    \grad^2 \ell_n(\theta) &= -\frac{1}{n}\sum_{i=1}^{n}  c^{\prime\prime}(X_i^\T\theta)\,X_i X_i^\T\,. \label{eq:logistic_regression_hessian}
\end{align}
The Hessian of the negative log-likelihood is positive semi-definite as $c$ is convex on $\eta > 0$. Moreover, the Hessian of the log-likelihood does not depend on the random variable $Y_i$. Therefore, $\grad^2 \ell_n(\theta) = \grad^2 \ell^{\star}(\theta)$ and same for higher order derivatives conditioning on the design $X_i$'s.

We now describe our particular setting of interest. We focus on the \textbf{random design} case, $X_i \simiid \mathsf{N}(0, I_d)$ for $i=1,\dotsc,n$.
We denote by $B \deq \max\{e, \norm{\theta^\star}_2\}$, and we treat the constants $c_0$, $c_{S_1}$, $r_0$, $r_1$, $B$ as dimension-free.

\begin{remark}
    As discussed in \cite{chardon2024finitesampleperformancemaximumlikelihood}, the regime of constant $B$ is an interesting one as it implies that logistic regression with Gaussian design has good statistical properties, e.g., the MLE exists and the excess risk decays at the rate of $O(\frac{d}{n})$
    with high probability when $n \gtrsim Bd$.
\end{remark}

The following theorem checks the remaining assumptions of \cref{theorem:random_poincare}.

\begin{theorem}\label{check:assumptions_lr}
    Adopt~\ref{ass:basic} for the logistic regression model with Gaussian design in \eqref{model:logistic_regression}.
    Then, the following assertions hold with constants depending on $B$ and with probability at least $1-O(\eta)$, provided $n\gg d +\log(1/\eta)$.
    \begin{itemize}
        \item \cref{assumption: random convexity of regular part at true parameter} holds with $c_{S_0} \gtrsim 1$~\cite[Theorem 6]{chardon2024finitesampleperformancemaximumlikelihood}.
        \item \cref{ass:subG} holds with $\sigma \lesssim 1$.
        \item \cref{ass:bounded_mixed_derivatives_ell_n} holds with $s_2 \lesssim 1$.
    \end{itemize}
\end{theorem}

\begin{corollary}\label{corollary:logistic_regression_main_theorem}
    Adopt~\ref{ass:basic}.
    There are constants $\bar c_0,\dotsc,\bar c_3$ depending on $c_0$, $c_{S_1}$, and $B$ such that
    with probability at least $1-O(\eta)$, 
    if
    \begin{align*}
        \delta_0 = \bar c_0 \log\frac{1}{\varepsilon}\,, \qquad \delta_1 = \bar c_1\log\frac{d_1}{\varepsilon}\,, \qquad d_1 \le d_0\,, \qquad
        n \ge \bar c_2 \,\Bigl[d_0d_1 \log^2\bigl(\frac{d}{\varepsilon}\bigr) + \log\frac{1}{\eta}\Bigr]\,,
    \end{align*}
    then the posterior distribution $\mu$ in \eqref{eq:random_laplace_density} conditioned on the good set $\hat\Theta_\delta$ satisfies a Poincar\'e inequality with Poincar\'e constant 
    at most
    \begin{align*}
        \CPI(\mugood) \le \frac{\bar c_3}{n}\,,
    \end{align*}
    and $\mu(\Theta\setminus \hat\Theta_\delta) \le \varepsilon$.
 \end{corollary}   
 \begin{proof}
    The result follows from~\cref{thm:concentration_log_concave},~\cref{theorem:random_poincare}, and \cref{check:assumptions_lr}.
 \end{proof}

\subsubsection{Poisson linear model}\label{section:poisson_linear_model}

We consider the Poisson linear model from~\citep{bochkinaBernsteinMisesTheoremNonregular2014}.
 Single photon emission computed tomography (SPECT) is a medical imaging technique in which a radioactively labelled tracer, known to concentrate in the tissue to be imaged, is introduced into the subject. Emitted particles are detected in a device called a gamma camera, forming an array of counts. Tomographic reconstruction is the process of inferring the spatial pattern of concentration of the tracer in the tissue from these counts. The Poisson linear model is given by
\begin{equation}\label{poisson_linear_model}
TY_i \simiid \mathsf{Poisson}(T A_i \theta)\,, \qquad i\in [n]\,.
\end{equation}

Here, $\theta \in \R^d$ is the unknown parameter to be estimated, representing the spatial distribution of the concentration of the tracer and typically discretized into $d$ pixels or voxels. The array $Y_i$ represents the rate of detected photons per time unit. The $n \times d$ matrix $A=(A_1^\T, \cdots, A_n^\T)^\T$ is formed so that each $A_i$ represents the sensitivity of the $i$-th detector to each pixel or voxel in the image. Specifically, $A_{ij}$ is the mean number of photons detected by the $i$-th detector per time unit and per unit concentration of the tracer in the $j$-th pixel or voxel, and it is non-negative. The parameter $T$ is a known positive constant representing the total observation time or exposure time during which the counts are collected. We assume that the model is well-specified, so that the data is drawn from a ground truth parameter $\theta^\star$.

\begin{assumption}\label{assumption:A and theta}
    For $i \in [n]$ and $j \in [d]$, $A_{ij} \geq 0$, $\theta_j^\star \geq 0.$ 
    Moreover, $\norm{A_i}_2 \leq a_2$, for all $i \in [n]$; $A_i\theta^\star \geq c > 0$ for all $i$ such that $Y_i > 0$.  We assume that at least one observation is not $0$, i.e., $Y_i > 0$ for some $i$. 
\end{assumption}

Note that $Y_i > 0$ implies $A_i\theta^\star > 0$, so the condition that $Y_i > 0$ implies $A_i\theta^\star \ge c > 0$ is a natural quantitative assumption.

The model has density  $p_\theta(y_i) = \frac{e^{-T A_i\theta}\,(T A_i\theta)^{Ty_i}}{(Ty_i)!}$, and the log-likelihood is given by $\ell(\theta; y_i) = -T A_i\theta + Ty_i \log{T A_i\theta} - \log((Ty_i)!)$. We consider the parameter space $\Theta = [0, \infty)^d$ and note that by~\cref{assumption:A and theta}, $A_i \theta^\star > 0$, and thus $A_i\theta > 0$ for all $\theta$ sufficiently close to $\theta^\star$.
The gradient of the log-likelihood is $$\grad \ell(\theta; y_i) = -TA_i + \frac{Ty_i A_i}{A_i\theta}\,,$$
and the Hessian is
$$\grad^2\ell(\theta; y_i) = -\frac{Ty_i A_i^\T {A_i}}{( A_i\theta)^2}\,.$$

Below, we assume that $R \le c/(2a_2)$, which simplifies the analysis but could be restrictive. We leave a more refined study for future work.

\begin{theorem}\label{theorem:poisson_check_assumptions}
     Adopt~\ref{ass:basic}, \cref{assumption: l_star convexity of regular part at true parameter}, and \cref{assumption:A and theta} for the Poisson linear model in \eqref{poisson_linear_model}. Then, the following assertions hold with probability $1-O(\eta)$, provided
     \begin{align*}
         n \gg \log\frac{1}{\eta} \vee \phi\Bigl(\frac{Ta_2^2}{cc_{S_0}^\star}\Bigr)\log\frac{d}{\eta} \qquad\text{and}\qquad R \le \frac{c}{2a_2}\,,
     \end{align*}
     where $\phi(x) \deq x \vee x^2$.

    \begin{itemize}
        \item \cref{assumption: random convexity of regular part at true parameter} holds with $c_{S_0} = c^\star_{S_0}/4$.
        \item \cref{ass:subG} holds with $\sigma \lesssim Ta_2$.
        \item \cref{ass:bounded_mixed_derivatives_ell_n} holds with  $s_2\lesssim Ta_2^2/c$.
    \end{itemize}
\end{theorem}

To simplify the statement of the next result, we treat the constants $c_0$, $c_{S_0}^\star$, $c_{S_1}$, $Ta_2$, and $a_2/c$ as dimension-free. We also assume that $R \le c/(2a_2)$.

\begin{corollary}\label{corollary:poisson_linear_model_main_theorem}
    Adopt~\ref{ass:basic}, \cref{assumption: l_star convexity of regular part at true parameter}, and \cref{assumption:A and theta} for the Poisson linear model in \eqref{poisson_linear_model}.
    There are constants $\bar c_0,\dotsc,\bar c_3$ depending on $c_0$, $c_{S_0}^\star$, $c_{S_1}$, $c$, $Ta_2$, and $a_2/c$ such that
    with probability at least $1-O(\eta)$, 
    if
    \begin{align*}
        R \leq \frac{c}{2a_2}\,, \quad \delta_0 = \bar c_0 \log\frac{1}{\varepsilon}\,, \quad \delta_1 = \bar c_1\log\frac{d_1}{\varepsilon}\,, \quad d_1 \le d_0\,, \quad
        n \ge \bar c_2 \Bigl[d_0d_1 \log^2\bigl(\frac{d}{\varepsilon}\bigr) + \log\frac{d}{\eta}\Bigr]\,,
    \end{align*}
    then the posterior distribution $\mu$ in \eqref{eq:random_laplace_density} conditioned on the good set $\hat\Theta_\delta$ satisfies a Poincar\'e inequality with Poincar\'e constant 
    at most
    \begin{align*}
        \CPI(\mugood) \le \frac{\bar c_3}{n}\,,
    \end{align*}
    and $\mu(\Theta\setminus \hat\Theta_\delta) \le \varepsilon$.
    
 \end{corollary} 
 
 \begin{proof}
    The result follows from~\cref{thm:concentration_log_concave},~\cref{theorem:random_poincare}, and \cref{theorem:poisson_check_assumptions}.
 \end{proof}
 

\subsection{Gaussian mixture model}\label{section:gaussian_mixture_model}
The Gaussian mixture model is a widely used model in statistics and machine learning \cite{lindsay1995mixture}. 
We consider the Gaussian mixture model with $k$ components, where $k \geq 2$. The posterior distribution is generally not log-concave and clearly multimodal, and it is therefore a challenging problem to efficiently draw samples. 
We consider the following model:
\begin{equation}
    X_i \sim \sum_{j=1}^k\omega_j\sN(\mu_j, \Sigma_j)\,,  \qquad i = 1, \dotsc, n\,, \label{model:gaussian_mixture_model}
\end{equation}
where $\omega_j \in (0, 1)$ are the mixing proportions satisfying $\sum_{j=1}^k \omega_j = 1$, $\mu_j\in \R^d$ is the mean of the $j$-th component, and $\Sigma_j \in \R^{d \times d}$ is the covariance matrix of the $j$-th component.

In what follows, we assume for simplicity that the mixing proportions $\omega_1,\dotsc,\omega_k$ and the covariances $\Sigma_1,\dotsc,\Sigma_k$ are known, so that the posterior distribution is over the unknown means.
We consider a compact parameter space to ensure the posterior distribution is well-defined. We also assume that the only case in which the ground truth parameter $\theta^\star = (\mu_1^\star,\dotsc,\mu_k^\star)$ lies at the boundary of the parameter space is when some coordinates of the $\mu_j^\star$ are zero.

In this simplified setup, our goal is to show that the landscape of the posterior, locally around $\theta^\star$, becomes benign at a reasonable value of the sample size $n$.
The main purpose of this example is to check our assumptions on a genuinely non-log-concave example, and we do not claim that our quantitative bounds are tight.
We also emphasize that our result only implies efficient sampling given a warm start (i.e., knowledge of the good set), which may not be easy to achieve and is investigated in a line of works on global optimization algorithms for highly non-convex problems (e.g., \cite{doi:10.1073/pnas.2519845123}) that are orthogonal to our work.

\begin{assumption}\label{assumption:gmm_dimension}
    Let $\theta = (\mu_1, \dotsc, \mu_k)$. The parameter space is defined as $\Theta \deq \mathcal C \cap [0,\infty)^{kd}$, where $\mathcal{C}\subset\mathbb{R}^d$ is a compact set with non-empty interior.
Assume that $B(\theta^\star, r_0, r_1) \cap [0,\infty)^{kd} \subset \Theta$.

Let $R_\Theta$ denote the radius of the parameter space $\Theta$, i.e., $R_\Theta \deq \max_{\theta \in \Theta}\norm{\theta}_2$. 
\end{assumption}
\cref{assumption:gmm_dimension} ensures that the regular part of $\theta^\star$ lies in the interior of the compact set $\mathcal{C}$, and the only boundary constraints come from the non-negative orthant.

\begin{assumption}\label{assumption:gmm_covariance}
    For any $j \in [k]$,
    $$0 < \lambda_{\min} \leq \norm{\Sigma_j}_{\rm op} \leq \lambda_{\max} < \infty$$ for some constants $\lambda_{\min}, \lambda_{\max}$.
\end{assumption}

One sufficient condition for posterior concentration in this example is that the population log-likelihood has a unique well-separated maximum.

\begin{assumption}[Well-separated mode]\label{assumption:exponential_decay_l}
    There exists $\zeta^\star > 0$ such that $$\sup_{\theta \in \Theta \setminus B(\theta^\star,r_0/2, r_1/2)} \ell^\star(\theta) - \ell^\star(\theta^\star) \leq - \zeta^\star\,.$$
\end{assumption}

We now describe our particular setting of interest. For simplicity, we suppress the dependence on the parameters $c_0$, $c_{S_0}^\star$, $c_{S_1}^\star$, $r_0$, $r_1$, $k$, $\lambda_{\min}$, $\lambda_{\max}$, $R_\Theta$, and $\zeta^\star$, treating them as dimension-free, although more detailed bounds are available in \cref{app:gmm_pf}.
We are ready to check the assumptions in \cref{sec:assumptions}  and posterior concentration for the Gaussian mixture model.

\begin{theorem}\label{theorem:gaussian_mixture_model_check_assumptions}
    Adopt~\ref{ass:basic}, \cref{assumption: l_star convexity of regular part at true parameter},
     and \cref{assumption:gmm_dimension,assumption:gmm_covariance}. Then, for the Gaussian mixture model in \eqref{model:gaussian_mixture_model}, with probability at least $1-O(\eta)$, provided $n \gg d^2\log^2(1/\eta)$:
    \begin{itemize}
        \item \cref{assumption: random convexity of regular part at true parameter} holds with $c_{S_0} = c_{S_0}^\star/4$.
        \item \cref{ass:subG} holds with $\sigma \lesssim 1$.
        \item \cref{ass:bounded_mixed_derivatives_ell_n} holds with $s_2 \lesssim  1$.
        \item \cref{assumption_contraction:compact_parameter_space} and \cref{assumption:exponential_decay} hold.
    \end{itemize}
\end{theorem}

\begin{corollary}\label{corollary:gmm_main_theorem}
    Adopt~\ref{ass:basic}, \cref{assumption: l_star convexity of regular part at true parameter}, \cref{assumption:gmm_dimension}, \cref{assumption:gmm_covariance} and \cref{assumption:exponential_decay_l} for the Gaussian mixture model in \eqref{model:gaussian_mixture_model}.
    There are constants $\bar c_0, \bar c_1, \bar c_2, \bar c_3$ depending on $c_0$, $c_{S_1}$, $c_{S_0}^\star$, $r_0$, $r_1$, $k$, $\lambda_{\min}$, $\lambda_{\max}$, $R_\Theta$, and $\zeta^\star$ such that
    with probability at least $1-O(\eta)$, 
    if
    \begin{align*}
        \delta_0 = \bar c_0 \log\frac{1}{\varepsilon}\,, \qquad \delta_1 = \bar c_1\log\frac{d_1}{\varepsilon}\,, \qquad d_1 \le d_0\,, \qquad
        n \ge \bar c_2 d^2 \,\Bigl[  \log^2\bigl(\frac{d}{\varepsilon}\bigr) + \log^2(\frac{1}{\eta})\Bigr]\,,
    \end{align*}
    then the posterior distribution $\mu$ in \eqref{eq:random_laplace_density} conditioned on the good set $\hat\Theta_\delta$ satisfies a Poincar\'e inequality with Poincar\'e constant 
    at most
    \begin{align*}
        \CPI(\mugood) \le \frac{\bar c_3}{n}\,,
    \end{align*}
    and $\mu(\Theta\setminus \hat\Theta_\delta) \le \varepsilon$.
    
 \end{corollary} 
 
 \begin{proof}
    The result follows from~\cref{proof_lemma: exponential_decay},~\cref{theorem:random_poincare}, and \cref{theorem:gaussian_mixture_model_check_assumptions}.
 \end{proof}

\section{Simulation results}\label{sec:simulation}

We have conducted two simulation studies to check how our theory works in practice using the three examples we studied in the previous section. The first simulation is about pre-asymptotic regime; the second simulation shows that our method performs well in asymptotic regime as well. In both simulations, we implement the projected Langevin Monte Carlo algorithm as described in \eqref{eq:projected_langevin} to sample from the posterior distributions.

Our first simulation investigates the high-dimensional, pre-asymptotic regime, setting the dimension $d=200$ and sample size $n=800$. With the sample size being small relative to the dimension, the problem is far from its asymptotic limit. We evaluate the convergence speed of Langevin Monte Carlo (LMC), which is also called the unadjusted Langevin algorithm (ULA), by reporting the effective sample size (ESS). The ESS is a standard measure of sampler efficiency estimating the number of independent draws that would provide the same estimation variance as the autocorrelated samples from the MCMC chain. A higher ESS thus indicates faster convergence and better mixing. All ESS diagnostics are computed using the rank-normalization methodology \citep{10.1214/20-BA1221} as implemented in the \texttt{arviz} package in Python.

The study consists of 20 independent trials for each of the three statistical models. In each trial, $n=800$ samples are generated from the model, and uniform priors are employed. We particularly focus on MCMC mixing performance given a warm start; to simulate this, we initialize the sampler by perturbing the true parameter values with standard Gaussian noise.
For each trial, the LMC sampler is run for 30,000 iterations using a step size from the range $[0.1, 0.5]$ depending on the model. We conservatively discard the first 20,000 burn-in iterations and report the bulk ESS per coordinate and provide a histogram of the bulk ESS for the log-likelihood ratio (LLR) summarized from the last 10,000 iterations from each of the 20 trials.

The second simulation explores the classical asymptotic regime, with dimension $d=10$ and a large sample size of $n=1000$. With the sample size substantially exceeding the dimension, we evaluate the frequentist coverage properties of the resulting $95\%$ credible sets. We expect the empirical coverage to closely match the nominal $95\%$ level, consistent with the well-established theoretical properties of Bayesian methods in such large-sample settings. 

For the logistic regression and Poisson linear models, the posteriors are log-concave, ensuring a unique global mode. We first locate this mode using the L-BFGS-B algorithm to obtain a warm start. From this starting point, we run the LMC for a total of 30,000 iterations. We employ a conservative burn-in period, discarding the initial 20,000 iterations and retaining the final 10,000 samples for the coverage analysis. For the Gaussian mixture model, which exhibits a multi-modal posterior, we first identify a high-density mode using the dual annealing algorithm. This stochastic global optimization method provides a warm start within the posterior contraction region. From this optimized starting point, we run LMC for a total of 30,000 iterations. We again apply a burn-in of 20,000 iterations and retain the final 10,000 samples for the coverage analysis. The step size of LMC ranges from $0.001$ to $0.01$.

\begin{figure}[htbp]
    \centering \includegraphics[width=\textwidth]{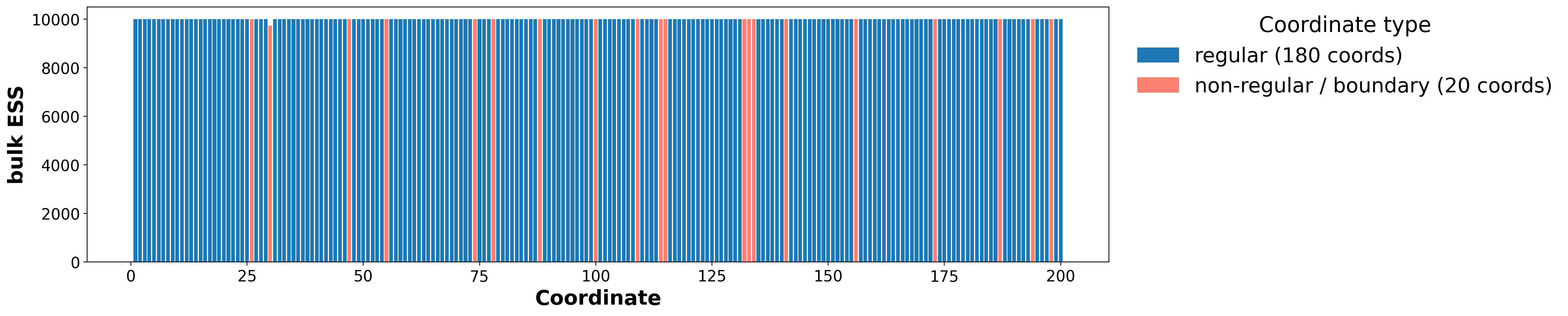}
    
    \caption{Effective sample size for each coordinate in logistic regression model, reported for one out of 20 trials with 10000 MCMC steps and step size 0.5. }
    \label{fig:ess_coord_logistic}
\end{figure}

\begin{figure}[htbp]
    \centering
    \includegraphics[width=\textwidth]{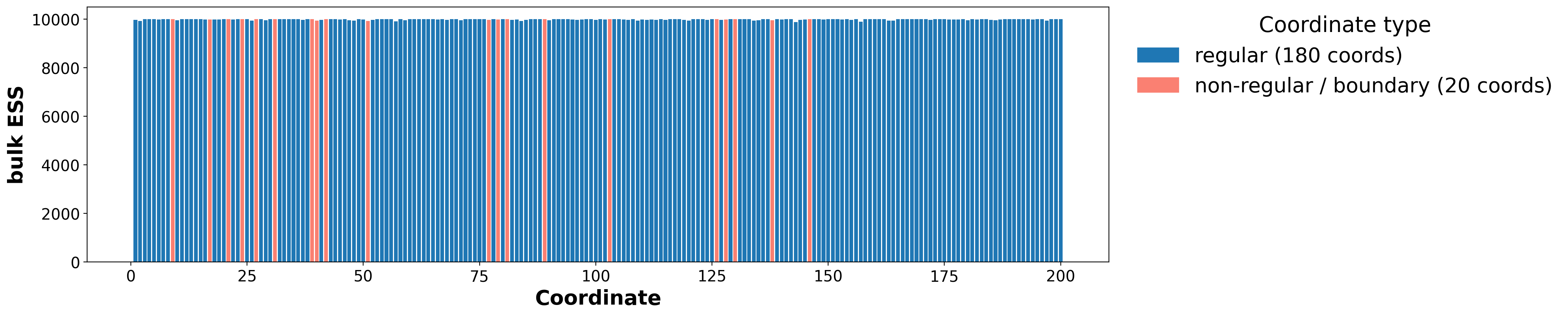}
    \caption{Effective sample size for each coordinate in Poisson linear model, reported for one out of 20 trials with 10000 MCMC steps and step size 0.1. }
    \label{fig:ess_coord_poisson}
\end{figure}

\begin{figure}[htbp]
    \centering \includegraphics[width=\textwidth]{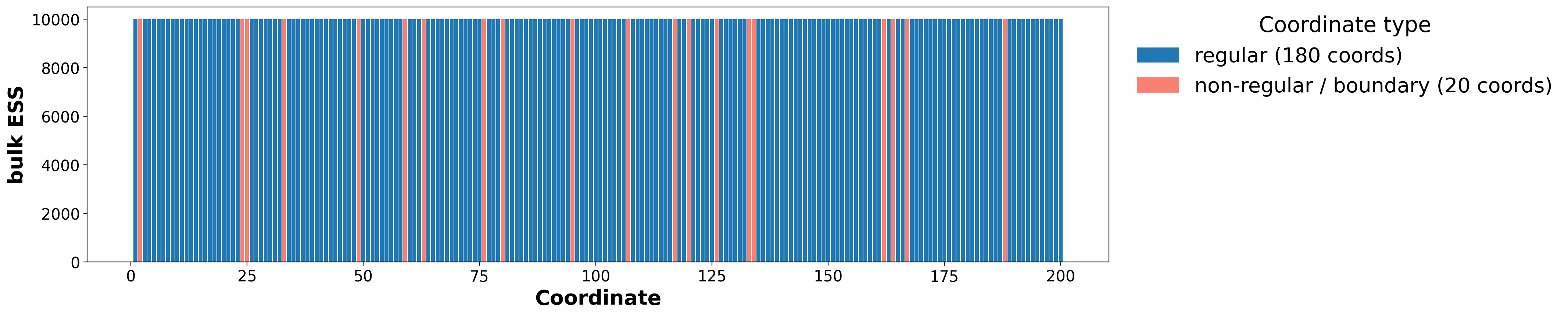}
    \caption{Effective sample size for each coordinate in Gaussian mixture model, reported for one out of 20 trials with 10000 MCMC steps and step size 0.1. }
    \label{fig:ess_coord_gaussian}
\end{figure}

\begin{figure}[htbp]
  \centering
    \begin{subfigure}{0.33\textwidth}
    \includegraphics[width=\linewidth]{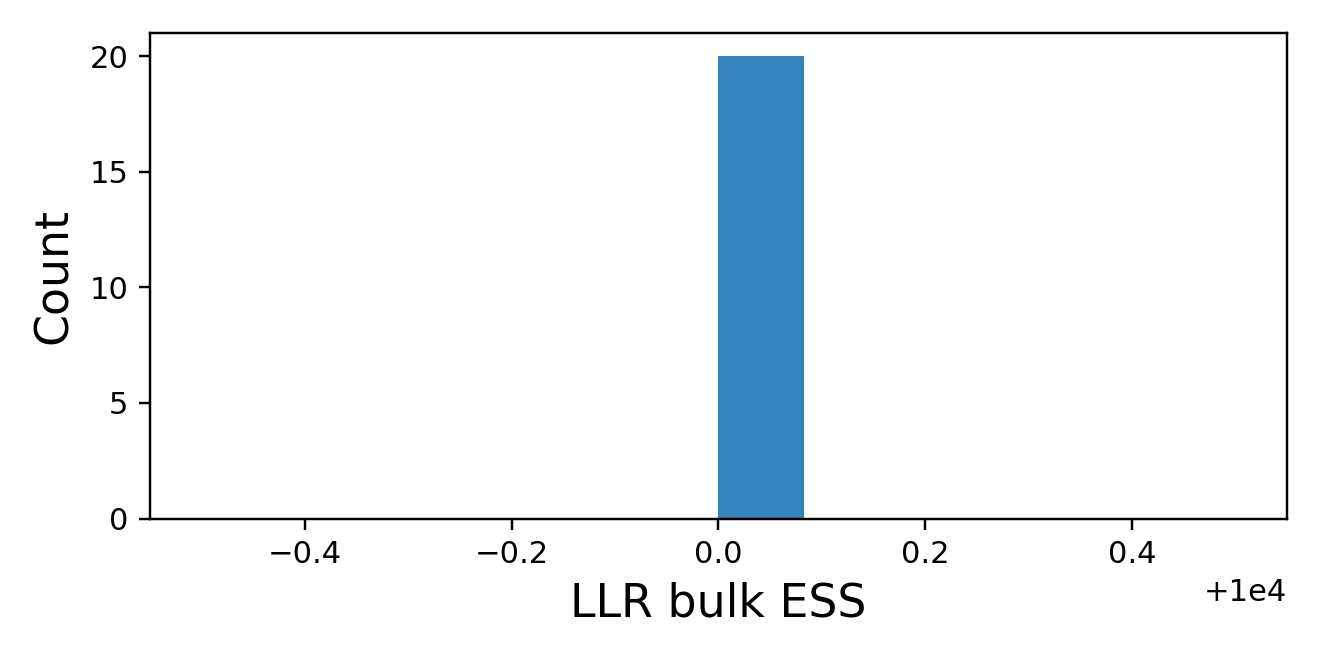}
    \caption{Logistic}\label{fig:ess_llr_logistic}
  \end{subfigure}\hfill
  \begin{subfigure}{0.33\textwidth}
    \includegraphics[width=\linewidth]{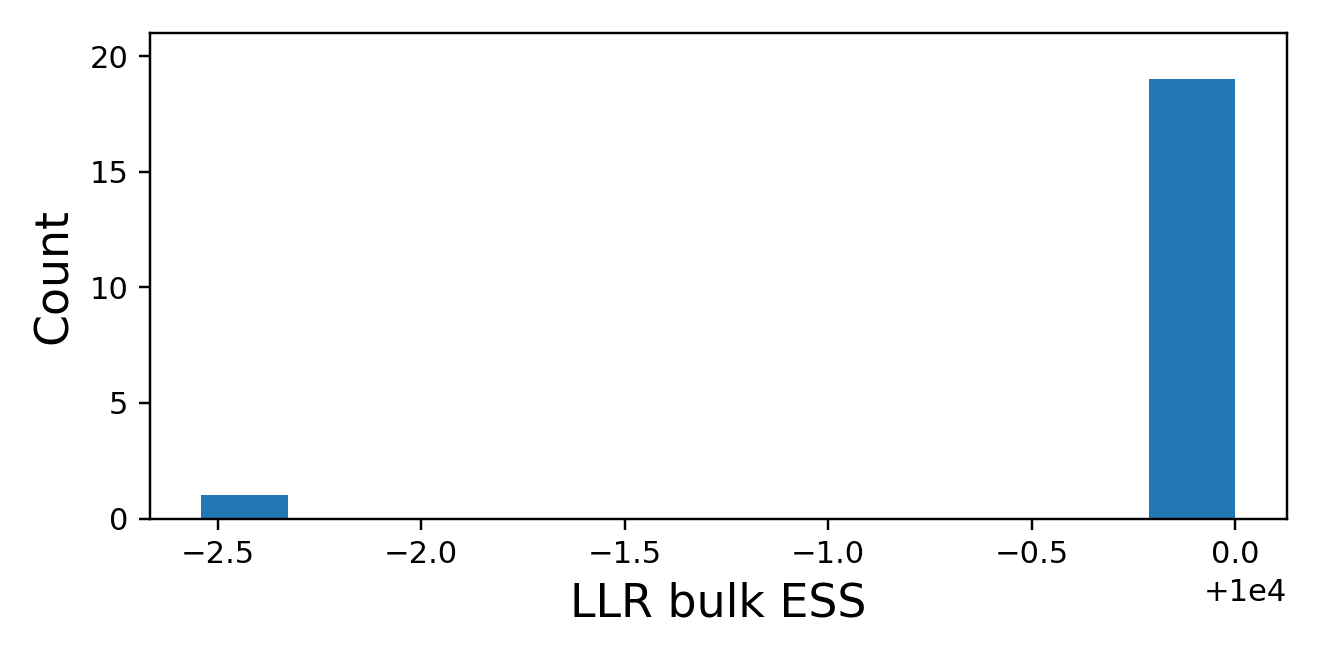}
    \caption{Poisson}\label{fig:ess_llr_poisson}
  \end{subfigure}\hfill
  \begin{subfigure}{0.33\textwidth}
    \includegraphics[width=\linewidth]{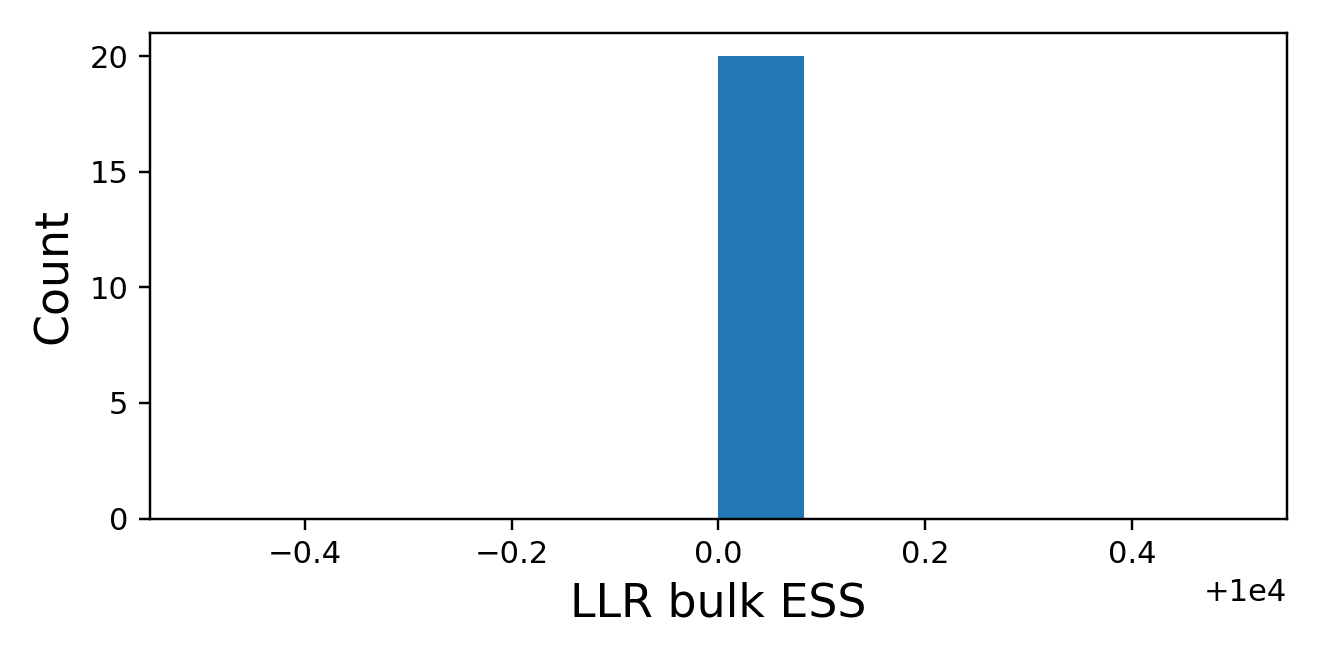}
    \caption{Gaussian mixture}\label{fig:ess_llr_gaussian}
  \end{subfigure}
  \caption{LLR bulk ESS across 20 trials for three models.}
  \label{fig:ess_llr_row}
\end{figure}

\begin{figure}[htbp]
    \centering
    \includegraphics[width=\textwidth]{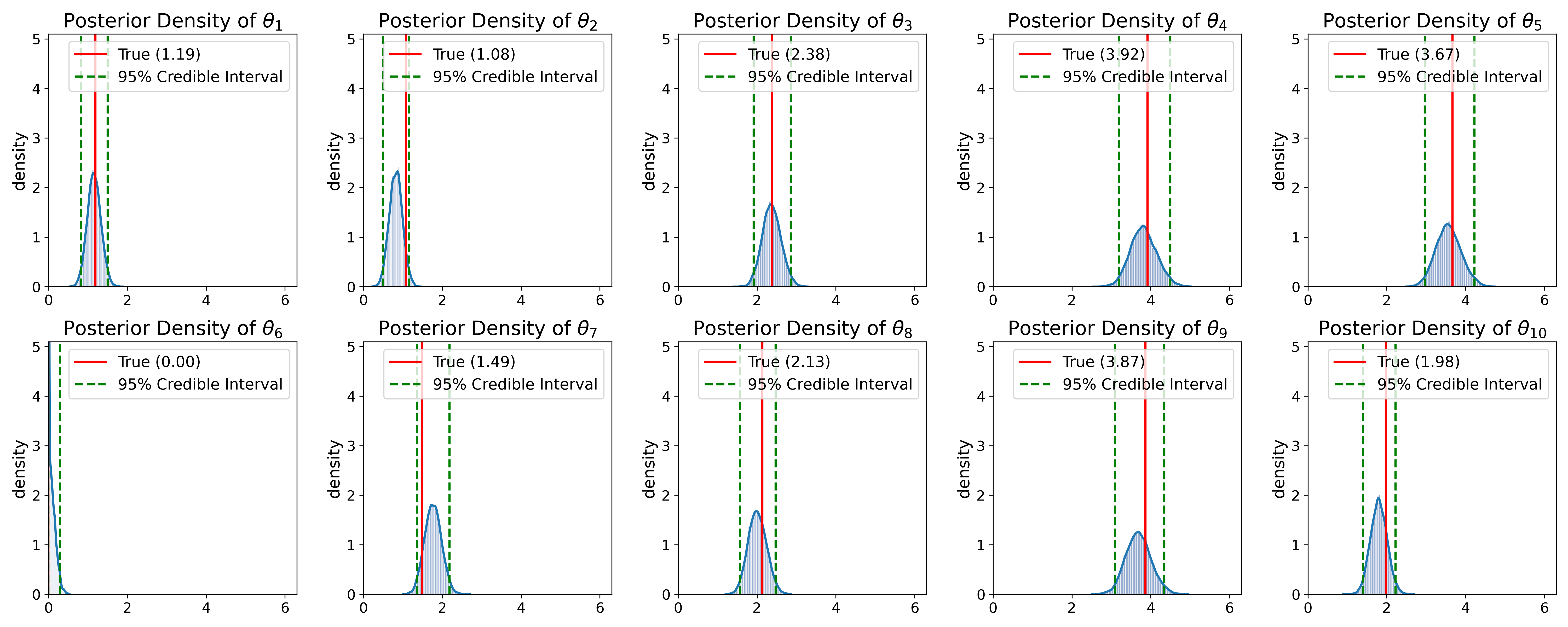}
    \caption{Empirical posterior density of one of the 20 MCMC runs for logistic regression models. The non-regular coordinate $\theta_6$ has a much narrower high-probability region compared to regular coordinates. }
    \label{fig:posterior_lr}
\end{figure}

\begin{figure}[htbp]
    \centering
    \includegraphics[width=.7\textwidth]{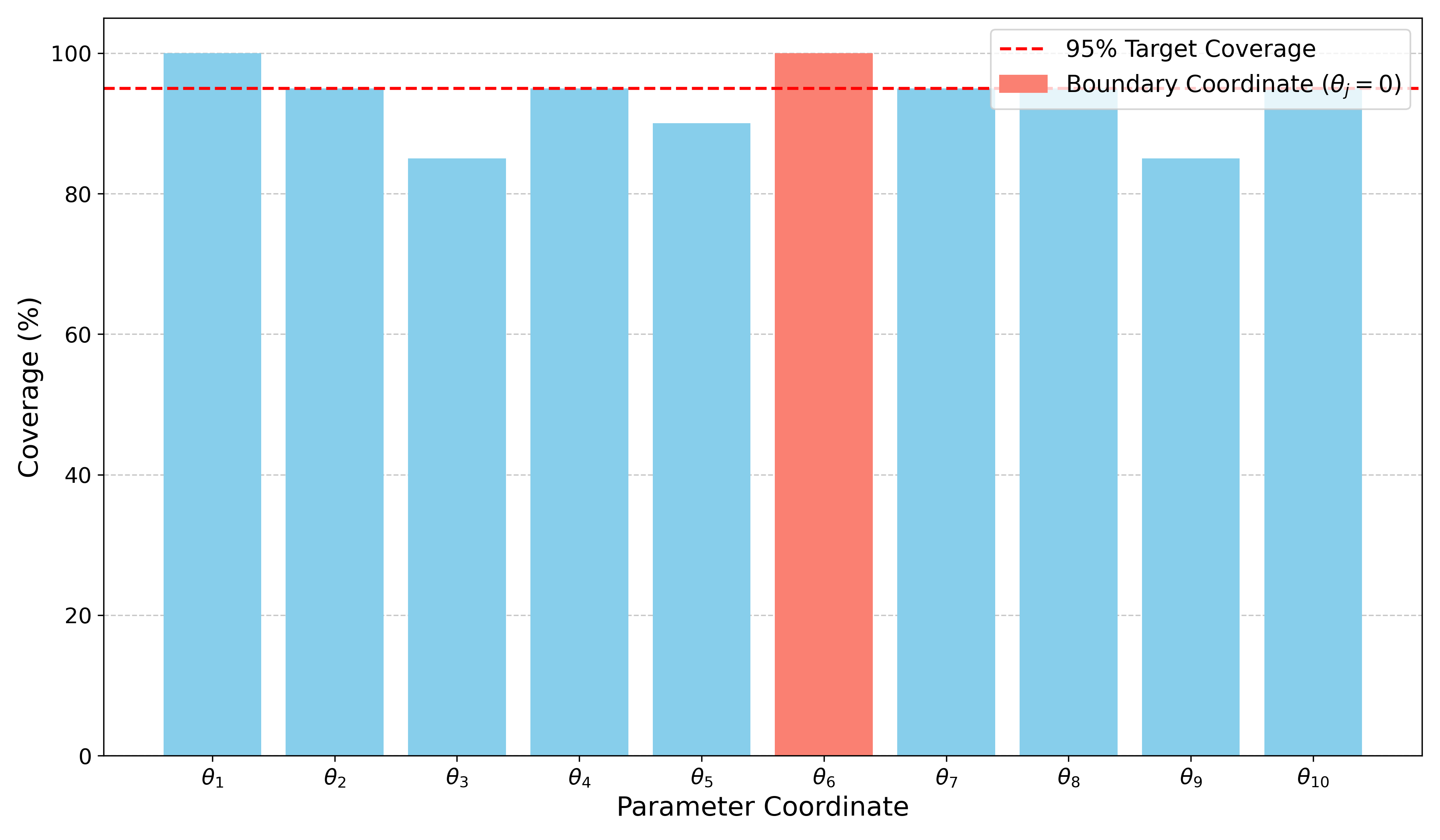}
    \caption{Summary of 95\% credible interval coverage for each parameter of the logistic regression model across 20 trials. The salmon-colored bar indicates the boundary coordinate where the true value is $\theta_6 = 0$.}
    \label{fig:coverage_lr}
\end{figure}

\begin{figure}[htbp]
    \centering
    \includegraphics[width=\textwidth]{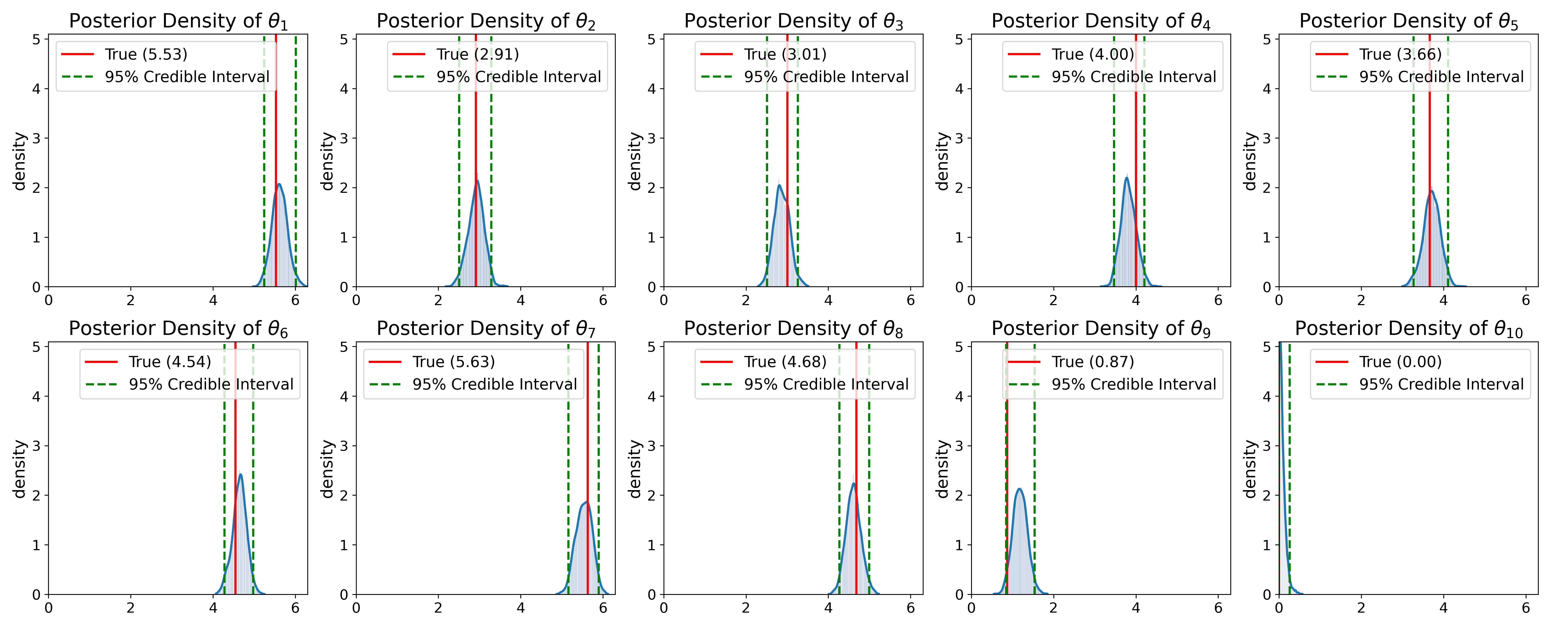}
    \caption{Empirical posterior density of one of the 20 MCMC runs for Poisson linear model. The non-regular coordinate is $\theta_{10}$.}
    \label{fig:posterior_poisson}
\end{figure}  

\begin{figure}[htbp]
    \centering
    \includegraphics[width=.7\textwidth]{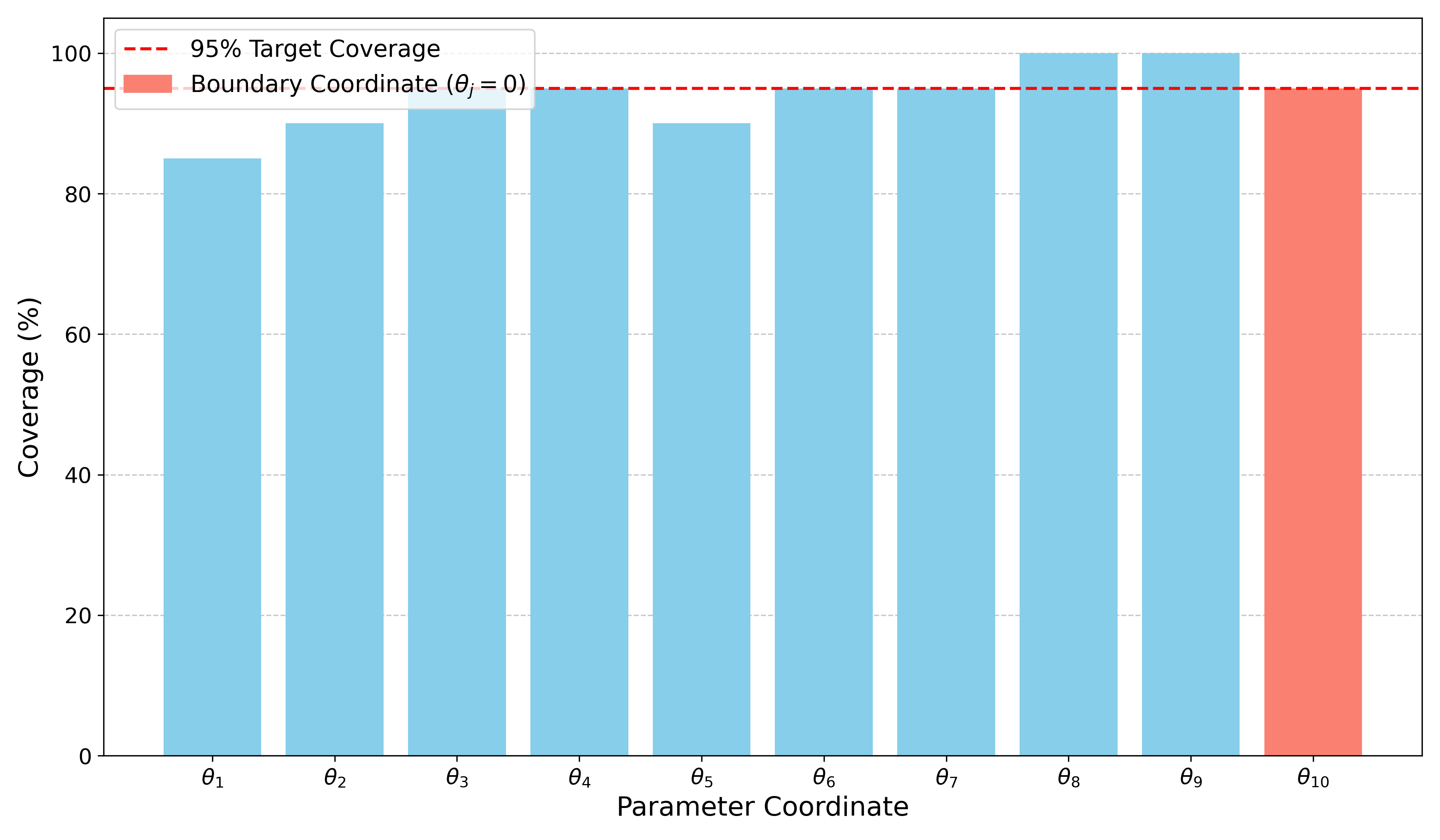}
    \caption{Summary of 95\% credible interval coverage for each parameter of the Poisson linear model across 20 trials. The salmon-colored bar indicates the boundary coordinate where the true value is $\theta_{10} = 0$.}
    \label{fig:coverage_poisson}
\end{figure}

\begin{figure}[htbp]
    \centering
    \includegraphics[width=\textwidth]{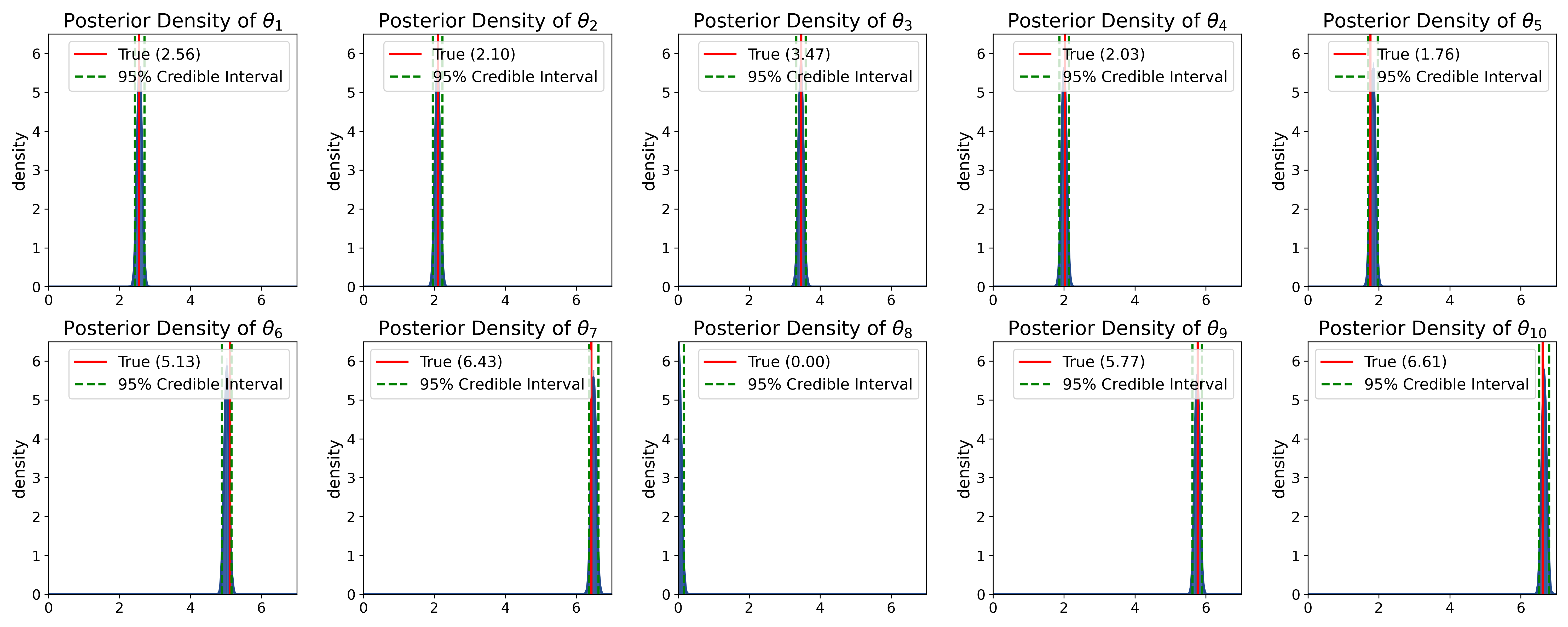}
    \caption{Empirical posterior density of one of the 20 MCMC runs for Gaussian mixture model with two modes with weights 0.7 and 0.3 respectively. The non-regular coordinate is $\theta_8$.}
    \label{fig:posterior_gmm}
\end{figure}  

\begin{figure}[htbp]
    \centering

\includegraphics[width=.7\textwidth]{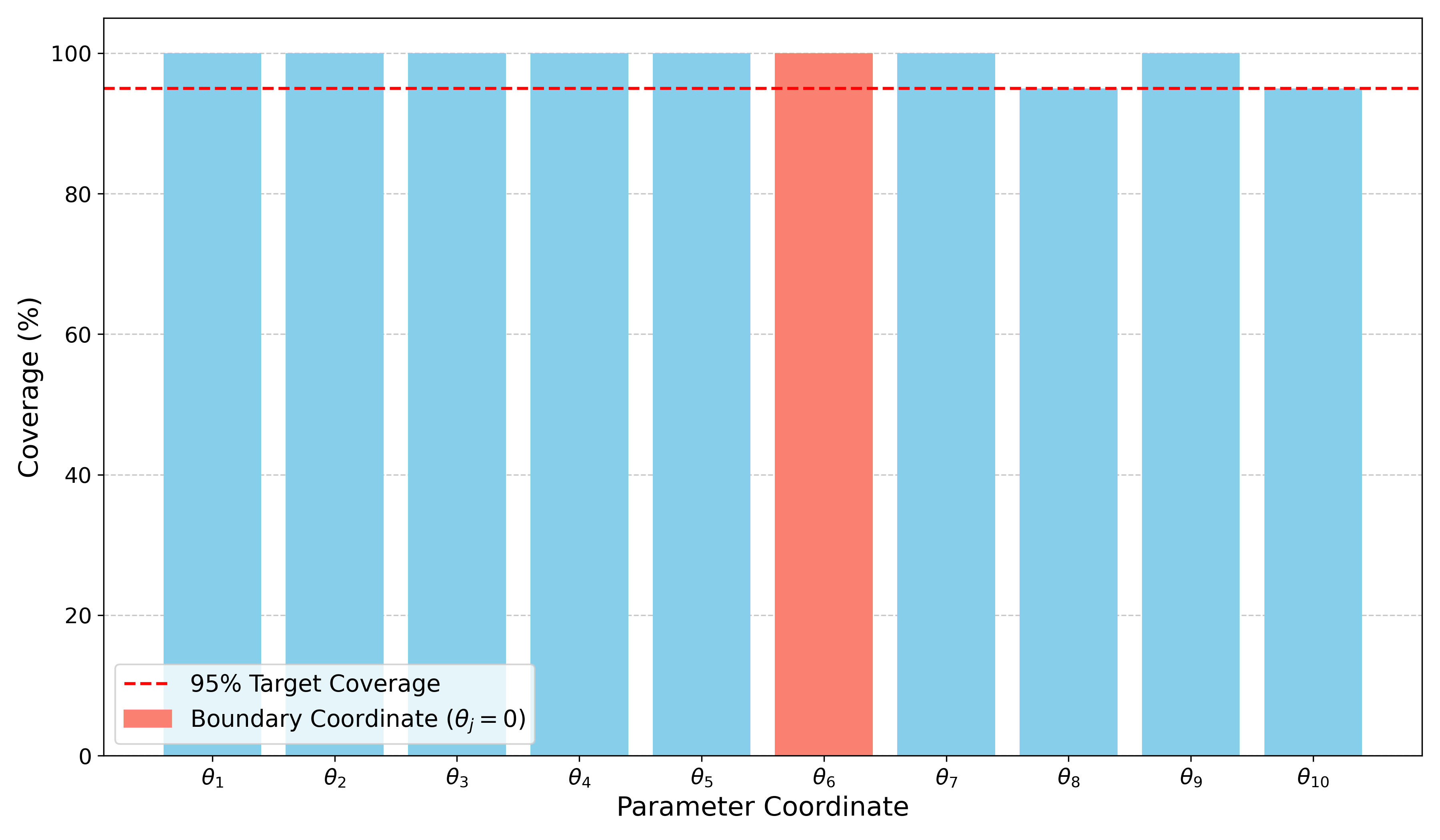}
    \caption{Summary of 95\% credible interval coverage for each parameter of the gaussian mixture model across 20 trials. The salmon-colored bar indicates the boundary coordinate where the true value is $\theta_8 = 0$.}
    \label{fig:coverage_gmm}
\end{figure}

\newpage
\section{Conclusion}\label{sec:conclusion}

In this work, we analyzed the problem of sampling from high-dimensional low-temperature Gibbs distributions on constrained sets, with a specific focus on non-regular models where the mode lies on the boundary of the parameter space. Our primary contribution is to establish a non-asymptotic sampling guarantee by deriving a Poincar\'e inequality for the target distribution restricted to a high-probability ``good set'' with a dimension-free Poincar\'e constant in the pre-asymptotic regime. We also provided a general recipe for verifying concentration of the target distribution on such good sets, which may be of independent interest for analyzing posterior contraction in non-regular models.
We applied this framework to Bayesian inference, demonstrating sampling guarantees for high-dimensional logistic regression, Poisson linear models, and Gaussian mixture models, all in regimes where the sample size $n$ grows with the dimension $d$ with a rate that is much milder than asymptotic theory would suggest. 

Our analysis relies on a local likelihood decomposition, and consequently our sampling guarantees are valid only when the algorithm is initialized with a warm start within the good set surrounding the mode. In non-log-concave settings, identifying this region in high-dimensional non-regular landscapes is non-trivial and remains an active area of research.

Finally, the tools we developed for handling boundary constraints offer a promising foundation for analyzing partially identified models. As noted in the introduction, partially identified models frequently arise in overparametrized neural networks and auction models. These settings often violate standard regularity conditions and have flat regions. We anticipate that our techniques for analyzing spectral gaps of constrained measures can be adapted to provide non-asymptotic sampling guarantees for the distributions arising in these models.

\bibliographystyle{unsrtnat}
\bibliography{refs}

\appendix
\section{Sufficient conditions for concentration on the good set}\label{app:concentration}

In this section, we provide the proofs for the sufficient conditions for concentration on the good set given in~\cref{sec:laplace_sufficient}.

We will repeatedly use the likelihood decomposition in~\cref{thm:likelihood_decomposition}, so we define $\error$ to be the upper bound on $\osc B$ therein:
\begin{align*}
    \error \deq 2s_2 \,\Bigl(\frac{\delta_0 \delta_1 \,(d_0 d_1)^{1/2}}{n^{3/2}} + \frac{\delta_1^2 d_1}{n^2}\Bigr)\,.
\end{align*}

We start by establishing the following growth bound. 

\begin{lemma}[Log-concave measures satisfy linear growth]\label{lem:log_concave_growth}
    Adopt~\cref{integrand} through~\cref{assumption:bounded_mixed_derivatives}.
    Additionally, assume that $\ell$ is concave and that $n \gg d_1/d_0$ (see~\eqref{eq:n_cond_1} below).
    Then, for all $\theta \in \Theta \setminus \hat\Theta_\delta$,
    \begin{align*}
        \ell(\hat\theta) - \ell(\theta)
        &\ge \Bigl(\frac{C_{S_0}\delta_0\sqrt{d_0}}{8\sqrt n} \wedge \frac{\sqrt{C_{S_0} C_{S_1} \delta_1}}{\sqrt{d_1 n}}\Bigr)\,\|\theta_{S_0} - \hat\theta_{S_0}\|_2 + \Bigl(1 - \frac{n\,\error}{C_{S_1} \delta_1} - \frac{1}{d_1}\Bigr)\,\|\theta_{S_1} - \hat\theta_{S_1}\|_{\nabla_{S_1} \ell(\hat\theta)}\,.
    \end{align*}
\end{lemma}
\begin{proof}
We begin with a lower bound on $\ell(\hat\theta) - \ell(\theta)$ for $\theta\in\hat\Theta_\delta$.
    By~\cref{thm:likelihood_decomposition},
    \begin{align*}
        \ell(\hat\theta) - \ell(\theta)
        &= B(\hat\theta) - B(\theta) + f(\hat\theta_{S_0}) - f(\theta_{S_0}) + \sum_{j\in S_1} [g_j(\hat\theta_j) - g_j(\theta_j)] \\
        &\ge \frac{C_{S_0}}{4}\,\|\theta_{S_0} - \hat\theta_{S_0}\|_2^2 + \|\theta_{S_1} - \hat\theta_{S_1}\|_{\nabla_{S_1} \ell(\hat\theta)} - \error\,,
    \end{align*}
    where we define the weighted norm
    \begin{align*}
        \|\theta_{S_1}\|_{\nabla_{S_1} \ell(\hat\theta)}
        &\deq \sum_{j\in S_1} (-\partial_j\ell(\hat\theta))\,|\theta_j|\,.
    \end{align*}
    Now, let $\theta \in \Theta\setminus\hat\Theta_\delta$, and let $t \in [0,1]$ be the largest $t$ such that $\theta(t) \deq (1-t)\,\hat\theta + t\,\theta \in \hat\Theta_\delta$.
    Since $\|\theta(t)_{S_0} - \hat\theta_{S_0}\|_2 = t\,\|\theta_{S_0} -\hat\theta_{S_0}\|_2 \le \delta_0\sqrt{d_0/n}$ and $\|\theta(t)_{S_1} - \hat\theta_{S_1}\|_\infty = t\,\|\theta_{S_1} - \hat\theta_{S_1}\|_\infty \le \delta_1/n$, this implies that $t = \min\{\delta_0\sqrt{d_0}/(\sqrt n\,\|\theta_{S_0} - \hat\theta_{S_0}\|_2), \delta_1/(n\,\|\theta_{S_1} - \hat\theta_{S_1}\|_\infty)\}$.
    So, by concavity of $\ell$,
    \begin{align*}
        \ell(\hat\theta) - \ell(\theta)
        &\ge \frac{\ell(\hat\theta) - \ell(\theta(t))}{\|\theta(t) - \hat\theta\|_2}\,\|\theta - \hat\theta\|_2
        \ge \frac{1}{t}\,\Bigl\{ \frac{C_{S_0}}{4}\,\|\theta(t)_{S_0} - \hat\theta_{S_0}\|_2^2 + \|\theta(t)_{S_1} - \hat\theta_{S_1}\|_{\nabla_{S_1}\ell(\hat\theta)} - \error\Bigr\} \\
        &= \frac{C_{S_0} t}{4}\,\|\theta_{S_0} - \hat\theta_{S_0}\|_2^2 + \|\theta_{S_1} - \hat\theta_{S_1}\|_{\nabla_{S_1} \ell(\hat\theta)} - \frac{\error}{t}\,.
    \end{align*}
    We now split into two cases. If $t = \delta_0\sqrt{d_0}/(\sqrt n\,\|\theta_{S_0} - \hat\theta_{S_0}\|_2)$, then
    \begin{align*}
        \frac{C_{S_0} t}{4}\,\|\theta_{S_0} - \hat\theta_{S_0}\|_2^2 - \frac{\error}{t}
        &= \Bigl(\frac{C_{S_0} \delta_0 \sqrt{d_0}}{4\sqrt n} - \frac{\error\sqrt n}{\delta_0\sqrt{d_0}}\Bigr)\,\|\theta_{S_0} - \hat\theta_{S_0}\|_2
        \ge \frac{C_{S_0} \delta_0 \sqrt{d_0}}{8\sqrt n}\,\|\theta_{S_0} - \hat\theta_{S_0}\|_2
    \end{align*}
    provided $\error \le C_{S_0} \delta_0^2 d_0/(8n)$.
    A sufficient condition for this to hold is
    \begin{align}\label{eq:n_cond_1}
        n \ge \bigl(\frac{32s_2}{C_{S_0}}\bigr)^2\, \frac{\delta_1^2}{\delta_0^2}\,\frac{d_1}{d_0}\,.
    \end{align}
    On the other hand, if $t = \delta_1/(n\,\|\theta_{S_1} - \hat\theta_{S_1}\|_\infty)$, then
    \begin{align*}
        \|\theta_{S_1} - \hat\theta_{S_1}\|_{\nabla_{S_1} \ell(\hat\theta)} - \frac{\error}{t}
        &\ge \Bigl(1 - \frac{n\,\error}{C_{S_1} \delta_1}\Bigr)\,\|\theta_{S_1} - \hat\theta_{S_1}\|_{\nabla_{S_1}\ell(\hat\theta)}\,.
    \end{align*}
    Furthermore,
    \begin{align*}
        \frac{C_{S_0} t}{4}\,\|\theta_{S_0} - \hat\theta_{S_0}\|_2^2 + \frac{1}{d_1}\,\|\theta_{S_1} - \hat\theta_{S_1}\|_{\nabla_{S_1} \ell(\hat\theta)}
        &\ge \frac{C_{S_0} \delta_1\,\|\theta_{S_0} - \hat\theta_{S_0}\|_2^2}{4n\,\|\theta_{S_1} - \hat\theta_{S_1}\|_1} + \frac{C_{S_1}}{d_1}\,\|\theta_{S_1} - \hat\theta_{S_1}\|_1 \\
        &\ge \frac{\sqrt{C_{S_0} C_{S_1} \delta_1}}{\sqrt{d_1 n}}\,\|\theta_{S_0} -\hat\theta_{S_0}\|_2\,.
    \end{align*}
    Combining both cases, we see that, provided~\eqref{eq:n_cond_1} holds,
    \begin{align*}
        \ell(\hat\theta) - \ell(\theta)
        &\ge \Bigl(\frac{C_{S_0}\delta_0\sqrt{d_0}}{8\sqrt n} \wedge \frac{\sqrt{C_{S_0} C_{S_1} \delta_1}}{\sqrt{d_1 n}}\Bigr)\,\|\theta_{S_0} - \hat\theta_{S_0}\|_2 + \Bigl(1 - \frac{n\,\error}{C_{S_1} \delta_1} - \frac{1}{d_1}\Bigr)\,\|\theta_{S_1} - \hat\theta_{S_1}\|_{\nabla_{S_1} \ell(\hat\theta)}\,.
    \end{align*}
    This concludes the proof.
\end{proof}

Next, let
\[
  \tilde\ell(\theta)  \deq  \frac{1}{n}\log \pi(\theta) + \ell(\theta)
\]
denote the prior-adjusted log-likelihood.
After taking the prior distribution into account, by \cref{assumption:Gradient bound on the prior}, we have the following lemma.

\begin{lemma}[Linear growth outside the good set with a prior]\label{lemma:linear_growth_with_prior}
    In the setting of~\cref{lem:log_concave_growth}, if we assume that $L_\pi \le \frac{n}{2\sqrt{d_0}}\,(\frac{C_{S_0}\delta_0\sqrt{d_0}}{8\sqrt n} \wedge \frac{\sqrt{C_{S_0} C_{S_1} \delta_1}}{\sqrt{d_1 n}})$, then for all $\theta \in \Theta \setminus \hat\Theta_\delta$,
    \begin{align*}
        \tilde\ell(\hat\theta) - \tilde\ell(\theta)
        &\ge \underbrace{\Bigl(\frac{C_{S_0}\delta_0\sqrt{d_0}}{16\sqrt n} \wedge \frac{\sqrt{C_{S_0} C_{S_1} \delta_1}}{2\sqrt{d_1 n}}\Bigr)}_{\eqqcolon \Cgr_0}\,\|\theta_{S_0} - \hat\theta_{S_0}\|_2 \\
        &\qquad{} + \underbrace{\Bigl(1 - \frac{n\,\error}{C_{S_1} \delta_1} - \frac{1}{d_1} - \frac{L_\pi}{C_{S_1} n}\Bigr)}_{\eqqcolon \Cgr_1}\,\|\theta_{S_1} - \hat\theta_{S_1}\|_{\nabla_{S_1} \ell(\hat\theta)}\,.
    \end{align*}
  \end{lemma}
  \begin{proof}
   Since the prior $\pi$ has a different mode than $\hat\theta$, we use~\cref{assumption:Gradient bound on the prior}:
   \begin{align*}
       \log \pi(\hat\theta) - \log\pi(\theta)
       &= \sum_{j=1}^d [\log\pi_j(\hat\theta_j) - \log \pi_j(\theta_j)]
       \ge -L_\pi\,\|\theta - \hat\theta\|_1 \\
       &\ge -L_\pi\sqrt{d_0}\,\|\theta_{S_0} - \hat\theta_{S_0}\|_2 - L_\pi\,\|\theta_{S_1} - \hat\theta_{S_1}\|_1\,.
   \end{align*}
   Divide by $n$ and add this inequality to the conclusion of~\cref{lem:log_concave_growth}.
\end{proof}

Before proving Theorem~\ref{thm:concentration_log_concave}, we state a lemma that bounds the expected squared distance between $\theta_{S_0}$ and $\hat\theta_{S_0}$ in a $\Theta(1)$ neighborhood of $\hat\theta$.

\begin{lemma}[Integration by parts]\label{lemma:term III}
Suppose that~\cref{integrand} and~\cref{assumption:Gradient bound on the prior} hold.
Furthermore, assume that the following conditions hold:
\begin{itemize} 
    \item For some constant $r_0 > 0$ and $ 0 < r_1 \leq \frac{C_{S_0}C_{S_1}}{2s_2^2}$, $\grad_{S_0}^2 \ell(\theta) \preceq -C_{S_0} I$ for all $\theta \in B(\hat\theta, r_0, 0)$, and $\partial_j\ell(\theta) < -C_{S_1}$ for all $j \in S_1$ and $\theta \in B(\hat\theta, 0, r_1)$.
    \item  $\sup_{\theta \in B(\hat\theta, r_0, r_1)} \norm{\grad^2 \ell(\theta)}_{\rm op} \leq s_2$.
\end{itemize}
Let $B \deq B(\hat\theta, r_0, r_1) \deq \{\theta \in \Theta : \norm{\theta_{S_0}-\hat\theta_{S_0}}_2 \leq r_0,\, \norm{\theta_{S_1} - \hat\theta_{S_1}}_\infty \leq r_1\}$. Let $\nu \deq \mu|_B$ be the measure $\mu$ conditioned on $B$. 
Then, we have the bound
\begin{align*}
    \Ex_\nu[\|\theta_{S_0} - \hat\theta_{S_0}\|_2^2] \leq \frac{2\,(d + \sqrt{d} L_\pi\, (r_0 + \sqrt{d_1} r_1))}{nC_{S_0}}\,.
\end{align*}
\end{lemma}
\begin{proof}
We start with integration by parts. Let $h(\theta) \deq \frac{1}{2}\,\norm{\theta - \hat\theta}_2^2$, so that $\nabla h(\theta) = \theta - \hat\theta$ and $\Delta h(\theta) = d$. The divergence theorem gives
\begin{align*}
\int_B (\Delta h + n\,\langle \nabla \tilde\ell , \nabla h \rangle)\, d\nu = \int_B \div{(\nu \grad h)} = \int_{\partial B} \langle \nabla h, \mathbf{n} \rangle\,d\nu\,,
\end{align*}
where $\mathbf{n}$ is the outward-pointing unit normal to the boundary $\partial B$.
We analyze the boundary term  $\int_{\partial B} \langle \theta - \hat\theta, \mathbf{n}(\theta) \rangle\, d\nu(\theta)$. The boundary $\partial B$ consists of three parts:
\begin{enumerate}
    \item[1.] The regular boundary $\{\theta \in B : \norm{\theta_{S_0}-\hat\theta_{S_0}}_2 = r_0\}$. Here, $\mathbf{n}= (\theta_{S_0} - \hat\theta_{S_0})/r_0$, so $\langle \theta - \hat\theta, \mathbf{n} \rangle = \langle \theta_{S_0} - \hat\theta_{S_0}, \mathbf{n} \rangle \ge 0$.
    \item[2.] The outer non-regular boundary $\{\theta \in B : \theta_j = r_1 \text{ for some } j \in S_1\}$. Here, $\mathbf{n} = e_j$. The integrand is $\langle \theta - \hat\theta, e_j \rangle = \theta_j - \hat\theta_j = r_1 - 0 = r_1 \ge 0$.
    \item[3.] The inner non-regular boundary $\{\theta \in B : \theta_j = 0 \text{ for some } j \in S_1\}$. Here, $\mathbf{n} = -e_j$. The integrand is $\langle \theta - \hat\theta, -e_j \rangle = -(\theta_j - \hat\theta_j) = 0$.
\end{enumerate}
Since the integrand is non-negative on all parts of the boundary $\partial B$, we have 
\begin{align*}
d + n\, \mathbb{E}_\nu[\langle \nabla \ell(\theta), \theta - \hat\theta \rangle] + \mathbb{E}_\nu[\langle \nabla \log \pi(\theta), \theta - \hat\theta \rangle] \ge 0\,.
\end{align*}
Since $\log \pi$ is $L_\pi$-Lipschitz, and $\norm{\theta - \hat\theta}_2 \leq r_0 + \sqrt{d_1} r_1$ for all $\theta \in B$, we have
\begin{align*}
    \mathbb{E}_\nu[\langle \nabla \log \pi(\theta), \theta - \hat\theta \rangle] &\leq \sqrt{d} L_\pi\, (r_0 + \sqrt{d_1} r_1)\,.
\end{align*}
Therefore, 
\begin{align*}
    d &\geq -n\, \mathbb{E}_\nu[\langle \nabla \ell(\theta), \theta - \hat\theta \rangle] - \sqrt{d} L_\pi\, (r_0 + \sqrt{d_1} r_1) \\ 
    &\geq -n \E_\nu[\langle\grad_{S_0} \ell(\theta), \theta_{S_0} - \hat\theta_{S_0}\rangle + \langle \nabla_{S_1} \ell(\theta), \theta_{S_1} - \hat\theta_{S_1} \rangle] - \sqrt{d} L_\pi\, (r_0 + \sqrt{d_1} r_1)\,.
\end{align*}

For the $S_0$ term:
\begin{align*}
\langle \nabla_{S_0} \ell(\theta), \theta_{S_0} - \hat\theta_{S_0} \rangle &= \langle \nabla_{S_0} \ell(\theta_{S_0}, \hat\theta_{S_1}) + (\nabla_{S_0} \ell(\theta) - \nabla_{S_0} \ell(\theta_{S_0}, \hat\theta_{S_1})), \theta_{S_0} - \hat\theta_{S_0} \rangle \\
&= \langle \nabla_{S_0} \ell(\theta_{S_0}, \hat\theta_{S_1}), \theta_{S_0} - \hat\theta_{S_0} \rangle + \Bigl\langle \int_0^1 \nabla^2_{S_0, S_1} \ell(\theta_{S_0}, t\theta_{S_1})\, \theta_{S_1}\, dt, \theta_{S_0} - \hat\theta_{S_0} \Bigr\rangle \\
&\le -C_{S_0}\, \norm{\theta_{S_0} - \hat\theta_{S_0}}_2^2 + \sup_{\theta \in B}{\norm{\nabla^2_{S_0, S_1} \ell(\theta)}_{\rm op}}\, \norm{\theta_{S_1} - \hat\theta_{S_1}}_2\, \norm{\theta_{S_0} - \hat\theta_{S_0}}_2 \\
&\le -C_{S_0}\, \norm{\theta_{S_0} - \hat\theta_{S_0}}_2^2 + s_2\, \norm{\theta_{S_0} - \hat\theta_{S_0}}_2\, \norm{\theta_{S_1} - \hat\theta_{S_1}}_2\,.
\end{align*}
For $S_1$ term, we have the upper bound
\begin{align*}
\langle \nabla_{S_1} \ell(\theta), \theta_{S_1} - \hat\theta_{S_1} \rangle &= \langle \nabla_{S_1} \ell(\theta)- \nabla_{S_1} \ell(\hat\theta_{S_0}, \theta_{S_1}) + \nabla_{S_1} \ell(\hat\theta_{S_0}, \theta_{S_1}), \theta_{S_1} - \hat\theta_{S_1} \rangle \\
&\le \langle \nabla_{S_1} \ell(\theta)- \nabla_{S_1} \ell(\hat\theta_{S_0}, \theta_{S_1}), \theta_{S_1} - \hat\theta_{S_1} \rangle + \langle \nabla_{S_1} \ell(\hat\theta_{S_0}, \theta_{S_1}), \theta_{S_1} - \hat\theta_{S_1} \rangle \\ 
&\le s_2\, \norm{\theta_{S_0} - \hat\theta_{S_0}}_2\, \norm{\theta_{S_1} - \hat\theta_{S_1}}_2 - C_{S_1}\, \norm{\theta_{S_1} - \hat\theta_{S_1}}_1\,.
\end{align*}
Combining these bounds gives
\begin{align*}
\langle \nabla \ell(\theta), \theta - \hat\theta \rangle \le  -C_{S_0}\, \norm{\theta_{S_0} - \hat\theta_{S_0}}_2^2 + 2s_2\, \norm{\theta_{S_0} - \hat\theta_{S_0}}_2\, \norm{\theta_{S_1} - \hat\theta_{S_1}}_2 - C_{S_1}\, \norm{\theta_{S_1} - \hat\theta_{S_1}}_1\,.
\end{align*}
We apply Young's inequality to the cross term,
\begin{align*} 
    s_2\, \norm{\theta_{S_0} - \hat\theta_{S_0}}_2\, \norm{\theta_{S_1} - \hat\theta_{S_1}}_2 &\le s_2\, \bigl(\frac{C_{S_0}}{4s_2}\, \norm{\theta_{S_0} - \hat\theta_{S_0}}_2^2 + \frac{s_2}{C_{S_0}}\, \norm{\theta_{S_1} - \hat\theta_{S_1}}_2^2\bigr) \\ 
    &= \frac{C_{S_0}}{4}\, \norm{\theta_{S_0} - \hat\theta_{S_0}}_2^2 + \frac{s_2^2}{C_{S_0}}\, \norm{\theta_{S_1} - \hat\theta_{S_1}}_2^2\,.
\end{align*}
Thus,
\begin{align*}
C_{S_0}\, \mathbb{E}_\nu[\norm{\theta_{S_0} - \hat\theta_{S_0}}_2^2] &\le \frac{d + \sqrt{d} L_\pi\, (r_0 + \sqrt{d_1} r_1)}{n} \\ 
&\qquad{} + \mathbb{E}_\nu\bigl[ \frac{C_{S_0}}{2}\, \norm{\theta_{S_0} - \hat\theta_{S_0}}_2^2 + \frac{2s_2^2}{C_{S_0}}\, \norm{\theta_{S_1} - \hat\theta_{S_1}}_2^2 - C_{S_1}\, \norm{\theta_{S_1} - \hat\theta_{S_1}}_1 \bigr]\,, \\
\frac{C_{S_0}}{2}\, \mathbb{E}_\nu[\norm{\theta_{S_0} - \hat\theta_{S_0}}_2^2] &\le \frac{d + \sqrt{d} L_\pi\, (r_0 + \sqrt{d_1} r_1)}{n} + \mathbb{E}_\nu\bigl[ \frac{2s_2^2}{C_{S_0}}\, \norm{\theta_{S_1} - \hat\theta_{S_1}}_2^2 - C_{S_1}\, \norm{\theta_{S_1} - \hat\theta_{S_1}}_1 \bigr]\,.
\end{align*}
If $r_1 \le \frac{C_{S_0} C_{S_1}}{2 s_2^2}$, then 
\begin{align*}
\frac{2s_2^2}{C_{S_0}} \E_\nu[\norm{\theta_{S_1} - \hat\theta_{S_1}}_2^2]
&\le \frac{2s_2^2}{C_{S_0}}\E_\nu[\norm{\theta_{S_1}-\hat\theta_{S_1}}_1\,\norm{\theta_{S_1} - \hat\theta_{S_1}}_\infty]
\le \frac{2r_1 s_2^2}{C_{S_0}}\E_\nu[\norm{\theta_{S_1}-\hat\theta_{S_1}}_1] \\
&\le C_{S_1}\E_\nu[\norm{\theta_{S_1}-\hat\theta_{S_1}}_1]\,.
\end{align*}
Thus,
$\mathbb{E}_\nu[\norm{\theta_{S_0} - \hat\theta_{S_0}}_2^2] \le \frac{2\,(d + \sqrt{d} L_\pi\, (r_0 + \sqrt{d_1} r_1))}{nC_{S_0}}$.
\end{proof}

\subsection{Proof of Theorem~\ref{thm:concentration_log_concave}}\label{appendix:proof_lemma: log concave}

\begin{proof}[Proof of~\cref{thm:concentration_log_concave}]
Let us break $\mu(\Theta \setminus \hat\Theta_{\delta})$ into four parts:
\begin{align*}
    {\rm I} &= \mu(\norm{\theta_{S_0}-\hat\theta_{S_0}}_2 \geq r_0)\,, \\ 
    {\rm II} &= \mu(\norm{\theta_{S_1}-\hat\theta_{S_1}}_{\infty} \geq r_1)\,, \\
    {\rm III} &= \mu\bigl(\delta_0\sqrt{\frac{d_0}{n}} \leq \norm{\theta_{S_0}-\hat\theta_{S_0}}_2 \leq r_0,\, \norm{\theta_{S_1} - \hat\theta_{S_1}}_\infty \leq r_1\bigr)\,, \\ 
    {\rm IV} &= \mu\bigl(\norm{\theta_{S_0} - \hat\theta_{S_0}}_2 \leq \delta_0 \sqrt{\frac{d_0}{n}},\, \frac{\delta_1}{n} \leq \norm{\theta_{S_1}- \hat\theta_{S_1}}_\infty \leq r_1 \bigr)\,.
\end{align*}
The goal is to bound each of these four terms by $\epsilon/4$.

\paragraph{Term I\@.}
By assumption, $\mu(\norm{\theta_{S_0}-\hat\theta_{S_0}}_2 \leq r_0^\prime) \geq 2/3$. We apply Borell's inequality~\citep{Bor1974Cvx}:
$$
\mu(\norm{\theta_{S_0}-\hat\theta_{S_0}}_2 \geq t r_0^\prime) \leq \frac{2}{3}\times 2^{-\frac{t+1}{2}}\,.
$$
To bound the tail probability ${\rm I}$ beyond $r_0$, we set $t = r_0 / r_0'$. The bound becomes
$$
{\rm I} \le \frac{2}{3}\times 2^{-\frac{r_0/r_0' + 1}{2}}\,.
$$
We require ${\rm I} \le \epsilon/4$. This holds if $t = r_0/r_0'$ satisfies the condition
$$
t = \frac{r_0}{r_0'} \ge 2 \log_2\bigl(\frac{8}{3\epsilon}\bigr) - 1\,.
$$
In other words, this is a condition on $r_0'$: $r_0' \le r_0 / \bigl( 2 \log_2(\frac{8}{3\epsilon}) - 1  \bigr)$.

\paragraph{Term II\@.}
The region of integration is $\Theta_{S_0} \times (\Theta_{S_1} \setminus B_\infty(\hat\theta_{S_1}, r_1))$, where $B_\infty(\hat\theta_{S_1}, r_1)$ is the $\ell_\infty$-ball of radius $r_1$ centered at $\hat\theta_{S_1}$. We bound the ratio:
    \begin{align*}
         {\rm II} = \frac{\int_{\Theta_{S_0} \times (\Theta_{S_1} \setminus B_\infty(\hat\theta_{S_1}, r_1))} \exp\{n\,(\tilde\ell(\theta) - \tilde\ell(\hat\theta))\}\,d\theta}{\int_\Theta \exp\{n\,(\tilde\ell(\theta) - \tilde\ell(\hat\theta))\}\,d\theta} 
         \le \frac{\int_{\Theta_{S_0} \times (\Theta_{S_1} \setminus B_\infty(\hat\theta_{S_1}, r_1))} \exp\{n\,(\tilde\ell(\theta) - \tilde\ell(\hat\theta))\}\,d\theta}{\int_{\hat\Theta_\delta} \exp\{n\,(\tilde\ell(\theta) - \tilde\ell(\hat\theta))\}\,d\theta}\,.
    \end{align*}
    By~\cref{lemma:linear_growth_with_prior}, if we assume that $L_\pi \le \frac{n}{2\sqrt{d_0}}\,(\frac{C_{S_0}\delta_0\sqrt{d_0}}{8\sqrt n} \wedge \frac{\sqrt{C_{S_0} C_{S_1} \delta_1}}{\sqrt{d_1 n}})$, the numerator $\text{Num(II)}$ is bounded above by
    \begin{align*}
        \text{Num(II)} &\le \int_{\Theta_{S_0} \times (\Theta_{S_1} \setminus B_\infty(\hat\theta_{S_1}, r_1))} \exp\bigl\{-n\Cgr_0\,\|\theta_{S_0} - \hat\theta_{S_0}\|_2 - n\Cgr_1\,\|\theta_{S_1} - \hat\theta_{S_1}\|_{\nabla_{S_1}\ell(\hat\theta)}\bigr\}\,d\theta \\
        &= \int_{\Theta_{S_0}} \exp\bigl\{-n\Cgr_0\,\|\theta_{S_0} - \hat\theta_{S_0}\|_2\bigr\}\,d\theta_{S_0} \\
        &\qquad{} \times \int_{\Theta_{S_1} \setminus B_\infty(\hat\theta_{S_1}, r_1)} \exp\bigl\{-n\Cgr_1\,\|\theta_{S_1} - \hat\theta_{S_1}\|_{\nabla_{S_1}\ell(\hat\theta)}\bigr\}\,d\theta_{S_1}\,.
    \end{align*}
    We bound the two resulting integrals. First, for the regular part, we integrate over the entire space as a simple upper bound:
    \begin{align*}
         \int_{\Theta_{S_0}} \exp\bigl\{-n\Cgr_0\,\|\theta_{S_0} - \hat\theta_{S_0}\|_2\bigr\}\,d\theta_{S_0} 
         &\le \int_{\mathbb{R}^{d_0}} \exp\bigl\{-n\Cgr_0\,\|z\|_2\bigr\}\,dz \\
         &= d_0 \mathsf V_{d_0} \int_0^\infty \exp\{-n\Cgr_0 r\}\,r^{d_0-1}\,dr \\
         &= \frac{d_0 \mathsf V_{d_0}}{(n\Cgr_0)^{d_0}} \int_0^\infty \exp(-r)\,r^{d_0-1}\,dr
         = \frac{d_0 \mathsf V_{d_0}}{(n\Cgr_0)^{d_0}} \,(d_0-1)! \\
         &\le ed_0\mathsf V_{d_0}\,\bigl(\frac{d_0}{en\Cgr_0}\bigr)^{d_0}\,,
    \end{align*}
    where $\mathsf V_{d_0}$ is the volume of the $d_0$-dimensional unit ball.
    For the non-regular part, we use a union bound:
    \begin{align*}
        &\int_{\|\theta_{S_1}\|_\infty \ge r_1,\, \theta_{S_1} \ge 0} \exp\Bigl\{-n\Cgr_1 \sum_{k \in S_1} \abs{\partial_k \ell(\hat\theta)} \theta_k\Bigr\}\,d\theta_{S_1}
        \le \sum_{j \in S_1} \int_{\substack{\theta_{S_1} \ge 0 \\ \theta_j \ge r_1}} \prod_{k \in S_1} \exp\{-n\Cgr_1\, \abs{\partial_k \ell(\hat\theta)}\, \theta_k\}\,d\theta_k \\
        &\qquad = \sum_{j \in S_1} \Bigl( \int_{r_1}^\infty \exp\{-n\Cgr_1\, \abs{\partial_j \ell(\hat\theta)}\, \theta_j\}\,d\theta_j \Bigr)\, \Bigl( \prod_{k \in S_1,\,k \ne j} \int_0^\infty \exp\{-n\Cgr_1\, \abs{\partial_k \ell(\hat\theta)}\, \theta_k\}\,d\theta_k \Bigr) \\
        &\qquad = \sum_{j \in S_1} \frac{\exp\{-n\Cgr_1\, \abs{\partial_j\ell(\hat\theta)}\, r_1\}}{n\Cgr_1\, \abs{\partial_j \ell(\hat\theta)}} \prod_{k \in S_1,\,k\ne j} \frac{1}{n\Cgr_1\, \abs{\partial_k \ell(\hat\theta)}} \\
        &\qquad = \Bigl( \prod_{k \in S_1} \frac{1}{n\Cgr_1\, \abs{\partial_k \ell(\hat\theta)}} \Bigr) \sum_{j \in S_1} \exp\{-n\Cgr_1\, \abs{\partial_j \ell(\hat\theta)}\, r_1\}\,.
    \end{align*} 
    Combining these, the numerator is bounded by:
    \[
        \text{Num(II)} \le ed_0 \mathsf V_{d_0} \,\bigl(\frac{d_0}{en\Cgr_0}\bigr)^{d_0} \, \Bigl( \prod_{k \in S_1} \frac{1}{n\Cgr_1\, |\partial_k \ell(\hat\theta)|} \Bigr) \sum_{j \in S_1} \exp\{-n\Cgr_1\, \abs{\partial_j \ell(\hat\theta)}\, r_1\}\,.
    \]
    Let $Z$ denote a standard Gaussian variable in $\R^{d_0}$. The denominator
    \begin{align*}
        \text{Den(II)} \deq \int_{\hat\Theta_\delta} \exp\{n\,(\tilde\ell(\theta) - \tilde\ell(\hat\theta))\}\,d\theta
    \end{align*}
    is bounded as follows.
    By~\cref{thm:likelihood_decomposition}, the denominator is at least
    \begin{align*}
 &\int_{\hat\Theta_\delta} \exp\Bigl\{n\,\Bigl(f(\theta_{S_0}) - f(\hat\theta_{S_0}) + \sum_{j\in S_1} \partial_j\ell(\hat\theta)\,\theta_j - \error - \frac{L_\pi}{n}\,\|\theta - \hat\theta\|_1\Bigr)\Bigr\}\,d\theta \\
        &\qquad \ge \int_{\hat\Theta_\delta} \exp\Bigl\{-\frac{ns_2}{2}\,\|\theta_{S_0} - \hat\theta_{S_0}\|_2^2 - \bigl(n+\frac{L_\pi}{C_{S_1}}\bigr)\,\|\theta_{S_1} - \hat\theta_{S_1}\|_{\nabla_{S_1}\ell(\hat\theta)} - n\,\error - L_\pi\,\|\theta_{S_0} - \hat\theta_{S_0}\|_1\Bigr\}\,d\theta \\
        &\qquad \ge \exp\Bigl( - \frac{L_\pi d_0\delta_0}{\sqrt n} - n\,\error\Bigr)\,\bigl(\frac{2\pi}{ns_2}\bigr)^{d_0/2}\,\Bigl(1-\PP\bigl(\bigl\lVert\frac{Z}{\sqrt{ns_2}}\bigr\rVert_2 \geq \delta_0 \sqrt{\frac{d_0}{n}}\bigr)\Bigr) \\ 
        &\qquad\qquad{}\times \int_{\hat\Theta_{\delta_1}} \exp\Bigl\{-\bigl(n + \frac{L_\pi}{C_{S_1}}\bigr)\,\|\theta_{S_1} - \hat\theta_{S_1}\|_{\nabla_{S_1}\ell(\hat\theta)}\Bigr\}\,d\theta \\
        &\qquad \geq \exp\Bigl( - \frac{L_\pi d_0\delta_0}{\sqrt n} - n\,\error\Bigr)\,\bigl(\frac{2\pi}{ns_2}\bigr)^{d_0/2}\, \Bigl(1-\exp\Bigl\{-d_0\,\frac{(\delta_0\sqrt{s_2}-1)^2}{2}\Bigr\}\Bigr)\\
        &\qquad\qquad{}\times \prod_{j\in S_1} \frac{1-\exp\bigl(\partial_j\ell(\hat\theta)\,\delta_1\,(1 + L_\pi/(C_{S_1} n))\bigr)}{(n+L_\pi/C_{S_1})\,|\partial_j \ell(\hat\theta)|} \\
        &\qquad \geq \frac{1}{2}\exp\Bigl( - \frac{L_\pi d_0\delta_0}{\sqrt n} - n\,\error\Bigr)\,\bigl(\frac{2\pi}{ns_2}\bigr)^{d_0/2} \prod_{j\in S_1} \frac{1-\exp\bigl(\partial_j\ell(\hat\theta)\,\delta_1\,(1 + L_\pi/(C_{S_1} n))\bigr)}{(n+L_\pi/C_{S_1})\,|\partial_j \ell(\hat\theta)|}\,,
    \end{align*}
     where we have the concentration inequality:
     \begin{align*}
         \PP(\|Z\|_2 - \sqrt{d_0} \ge t) \le \exp\bigl(-\frac{t^2}{2}\bigr)
     \end{align*}
    for $t > 0$.
    The last inequality follows provided $\delta_0 \sqrt{s_2} \gg 1$.
    Taking the ratio $\text{Num(II)}/\text{Den(II)}$, the terms $\prod_{j\in S_1} |\partial_j \ell(\hat\theta)|$ cancel, yielding:
    \begin{align*}
        {\rm II} &\le \frac{2ed_0 \mathsf V_{d_0}\, (\frac{d_0}{en\Cgr_0})^{d_0}\, d_1 \exp\{-n\Cgr_1 C_{S_1} r_1\} \prod_{k \in S_1} (n\Cgr_1)^{-1}}{\exp\Bigl( - \frac{L_\pi d_0\delta_0}{\sqrt n} - n\,\error\Bigr)\,\bigl(\frac{2\pi}{ns_2}\bigr)^{d_0/2} \prod_{j\in S_1} \frac{1-\exp\bigl(\partial_j\ell(\hat\theta)\,\delta_1\,(1 + L_\pi/(C_{S_1} n))\bigr)}{n+L_\pi/C_{S_1}}} \\
        &= \frac{2ed_0 \mathsf V_{d_0}\, (\frac{d_0}{en\Cgr_0})^{d_0}}{\exp\bigl( - \frac{L_\pi d_0\delta_0}{\sqrt n} - n\,\error\bigr)\,\bigl(\frac{2\pi}{ns_2}\bigr)^{d_0/2}} \\
        &\qquad \times d_1 \exp\{-n\Cgr_1 C_{S_1} r_1\}  \prod_{j \in S_1} \frac{n+L_\pi/C_{S_1}}{n\Cgr_1\, \bigl(1-\exp\bigl(\partial_j\ell(\hat\theta)\,\delta_1\,(1 + L_\pi/(C_{S_1} n))\bigr)\bigr)} \\
        &\le 2ed_1\sqrt{\frac{d_0}{\pi}}\, \bigl(\frac{\sqrt{d_0 s_2}}{\Cgr_0\sqrt{en}}\bigr)^{d_0} \exp\Bigl\{-n\Cgr_1 C_{S_1} r_1 + \frac{L_\pi d_0\delta_0}{\sqrt n}+n\, \error\Bigr\} \\ 
        &\qquad \times \Bigl( \frac{1+ L_\pi/(C_{S_1} n)}{\Cgr_1\,\bigl(1-\exp\{-\delta_1\,C_{S_1}\,(1 + L_\pi/(C_{S_1}n))\}\bigr)} \Bigr)^{d_1}
    \end{align*}
    where we used that $\mathsf V_{d_0} \sim \frac{1}{\sqrt{\pi d_0}}\,(2\pi e/d_0)^{d_0/2}$.
    Let us start with the last term.
    If we assume: $L_\pi \lesssim C_{S_1} n$, $\delta_1 \gtrsim C_{S_1}^{-1} \log d_1$, as well as the condition

    \begin{align}\label{eq:n_cond_2}
        n \gtrsim \frac{s_2^2 \delta_0 d_0 d_1}{C_{S_1}^2} \vee \frac{s_2 \delta_1 d_1}{C_{S_1}}
    \end{align}
    to ensure that $\Cgr_1 = \Omega(1)$, we have
    \begin{align*}
        \Bigl[\frac{1 + L_\pi/(C_{S_1} n)}{\Cgr_1\,\bigl(1-\exp\{-\delta_1\,(C_{S_1} + L_\pi/n)\}\bigr)}\Bigr]^{d_1}
        &\lesssim e^{d_1} \Bigl(1 + O\bigl(\frac{1}{d_1}\bigr)\Bigr)^{d_1}
         = O(e^{d_1})\,.
    \end{align*} 
   Therefore, it remains to control
    \begin{align*}
        d_1\sqrt{\frac{d_0}{\pi}}\, \bigl(\frac{\sqrt{d_0 s_2}}{\Cgr_0\sqrt{en}}\bigr)^{d_0} \exp\Bigl\{-n\Cgr_1 C_{S_1} r_1 + \frac{L_\pi d_0\delta_0}{\sqrt n}+n\, \error + d_1\Bigr\}\,.
    \end{align*}
    One can see that this is at most $\varepsilon/4$ provided
    \begin{align*}
        n \gg \bigl(\frac{L_\pi \delta_0 d_0}{C_{S_1} r_1}\bigr)^{2/3} \vee \bigl(\frac{s_2 \delta_0 \delta_1 \sqrt{d_0 d_1}}{C_{S_1} r_1}\bigr)^{2/3} \vee \bigl( \frac{s_2 \delta_1^2 d_1}{C_{S_1} r_1}\bigr)^{1/2} \vee \frac{d_0}{C_{S_1} r_1} \log \frac{\sqrt{d_0 s_2}}{\Cgr_0\sqrt n} \vee \frac{\log(d_0 d_1^2/\varepsilon)}{C_{S_1} r_1}\,.
    \end{align*}
     
%

\paragraph{Term III\@.}
We seek to bound ${\rm III} = \mu(\delta_0\sqrt{\frac{d_0}{n}} \leq \norm{\theta_{S_0}-\hat\theta_{S_0}}_2 \leq r_0,\, \norm{\theta_{S_1} - \hat\theta_{S_1}}_\infty \leq r_1)$.
Let $B \deq B(\hat\theta, r_0, r_1) \deq \{\theta \in \Theta : \norm{\theta_{S_0}-\hat\theta_{S_0}}_2 \leq r_0, \norm{\theta_{S_1} - \hat\theta_{S_1}}_\infty \leq r_1\}$. 

 Let $\nu \deq \mu|_B$ be the measure $\mu$ conditioned on $B$. 
We will bound this term using Markov's inequality, which requires a bound on the second moment $\mathbb{E}_\nu[\norm{\theta_{S_0} - \hat\theta_{S_0}}_2^2]$.

By \cref{lemma:term III}, 
$\mathbb{E}_\nu[\norm{\theta_{S_0} - \hat\theta_{S_0}}_2^2] \le \frac{2\,(d + \sqrt{d} L_\pi\, (r_0 + \sqrt{d_1} r_1))}{nC_{S_0}}$. By Markov's inequality, we have
\begin{align*}
\nu\Bigl(\norm{\theta_{S_0} - \hat\theta_{S_0}}_2 \ge \sqrt{\frac{6\,(d + \sqrt{d} L_\pi\, (r_0 + \sqrt{d_1} r_1))}{C_{S_0} d_0}}  \sqrt{\frac{d_0}{n}}\Bigr) \le \frac{1}{3}\,. 
\end{align*}
We can then apply Borell's inequality~\cite{Bor1974Cvx}:
\begin{align*}
&\mu\Bigl(r_0 \geq \norm{\theta_{S_0}-\hat\theta_{S_0}}_2 \geq t \sqrt{\frac{6\,(d + \sqrt{d} L_\pi\, (r_0 + \sqrt{d_1} r_1))}{C_{S_0} d_0}} \sqrt{\frac{d_0}{n}} \text{ and } \norm{\theta_{S_1} - \hat\theta_{S_1}}_\infty \leq r_1\Bigr) \\ 
&\qquad \leq \nu\Bigl(\norm{\theta_{S_0}-\hat\theta_{S_0}}_2 \geq t \sqrt{\frac{6\,(d + \sqrt{d} L_\pi\, (r_0 + \sqrt{d_1} r_1))}{C_{S_0} d_0}} \sqrt{\frac{d_0}{n}}\Bigr) \\\ 
&\qquad \leq \frac{2}{3}\times 2^{-\frac{t+1}{2}}\,.
\end{align*}
To ensure ${\rm III} \leq \frac{\epsilon}{4}$, it is sufficient to have $\delta_0 \ge (2 \log_2(\frac{8}{3\epsilon}) - 1)\sqrt{\frac{6\,(d + \sqrt{d} L_\pi\, (r_0 + \sqrt{d_1} r_1))}{C_{S_0} d_0}}$.

\paragraph{Term IV\@.}
We want to bound the mass of the region
$A \deq \bigl\{\theta \in\Theta: \norm{\theta_{S_0} - \hat\theta_{S_0}}_2 \leq \delta_0\sqrt{\frac{d_0}{n}},\, \frac{\delta_1}{n} \leq \norm{\theta_{S_1} - \hat\theta_{S_1}}_\infty \leq r_1 \bigr\}$.

Let us first define the local neighborhood
$$G \deq \Bigl\{\theta \in \Theta : \norm{\theta_{S_0} - \hat\theta_{S_0}}_2 \leq \delta_0\sqrt{\frac{d_0}{n}}\,,\; \norm{\theta_{S_1} - \hat\theta_{S_1}}_\infty \leq r_1 \Bigr\}\,.$$
Notice that $A \subseteq \bigcup_{j \in S_1} A_j$, where $A_j \deq \bigl\{\theta \in G : \theta_j - \hat\theta_j \geq \frac{\delta_1}{n}\bigr\}$.

For any $\theta \in G$ and $j \in S_1$, let $\bar\theta \deq (\hat\theta_{S_0}, \theta_{S_1})$. Since $\bar\theta \in B(\hat\theta, 0, r_1)$, our local gradient condition implies $\partial_j \ell(\bar\theta) \leq -C_{S_1}$. Moreover,
\begin{align*}
    |\partial_j \ell(\theta) - \partial_j \ell(\bar\theta)| 
    &\leq \sup_{u \in G} {\norm{\nabla^2 \ell(u)}_{\mathrm{op}}} \, \norm{\theta_{S_0} - \hat\theta_{S_0}}_2 \leq s_2 \delta_0 \sqrt{\frac{d_0}{n}}\,.
\end{align*}
Provided that $n \gtrsim s_2^2 \delta_0^2 d_0 / C_{S_1}^2$,
 we obtain $\partial_j \ell(\theta) \leq -\frac{C_{S_1}}{2}$ for all $\theta \in G$.

For a fixed $j \in S_1$, we now bound $\mu(A_j)$. Let $\tilde\mu(\theta) \deq \pi(\theta)\exp\{n\ell(\theta)\}1_\Theta(\theta)$ denote the unnormalized density. Let $t \deq \frac{\delta_1}{n}$. Since $G$ is defined by $\ell_\infty$ bounds and $\theta_j - \hat\theta_j \geq t$ on $A_j$, the shifted point $\theta - te_j$ for $\theta \in A_j$ remains within $G$. 

Using the product structure of the prior, we compare the unnormalized density at $\theta$ and $\theta - t e_j$:
\begin{align*}
    \log \frac{\tilde\mu(\theta)}{\tilde\mu(\theta - t e_j)} 
    &= \log \frac{\pi_j(\theta_j)}{\pi_j(\theta_j - t)} + n\, \bigl( \ell(\theta) - \ell(\theta - t e_j) \bigr) \\
    &= \log \frac{\pi_j(\theta_j)}{\pi_j(\theta_j - t)} + n \int_0^t \partial_j \ell(\theta - u e_j) \, d u\,.
\end{align*}
 The first term is at most $L_\pi t$ as $\log \pi_j$ is $L_\pi$-Lipschitz. For the second term, since the line segment from $\theta - t e_j$ to $\theta$ lies in $G$, we have $\partial_j \ell(\theta - u e_j) \leq - \frac{C_{S_1}}{2}$. Therefore,
\begin{align*}
    \log \frac{\tilde\mu(\theta)}{\tilde\mu(\theta - t e_j)} \leq L_\pi t - n\, \frac{C_{S_1}}{2}\, t = - \Bigl( \frac{C_{S_1}}{2} - \frac{L_\pi}{n} \Bigr)\, \delta_1\,.
\end{align*}
Exponentiating this inequality gives $\tilde\mu(\theta) \leq \exp\bigl(- ( \frac{C_{S_1}}{2} - \frac{L_\pi}{n} )\, \delta_1 \bigr)\, \tilde\mu(T_j(\theta))$.

Integrating this bound over $\theta \in A_j$ and applying the change of variables $\vartheta = T_j(\theta)$ which has Jacobian $1$, we obtain
\begin{align*}
    \int_{A_j} \tilde\mu(\theta) \, d\theta 
    &\leq \exp\Bigl(- \Bigl( \frac{C_{S_1}}{2} - \frac{L_\pi}{n} \Bigr)\, \delta_1 \Bigr) \int_{A_j} \tilde\mu(\theta - t e_j) \, d\theta \\
    &\leq \exp\Bigl(- \Bigl( \frac{C_{S_1}}{2} - \frac{L_\pi}{n} \Bigr)\, \delta_1 \Bigr) \int_G \tilde\mu(\vartheta) \, d\vartheta\,,
\end{align*}
where the last inequality follows since $T_j(A_j) \subseteq G$. Dividing both sides by the normalizing constant $Z = \int_\Theta \tilde\mu(\vartheta) \, d\vartheta$,
\begin{align*}
    \mu(A_j) \leq \exp\Bigl(- \Bigl( \frac{C_{S_1}}{2} - \frac{L_\pi}{n} \Bigr)\, \delta_1 \Bigr)\, \mu(G) \leq \exp\Bigl(- \Bigl( \frac{C_{S_1}}{2} - \frac{L_\pi}{n} \Bigr)\, \delta_1 \Bigr)\,.
\end{align*}

Finally, we apply a union bound over all $j \in S_1$:
\begin{align*}
    {\rm IV} \leq \sum_{j \in S_1} \mu(A_j) \leq d_1 \exp\Bigl(- \Bigl( \frac{C_{S_1}}{2} - \frac{L_\pi}{n} \Bigr)\, \delta_1 \Bigr)\,.
\end{align*}
This can be made at most $\varepsilon/4$ provided that 
\begin{align*}
    n \gg \frac{L_\pi}{C_{S_1}} \vee \frac{s_2^2 \delta_0^2 d_0}{C_{S_1}^2} \qquad \text{and} \qquad \delta_1 \gtrsim \frac{\log(d_1/\varepsilon)}{C_{S_1}}\,.
\end{align*}


\paragraph{Finishing the proof.} Collating all of the conditions and keeping the dominant terms yields the statement of the theorem.
Here, we provide a more explicit set of conditions for later reference; for simplicity, we assume $L_\pi = 0$ and $d_1 \le d_0$ (so that $d_0 \asymp d$), and that the parameters $C_{S_0}$, $C_{S_1}$, $r_0$, $r_1$, and $s_2$ are polynomial in $d$.
From terms III and IV\@, we take
\begin{align*}
    \delta_0 \asymp \frac{1}{\sqrt{C_{S_0}}} \log\frac{1}{\varepsilon}\,, \qquad \delta_1 \asymp \frac{1}{C_{S_1}} \log \frac{d_1}{\varepsilon}\,.
\end{align*}
Noting that~\eqref{eq:n_cond_1} and the second term of~\eqref{eq:n_cond_2} are subsumed by the first term in~\eqref{eq:n_cond_2}, it leads to the condition
\begin{align*}
    n \gg \Bigl[\frac{s_2^2 d_0 d_1}{C_{S_0} C_{S_1}^2} + \frac{d_0}{C_{S_1} r_1} + \frac{s_2^{2/3} d_0^{1/3} d_1^{1/3}}{C_{S_0}^{1/3} C_{S_1}^{4/3} r_1^{2/3}} + \frac{s_2 d_1}{C_{S_1}^2} + \frac{s_2^{1/2} d_1^{1/2}}{C_{S_1}^{3/2} r_1^{1/2}} \Bigr] \log^2\bigl(\frac{d}{\varepsilon}\bigr)\,.
\end{align*}
\end{proof}

\subsection{Proof of Theorem~\ref{proof_lemma: exponential_decay}}\label{appendix:proof_lemma: exponential_decay}

\begin{proof}[Proof of \cref{proof_lemma: exponential_decay}]
    As in the proof of \cref{thm:concentration_log_concave}, let us break $\mu(\Theta \setminus \hat\Theta_{\delta})$ into four parts:
\begin{align*}
    {\rm I} &= \mu(\norm{\theta_{S_0}-\hat\theta_{S_0}}_2 \geq r_0)\,, \\ 
    {\rm II} &= \mu(\norm{\theta_{S_1}-\hat\theta_{S_1}}_{\infty} \geq r_1)\,, \\
    {\rm III} &= \mu\bigl(\delta_0\sqrt{\frac{d_0}{n}} \leq \norm{\theta_{S_0}-\hat\theta_{S_0}}_2 \leq r_0,\, \norm{\theta_{S_1} - \hat\theta_{S_1}}_\infty \leq r_1\bigr)\,, \\ 
    {\rm IV} &= \mu\bigl(\norm{\theta_{S_0} - \hat\theta_{S_0}}_2 \leq \delta_0 \sqrt{\frac{d_0}{n}},\, \frac{\delta_1}{n} \leq \norm{\theta_{S_1}- \hat\theta_{S_1}}_\infty \leq r_1\bigr)\,.
\end{align*}
All of the above regions are understood to be intersected with $\Theta$.
Since the proof is mostly similar to the one for \cref{thm:concentration_log_concave}, we keep the arguments brief.

\paragraph{Term I.}
    For term I, we bound the ratio:
    \begin{align*}
        {\rm I} &= \frac{\int_{(\Theta_{S_0} \setminus B_2(\hat\theta_{S_0}, r_0)) \times \Theta_{S_1} } \exp\{n\,(\tilde\ell(\theta) - \tilde\ell(\hat\theta))\}\,d\theta}{\int_\Theta \exp\{n\,(\tilde\ell(\theta) - \tilde\ell(\hat\theta))\}\,d\theta} \\
        &\leq \frac{\int_{(\Theta_{S_0} \setminus B_2(\hat\theta_{S_0}, r_0)) \times \Theta_{S_1} } \exp\{-n\zeta + \sqrt{d_0} L_\pi\, \norm{\theta_{S_0} - \hat\theta_{S_0}}_2 + d_1 L_\pi\, \norm{\theta_{S_1} - \hat\theta_{S_1}}_\infty\}\,d\theta}{\int_{\hat\Theta_\delta} \exp\{n\,(\tilde\ell(\theta) - \tilde\ell(\hat\theta))\}\,d\theta} \\
        &\leq \frac{\mathsf V_{d_0}(R_\Theta^{d_0}-r_0^{d_0}) \mathsf V_{d_1} R_\Theta^{d_1}\exp(-n\zeta + \sqrt{d_0} L_\pi R_\Theta + d_1 L_{\pi} R_\Theta)}{\frac{1}{2}\exp\Bigl( - \frac{L_\pi d_0\delta_0}{\sqrt n} - n\,\error\Bigr)\,\bigl(\frac{2\pi}{ns_2}\bigr)^{d_0/2} \prod_{j\in S_1} \frac{1-\exp\bigl(\partial_j\ell(\hat\theta)\,\delta_1\,(1 + L_\pi/(C_{S_1} n))\bigr)}{(n+L_\pi/C_{S_1})\,|\partial_j \ell(\hat\theta)|}}\,,
    \end{align*}
    where the lower bound for the denominator comes from the proof of Term II in \cref{appendix:proof_lemma: log concave}.
    This yields the bound
    \begin{align*}
        {\rm I}
        &\lesssim \frac{\frac{1}{\sqrt{d_0 d_1}}\,(\frac{ns_2 e}{d_0})^{d_0/2} (\frac{2\pi e}{d_1})^{d_1/2} R_{\Theta}^{d} \exp(-n\zeta + \sqrt{d_0} L_\pi R_\Theta + d_1 L_{\pi} R_\Theta + \frac{L_\pi d_0\delta_0}{\sqrt n} + n\,\error)}{\prod_{j\in S_1} \frac{1-\exp\bigl(C_{S_1}\,\delta_1\,(1 + L_\pi/(C_{S_1} n))\bigr)}{(n+L_\pi/C_{S_1})\,|\partial_j \ell(\hat\theta)|}}\,.
    \end{align*}
    By same arguments as in \cref{appendix:proof_lemma: log concave}, if we assume that $L_\pi \lesssim C_{S_1} n/d_1$, $\delta_1 \gtrsim C_{S_1}^{-1} \log d_1$, 
    and
    \begin{align*}
        n
        &\gg \frac{1}{\zeta}\,\Bigl[\sqrt{d_0} L_\pi R_\Theta \vee d_1 L_\pi R_\Theta  \vee d \log R_\Theta \vee \log \frac{1}{\varepsilon}\Bigr] \\
        &\qquad{} \vee \bigl(\frac{L_\pi d_0 \delta_0}{\zeta}\bigr)^{2/3} \vee \bigl(\frac{s_2 \delta_0 \delta_1 \sqrt{d_0 d_1}}{\zeta}\bigr)^{2/3} \vee \bigl(\frac{s_2 \delta_1^2 d_1}{\zeta}\bigr)^{1/2}\,,
    \end{align*}
    Term I is bounded by $\varepsilon/4$.

\paragraph{Term II.}
This term is bounded in exactly the same way as Term I\@.

\paragraph{Terms III and IV.}
    Term III and IV are bounded exactly the same as in the proof of \cref{appendix:proof_lemma: log concave} because of the log-concavity in $B(\hat\theta, r_0, r_1)$, local strong log-concavity in the regular part and the strictly negative gradient in the non-regular part, albeit this bound is looser in this well-separated mode case since we do not have mass outside of a compact set of radius $R_\Theta$.

\paragraph{Finishing the proof.} Collating all of the conditions finishes the proof.
\end{proof}
\section{Checking the assumptions for the examples}


\subsection{Proof of Theorem~\ref{check:assumptions_lr}}
\begin{proof}
    We will check each assumption.
\paragraph{\cref{assumption: random convexity of regular part at true parameter}.}
    This follows directly from Theorem 6 in \cite{chardon2024finitesampleperformancemaximumlikelihood}.

\paragraph{\cref{ass:subG}.}
Recall that
\begin{align*} 
    \grad \ell(\theta; X_1) &= \bigl(Y_1 - c^{\prime}(X_1^\T\theta)\bigr)\,X_1\,.
\end{align*}
Since $0 \le c^\prime \leq 1$ and $Y_1 \in \{0,1\}$, the term $Y_1 - c^{\prime}(X_1^\T\theta)$ is bounded between $-1$ and $1$. Given that $X_1$ is a standard Gaussian random vector, we have
\begin{align*}
        \|\langle v, \nabla \ell(\theta; X_1) - \E\nabla \ell(\theta; X_1)\rangle\|_{\psi_1} \lesssim \norm{\inangle{v, X_1}}_{\psi_2} \lesssim 1\,.
    \end{align*}

\paragraph{\cref{ass:bounded_mixed_derivatives_ell_n}.}
Since $\grad^2 \ell_n(\theta) = -n^{-1} \sum_{i=1}^n c^{\prime\prime}(X_i^\T\theta)\,X_i X_i^\T$ and $c^{\prime\prime}(\eta) = \frac{\exp{\eta}}{(1+\exp{\eta})^2} \in [0,1]$,
\begin{align*}
    0\preceq -\nabla^2 \ell_n(\theta)
    \preceq \frac{1}{n} \sum_{i=1}^n X_i X_i^\T\,.
\end{align*}
By~\cite[Example 6.1]{Wai19Stats}, with probability at least $1-\eta$,
$$\bigl\lVert \frac{1}{n}\sum_{i=1}^n X_i X_i^\T\bigr\rVert_{\rm op} \lesssim 1 + \sqrt{\frac{d+\log(1/\eta)}{n}} + \frac{d+\log(1/\eta)}{n}\,.$$
Thus, the operator norm of the Hessian is $O(1)$ provided $n \gg d + \log(1/\eta)$.
\end{proof}


\subsection{Proof of Theorem~\ref{theorem:poisson_check_assumptions}}

\begin{proof}
Since $A_i\theta^\star \geq c$ for all $i \in [n]$, and $\norm{A_i}_2 \leq a_2$, then for any $\theta \in B(\theta^\star, r_0, r_1)$,
$$A_i\theta \geq A_i\theta^\star - \norm{A_i}_2\,\norm{\theta - \theta^\star}_2 \geq A_i\theta^\star - a_2 R \geq A_i\theta^\star/2 \ge c/2\,,$$
where we used the assumption that $R \leq c/(2a_2)$.

We note that since $TY_i \sim \Pois(TA_i\theta^\star)$, then $\norm{TY_i}_{\psi_1} \lesssim TA_i\theta^\star$.
In particular,~\cite[Theorem 2.8.1]{Vershynin_2018} shows that with probability at least $1-\eta$,
\begin{align}\label{eq:pois_concentration}
    \frac{1}{n} \sum_{i=1}^n \frac{TY_i}{A_i\theta^\star}
    &\lesssim T\,\Bigl(1 + \sqrt{\frac{\log(1/\eta)}{n}} + \frac{\log(1/\eta)}{n}\Bigr)\,.
\end{align}

We first verify~\cref{ass:subG} and~\cref{ass:bounded_mixed_derivatives_ell_n}.
Then, to verify~\cref{assumption: random convexity of regular part at true parameter}, we invoke \cref{lemma: l_n convexity}, which first requires us to show Lipschitz continuity of the Hessian (\cref{assumption:holder_hessian}).

\paragraph{\cref{ass:subG}.} For any $\theta \in B(\theta^\star, r_0, r_1)$,
\begin{align*}
        \nabla \ell(\theta; Y_i) - \E\nabla \ell(\theta; Y_i)
        &= -TA_i + \frac{TY_i A_i}{A_i\theta} - \bigl(- TA_i + \frac{\E[TY_i] A_i}{A_i\theta}\bigr) \\[0.25em]
        &=  (TY_i - \E[TY_i])\, \frac{A_i}{A_i\theta}\,.
\end{align*}
Since $\norm{A_i}_2 \leq a_2$ and $TY_i \sim \Pois(TA_i\theta^\star)$, for any unit vector $v \in S^{d-1}$,
\begin{align*}
        \|\langle v, \nabla \ell(\theta; Y_i) - \E\nabla \ell(\theta; Y_i)\rangle\|_{\psi_1} \lesssim TA_i\theta^\star\, \frac{a_2}{A_i \theta} \lesssim Ta_2\,.
    \end{align*}

\paragraph{\cref{ass:bounded_mixed_derivatives_ell_n}.}
Fix $\theta \in B(\theta^\star, r_0, r_1)$. 
Then, by~\eqref{eq:pois_concentration}, we have
$$\sup_{\theta \in B(\theta^\star,r_0,r_1)}\norm{\grad^2\ell_n(\theta)}_{\rm op} = \sup_{\theta \in B(\theta^\star,r_0,r_1)}{\Bigl\lVert \frac{1}{n}\sum_{i=1}^n \frac{TY_i A_i^\T {A_i}}{(A_i \theta)^2}\Bigr\rVert_{\rm op}} \lesssim \frac{a_2^2}{c}\,\frac{1}{n}\sum_{i=1}^n \frac{TY_i}{A_i \theta^\star} \lesssim \frac{Ta_2^2}{c}$$
provided $n \gg \log(1/\eta)$.

\paragraph{\cref{assumption:holder_hessian}.} 
We take $\gamma= 1$. Let $f(\theta) \deq \frac{1}{(A_i\theta)^2}$. Then, $\grad f(\theta) = -\frac{2A_i^T}{(A_i\theta)^3}$. Therefore, by the mean value theorem, for any $\theta, \theta^\prime \in B(\theta^\star, r_0, r_1)$ and $i$ with $Y_i > 0$,
\begin{align*} 
    \bigl\lvert\frac{1}{(A_i\theta)^2} - \frac{1}{(A_i\theta^\prime)^2}\bigr\rvert
    &= \abs{\langle \grad f(\tilde{\theta}), \theta - \theta^\prime\rangle} \leq \norm{\grad f(\tilde{\theta})}_2\, \norm{\theta - \theta^\prime}_2 \leq \frac{2\,\norm{A_i}_2}{(A_i\tilde{\theta})^3}\,\norm{\theta - \theta^\prime}_2 \\
    &\lesssim \frac{a_2}{(A_i\theta^\star)^3}\, \norm{\theta - \theta^\prime}_2
\end{align*}
for some $\tilde{\theta}$ in the line segment between $\theta$ and $\theta^\prime$.

Uniformly over all $\theta, \theta^\prime \in B(\theta^\star, r_0, r_1)$, 
\begin{align*} 
 \norm{\grad^2 \ell_{n}(\theta) - \grad^2 \ell_n(\theta^\prime)}_{\rm op}
 &= \Bigl\lVert\frac{1}{n} \sum_{i=1}^n TY_i A_i^\T A_i\,\bigl(\frac{1}{(A_i\theta)^2}-\frac{1}{(A_i\theta^\prime)^2}\bigr)\Bigr\rVert_{\rm op} \\ 
&\lesssim \frac{a_2^2}{n} \sum_{i=1}^n TY_i\,\frac{a_2\,\norm{\theta-\theta^\prime}_2}{(A_i\theta^\star)^3} \\ 
&\le \frac{a_2^3}{c^2n} \sum_{i=1}^n \frac{TY_i}{A_i\theta^\star}\,\norm{\theta-\theta^\prime}_2 \\ 
&\lesssim \frac{Ta_2^3}{c^2}\, \norm{\theta-\theta^\prime}_2 
\text{ by \eqref{eq:pois_concentration} with probability at least } 1-\eta\,,
\end{align*}
provided $n \gg \log(1/\eta)$.
Therefore,~\cref{assumption:holder_hessian} holds with $s_3 \lesssim Ta_2^3/c^2$.

\paragraph{\cref{assumption: random convexity of regular part at true parameter}.}
Applying \cref{lemma: l_n convexity} gives
$$\inf_{\theta \in B(\theta^\star, r_0, r_1)}\lambda_{\min}(-\grad^2_{S_0} \ell_n(\theta)) \geq \frac{c_{S_0}^\star}{2} - \max_{\theta \in N_\epsilon(\theta^\star, r_0, r_1)}\norm{\grad^2_{S_0}\ell_n(\theta) - \grad^2_{S_0}\ell^\star(\theta)}_{\rm op}$$ with probability at least $1-\eta$, provided that $N_\varepsilon(\theta^\star,r_0,r_1)$ is an $\varepsilon$-covering net of $B(\theta^\star,r_0,r_1)$ with $\varepsilon \le \frac{c_{S_0}^\star}{2s_{3}}$. Now, we need the following uniform bound over any $\theta\in B(\theta^\star, r_0, r_1)$: by the matrix Bernstein inequality~\cite[Theorem 6.2]{troppUserfriendlyTailBounds2012},
\begin{align*}
        \norm{\grad^2_{S_0} \ell_n(\theta) - \E\grad^2_{S_0} \ell_n(\theta)}_{\rm op}
        &= \Bigl\lVert \frac{1}{n}\sum_{i=1}^n \frac{(TY_i - \E[TY_i])\, A_i^\T {A_i}}{( A_i \theta)^2}\Bigr\rVert_{\rm op}\\ 
        &\lesssim \frac{Ta_2^2}{c}\,\Bigl(\sqrt{\frac{\log(d/\eta)}{n}} + \frac{\log(d/\eta)}{n}\Bigr)\,,
\end{align*}
with probability at least $1-\eta$.
This is at most $c_{S_0}^\star/4$, provided
\begin{align*}
    n \gg \Bigl[\frac{Ta_2^2}{cc_{S_0}^\star} \vee \Bigl(\frac{Ta_2^2}{cc_{S_0}^\star}\Bigr)^2\Bigr] \log \frac{d}{\eta}\,.
\end{align*}
This implies that~\cref{assumption: random convexity of regular part at true parameter} holds with $c_{S_0} = c_{S_0}^\star/4$.
\end{proof}


\subsection{Some properties of Gaussian mixture models}
\paragraph{Basic properties.}
The mean and covariance of the random variable $X_i$ are given by
\begin{align*} 
    \bar\mu &\coloneqq \E X_i = \sum_{j=1}^{k}\omega_j\mu_j^\star\,, \\
    \Sigma &\coloneqq \cov(X_i) = \E[\cov(X_i)\mid Z_i] + \cov(\E[X_i\mid Z_i]) \\ 
    &= \sum_{j=1}^k \omega_j\Sigma_j + \sum_{j=1}^k \omega_j\, (\mu_j-\bar\mu)\,(\mu_j-\bar\mu)^\T\,,
\end{align*}
where $Z_i \in [k]$ is a latent variable denoting component membership.

We calculate the derivatives of the likelihood. 
The empirical log-likelihood is given by
\begin{align*}
    \ell_n(\theta) &= \frac{1}{n}\sum_{i=1}^{n} \log\Bigl(\sum_{j \in [k]}\omega_j\sN(X_i \mid \mu_j, \Sigma_j)\Bigr) \\ 
    &= \frac{1}{n}\sum_{i=1}^{n} \log\Bigl(\sum_{j \in [k]}\omega_j\exp\bigl\{-\frac{1}{2}\,(X_i - \mu_j)^\T\Sigma_j^{-1}(X_i - \mu_j)\bigr\}\Bigr) + \frac{1}{n}\sum_{i=1}^{n} \log\Bigl(\sum_{j \in [k]}\frac{\omega_j}{\sqrt{(2\pi)^d\det\Sigma_j}}\Bigr)\,.
\end{align*}
To simplify notation, for $i \in [n]$, $j \in [k]$, define $$\gamma_{ij} \deq \frac{\omega_j\sN(X_i \mid \mu_j, \Sigma_j)}{\sum_{j' \in [k]}\omega_{j'} \sN(X_i \mid \mu_{j'}, \Sigma_{j'})}\,, \qquad g_{ij} = \Sigma_j^{-1}(X_i - \mu_j)\,.$$ Let $\gamma_i \deq (\gamma_{i1}, \dotsc, \gamma_{ik})^\T$, $g_i \deq (g_{i1}, \dotsc, g_{ik})^\T$. 
The gradient of the log-likelihood with respect to $\mu$ is given by 
\begin{align}\label{gmm_gradient}
    \nabla_{\mu_j} \ell_n(\theta)&= \frac{1}{n}\sum_{i=1}^{n} \frac{\omega_j\sN(X_i \mid \mu_j, \Sigma_j)\,\Sigma_j^{-1}( X_i - \mu_j)}{\sum_{j' \in [k]}\omega_{j'}\sN(X_i \mid \mu_{j'}, \Sigma_{j'})} = \frac{1}{n}\sum_{i=1}^{n} \gamma_{ij} g_{ij}\,. 
\end{align}
The diagonal blocks of the Hessian are given by
\begin{align} \label{gmm_hessian}
    &\nabla^2_{\mu_j,\mu_j}\ell_n(\theta) = \frac{1}{n}\sum_{i=1}^{n} \bigl(\gamma_{ij} g_{ij} g_{ij}^\T - \gamma_{ij}^2 g_{ij} g_{ij}^\T - \gamma_{ij} \Sigma_j^{-1}\bigr) = \frac{1}{n}\sum_{i=1}^{n} \bigl(\gamma_{ij}\, (1-\gamma_{ij})\, g_{ij} g_{ij}^\T - \gamma_{ij} \Sigma_j^{-1} \bigr)\,.
\end{align}
The mixed Hessian (for $j \neq j^\prime$) is given by
\begin{align}\label{gmm_mixed_hessian}
    &\nabla^2_{\mu_{j}, \mu_{j^\prime}}\ell_n(\theta) = -\frac{1}{n}\sum_{i=1}^{n} \gamma_{ij} \gamma_{ij^{\prime}} g_{ij} g_{ij^\prime}^\T\,.
\end{align}
\paragraph{Concentration properties.}
We explore the concentration properties of the random variable $X_i$.
Throughout, let $K_1 \deq R_\Theta + \max_{j\in [k]}\sqrt{\Tr(\Sigma_j)}$.
\begin{lemma}\label{lemma:gmm_subgaussian}
    The following bound holds:
    \[
    \norm{\norm{X_i}_2}_{\psi_2} \vee \norm{\norm{X_i - \bar\mu}_2}_{\psi_2} \lesssim K_1 = R_\Theta + \max_{j\in [k]} \sqrt{\Tr(\Sigma_j)}\,.
    \]
\end{lemma}
\begin{proof}
    Since $X_i$ is a mixture of $k$ Gaussians, it suffices to bound the sub-Gaussian norm for an arbitrary component $j \in [k]$. Let $Y \sim \sN(\mu_j^\star, \Sigma_j)$. We can represent $Y$ as $Y = \mu_j^\star + \Sigma_j^{1/2}Z$, where $Z \sim \mathcal{N}(0, I_d)$. By the triangle inequality for the sub-Gaussian norm, we have
    \begin{align*}
        \norm{\norm{Y}_2}_{\psi_2} &= \norm{\norm{\mu_j^\star + \Sigma_j^{\frac{1}{2}}Z}_2}_{\psi_2} \\
        &\leq \norm{\mu_j^\star}_2 + \norm{\norm{\Sigma_j^{\frac{1}{2}}Z}_2}_{\psi_2}\,.
    \end{align*}
    To bound the second term, define the function $f(z) \deq \norm{\Sigma_j^{1/2} z}_2$. For any $z_1, z_2 \in \mathbb{R}^d$, the Lipschitz continuity of $f$ follows from the definition of the operator norm:
    \[
    |f(z_1) - f(z_2)| \leq \norm{\Sigma_j^{1/2}(z_1 - z_2)}_2 \leq \norm{\Sigma_j^{1/2}}_{\text{op}}\, \norm{z_1 - z_2}_2\,.
    \]
    Thus, $f$ is Lipschitz with constant $\norm{\Sigma_j}_{\text{op}}^{1/2}$. By the Gaussian concentration inequality for Lipschitz functions (see, e.g.,~\cite[Theorem 5.2.2]{Vershynin_2018}), it holds that $\norm{f(Z) - \mathbb{E} f(Z)}_{\psi_2} \lesssim \norm{\Sigma_j}_{\text{op}}^{1/2}$.
    
    Furthermore, by Jensen's inequality, 
    \[
    \mathbb{E} f(Z) \leq \bigl(\mathbb{E}[\norm{\Sigma_j^{1/2}Z}_2^2]\bigr)^{1/2} = \left(\Tr(\Sigma_j)\right)^{1/2}\,.
    \]
    Combining the concentration result with the expectation bound yields
    \[
    \norm{\norm{\Sigma_j^{\frac{1}{2}}Z}_2}_{\psi_2} \leq \norm{f(Z) - \mathbb{E} f(Z)}_{\psi_2} + \norm{\mathbb{E} f(Z)}_{\psi_2} \lesssim \norm{\Sigma_j}_{\text{op}}^{1/2} + \sqrt{\Tr(\Sigma_j)}\,.
    \]
    Since $\norm{\Sigma_j}_{\text{op}} \leq \Tr(\Sigma_j)$, the overall bound is of order $\sqrt{\Tr(\Sigma_j)}$. Finally, observing that $\norm{\mu_j^\star}_2 \leq R_\Theta$ yields the bound for a single component.

    Let $J_i$ denote the index of the component for $X_i$.
    Then,
    \begin{align*}
        \|\|X_i\|_2\|_{\psi_2}
        &\le \|\|X_i - \mu^\star_{J_i}\|_2\|_{\psi_2} + \|\|\mu^\star_{J_i}\|_2\|_{\psi_2}
        \lesssim \|\|X_i - \mu^\star_{J_i}\|_2\|_{\psi_2} + R_\Theta
    \end{align*}
    since $\mu^\star_{J_i}$ is bounded.
    The first term is bounded by $R_\Theta + \max_{j\in [k]}\sqrt{\Tr(\Sigma_j)}$, by conditioning on the component.
The bound for $\norm{X_i - \bar\mu}_2$ follows by the triangle inequality.
\end{proof}

\begin{corollary}\label{cor:gmm_concentrations}
    Let $K_1 \deq R_\Theta + \max_{j\in [k]}\sqrt{\Tr(\Sigma_j)}$.
    There exists a universal constant $c' > 0$ such that for any $t > 0$, we have
    \begin{align*}
        \PP\Bigl(\frac{1}{n}\sum_{i=1}^{n}\bigl\{ \norm{X_i}_2 - \E[\norm{X_i}_2]\bigr\} > t\Bigr) &\leq 2\exp\Bigl(-\frac{c' nt^2}{K_1^2}\Bigr)\,, \\
        \PP\Bigl(\frac{1}{n}\sum_{i=1}^{n} \bigl\{\norm{X_i}_2^2 - \E[\norm{X_i}_2^2]\bigr\} > t\Bigr) &\leq 2\exp\Bigl(-c' n \min\bigl(\frac{t^2}{K_1^4}, \frac{t}{K_1^2}\bigr)\Bigr)\,, \\
        \PP\Bigl(\frac{1}{n}\sum_{i=1}^{n} \bigl\{\norm{X_i - \bar\mu}_2 - \E[\norm{X_i - \bar\mu}_2]\bigr\} > t\Bigr) &\leq 2\exp\Bigl(-\frac{c'nt^2}{K_1^2}\Bigr)\,.
    \end{align*}
    Furthermore, the second moment satisfies the bound
    \[
    \E[\norm{X_i-\bar\mu}_2^2] \le \E[\norm{X_i}_2^2] \leq d \lambda_{\max} + R_\Theta^2\,.
    \]
\end{corollary}

\begin{proof}
    We first establish the bound on the expected value. By the law of total expectation and the properties of the trace, we have
    \[
    \E[\norm{X_i}_2^2] = \sum_{j=1}^k \omega_j \left(\Tr(\Sigma_j) + \norm{\mu_j^\star}_2^2\right).
    \]
    Using the bound $\Tr(\Sigma_j) \leq d \norm{\Sigma_j}_{\text{op}} \leq d \lambda_{\max}$ and $\norm{\mu_j^\star}_2 \leq R_\Theta$, we obtain $\E[\norm{X_i}_2^2] \leq d\lambda_{\max} + R_\Theta^2$.

    The concentration inequalities follow from standard results for sums of independent sub-Gaussian and sub-exponential random variables. Specifically, by \cref{lemma:gmm_subgaussian}, $\norm{X_i}_2$ is sub-Gaussian with norm bounded by $O(K_1)$. The square $\norm{X_i}_2^2$ is therefore sub-exponential with norm $\norm{\norm{X_i}_2^2}_{\psi_1} = \norm{\norm{X_i}_2}_{\psi_2}^2 \lesssim K_1^2$. Applying Bernstein's inequality yields the stated probability bounds.
\end{proof}

\begin{lemma}[{\cite[Exercise 4.7.3]{Vershynin_2018}}]\label{lemma:gmm_covariance_concentration}
    Let $\hat\Sigma_n \deq \frac{1}{n}\sum_{i=1}^{n} (X_i - \E X_i) (X_i - \E X_i)^\T$ be the sample covariance matrix. Then, for any $u \geq 0$, with probability at least $1-2e^{-u}$, we have
    $$\norm{\hat\Sigma_n - \Sigma}_{\rm op} \lesssim \frac{R_\Theta^2 + \lambda_{\max}}{\lambda_{\min}}\,\Bigl(\sqrt{\frac{d+u}{n}} + \frac{d+u}{n}\Bigr)\,\norm{\Sigma}_{\rm op}\,.$$
\end{lemma}
To establish \cref{lemma:gmm_covariance_concentration}, it suffices to verify the sub-Gaussian concentration of the centered random vector as required by~\cite[Theorem 4.7.1]{Vershynin_2018}:
\begin{proposition}\label{prop:subgaussian_property}
    The centered random vector $X_i - \bar\mu$ is sub-Gaussian. Specifically,
    $$\norm{\inangle{X_i - \bar\mu, x}}_{\psi_2} \lesssim \frac{R_\Theta + \sqrt{\lambda_{\max}}}{\sqrt{\lambda_{\min}}}\, \norm{\inangle{X_i - \bar\mu, x}}_{L^2}$$
    for any $x \in \R^d$. 
\end{proposition}
\begin{proof}
    Let $x \in \R^d$ be fixed. The scalar random variable $\inangle{X_i-\bar\mu, x}$ follows a mixture distribution with components $\sN(\langle \mu_j^\star - \bar\mu, x\rangle, \inangle{\Sigma_j x, x})$. By the triangle inequality for the sub-Gaussian norm, we decompose the variable into the contribution from the component means and the within-component fluctuation:
    \begin{align*}
        \norm{\inangle{X_i-\bar\mu, x}}_{\psi_2} &\leq \norm{\inangle{X_i - \mu_{J_i}^\star, x}}_{\psi_2} + \norm{\inangle{\mu_{J_i}^\star - \bar\mu, x}}_{\psi_2}\,,
    \end{align*}
    where $J_i \in [k]$ denotes the latent component index. Conditional on $J_i=j$, the first term is Gaussian with standard deviation $\sqrt{\inangle{\Sigma_j x, x}} \leq \sqrt{\lambda_{\max}(\Sigma_j)}\,\norm{x}_2$. The second term involves a bounded random variable, as $\abs{\langle \mu_j^\star - \bar\mu, x\rangle} \leq \norm{\mu_j^\star - \bar\mu}_2\, \norm{x}_2 \leq 2R_\Theta\,\norm{x}_2$, which implies that its sub-Gaussian norm is bounded by $O(R_\Theta\, \norm{x}_2)$.
    
    Combining these bounds, we have
    \begin{align}\label{eq:gmm_subg_in_dir}
    \norm{\inangle{X_i-\bar\mu, x}}_{\psi_2} \lesssim  (R_\Theta + \sqrt{\lambda_{\max}})\, \norm{x}_2\,.
    \end{align}
We observe that $$\norm{\inangle{X_i - \bar\mu, x}}_{L^2} = \sqrt{\inangle{\Sigma x, x}} \geq \sqrt{\lambda_{\min}(\Sigma)}\,\norm{x}_2 \geq \sqrt{\lambda_{\min}}\, \norm{x}_2\,,$$ which proves the result.
\end{proof}

\begin{lemma}\label{lemma:result_gmm_covariance_concentration}Let
\begin{align*}
    n\gg \frac{R_\Theta^4 + \lambda_{\max}^2}{\lambda_{\min}^2}\,\bigl(d+\log(1/\eta)\bigr)\,.
\end{align*}
    Uniformly over $\theta\in\Theta$ and $j,l \in [k]$, with probability at least $1-\eta$, we have
    \begin{align*}\Bigl\lVert\frac{1}{n}\sum_{i=1}^{n} g_{ij} g_{il}^\T\Bigr\rVert_{\rm op} 
    &\lesssim \frac{R_\Theta^2 + \lambda_{\max}}{\lambda_{\min}^2}\,.
    \end{align*}
\end{lemma}
\begin{proof}
    By~\cref{lemma:gmm_covariance_concentration} and the condition on $n$,
    \begin{align*} 
        \Bigl\lVert\frac{1}{n}\sum_{i=1}^{n} g_{ij} g_{il}^\T\Bigr\rVert_{\rm op} &= \Bigl\lVert\frac{1}{n}\sum_{i=1}^{n} \Sigma_j^{-1}(X_i - \mu_j)(X_i - \mu_l)^\T\Sigma_l^{-1}\Bigr\rVert_{\rm op} \\
        &\le\frac{1}{\lambda_{\min}^2}\, \Bigl\lVert \frac{1}{n} \sum_{i=1}^n (X_i - \mu_j)(X_i - \mu_l)^\T \Bigr\rVert_{\rm op} \\
        &\le \frac{1}{\lambda_{\min}^2}\,\Bigl( \Bigl\lVert \frac{1}{n} \sum_{i=1}^n (\mu_j - \bar\mu)(X_i - \mu_l)^\T\Bigr\rVert_{\rm op} + \Bigl\lVert \frac{1}{n} \sum_{i=1}^n (\mu_j - \bar\mu) (\mu_l-\bar\mu)^\T \Bigr\rVert_{\rm op} \\
        &\qquad\qquad{}+ \Bigl\lVert \frac{1}{n} \sum_{i=1}^n (X_i - \bar\mu)(X_i -\bar\mu)^\T\Bigr\rVert_{\rm op}\Bigr) \\
        &\lesssim\frac{1}{\lambda_{\min}^2}\,\Bigl(R_\Theta \,\Bigl\lVert \frac{1}{n} \sum_{i=1}^n (X_i - \mu_l)\Bigr\rVert_2 + R_\Theta^2 + \lambda_{\max} \Bigr) \\
        &\lesssim\frac{1}{\lambda_{\min}^2}\,\Bigl(R_\Theta \,\Bigl\lVert \frac{1}{n} \sum_{i=1}^n (X_i - \bar\mu)\Bigr\rVert_2 + R_\Theta^2 + \lambda_{\max} \Bigr)\,. 
    \end{align*}
    Next, we note that for any unit vector $v$, $\norm{\langle X_i-\bar\mu, v\rangle}_{\psi_2} \lesssim R_\Theta+\sqrt{\lambda_{\max}}$ by~\eqref{eq:gmm_subg_in_dir}, and hence we have $\norm{n^{-1} \sum_{i=1}^n \langle X_i-\bar\mu,v\rangle}_{\psi_2} \lesssim (R_\Theta+\sqrt{\lambda_{\max}})/\sqrt n$. By a standard covering argument, this implies $\norm{n^{-1} \sum_{i=1}^n (X_i-\bar\mu)}_2 \lesssim (R_\Theta +\sqrt{\lambda_{\max}})\sqrt{(d+\log(1/\eta))/n}$ with probability at least $1-\eta$.  By the condition on $n$, we conclude the result.
\end{proof}

\begin{lemma}[Uniform convergence of the GMM log-likelihood]\label{lemma:gmm_log_likelihood_uniform_convergence}
    If $n \gg d\log^2 d$, 
    then with probability at least $1 - \eta$, the empirical log-likelihood converges uniformly to the population log-likelihood with the rate:
    \begin{align*}
        \sup_{\theta \in \Theta} \abs{\ell_n(\theta) - \ell^{\star}(\theta)} \lesssim \Bigl(\frac{R_\Theta^2\sqrt d + R_\Theta \sqrt{\lambda_{\max}}\,d}{\sqrt n} + \frac{\lambda_{\max}\, d\log n}{n}\Bigr) \,\frac{\log(1/\eta)}{\lambda_{\min}}\,. 
    \end{align*}
\end{lemma}
\begin{proof}
    First, we establish the local Lipschitz continuity of the log-likelihood. From~\eqref{gmm_gradient}, we have $\nabla_{\mu_j}\ell(x_i;\theta) = \gamma_{ij}\Sigma_j^{-1}(x_i - \mu_j)$ for each component $j$.
    Thus, we have the bound:
    $$ \norm{\nabla_{\mu_j}\ell(x;\theta)}_2 \le \norm{\Sigma_j^{-1}(x-\mu_j)}_2 \le \frac{1}{\lambda_{\min}}\, \norm{x-\mu_j}_2\,. $$
    Summing over all components, we obtain:
    \begin{align*}
        \norm{\nabla_\theta \ell(x;\theta)}_2^2 
        &= \sum_{j=1}^k \norm{\nabla_{\mu_j}\ell(x;\mu)}_2^2 
        \le \sum_{j=1}^k \norm{\Sigma_j^{-1}(x-\mu_j)}_2^2
        \lesssim \frac{R_\Theta^2}{\lambda_{\min}^2} + \frac{1}{\lambda_{\min}^2}\, \norm{x-\bar{\mu}}_2^2
    \end{align*}
    where the final inequality follows from the triangle inequality and the boundedness of $\Theta$. Consequently, for any $x \in \mathbb{R}^d$, the mapping $\theta \mapsto \ell(x; \theta)$ is Lipschitz continuous with constant at most $G(x) \deq C\lambda_{\min}^{-1}\,(R_\Theta + \norm{x-\bar{\mu}}_2)$ for some universal constant $C$. By \cref{cor:gmm_concentrations}, $G(X)$ is sub-Gaussian with $\mathbb{E} G(X) \lesssim (R_\Theta + \sqrt{\lambda_{\max} d})/\lambda_{\min}$. Moreover, if $n \gg \log(1/\eta)$, then with probability at least $1-\eta$, we have $G(X) \lesssim (R_\Theta+\sqrt{\lambda_{\max} d})/\lambda_{\min}$.

    By standard maximal inequalities~\cite[see][Example 19.7 and Corollary 19.35]{Vaart_1998},
    \begin{align*}
        \E\sup_{\theta\in\Theta}{|\ell_n(\theta) - \ell^\star(\theta)|} &\lesssim \frac{1}{\sqrt n} \int_0^\infty \sqrt{\log N_{[\,]}(\varepsilon, \{\ell(\cdot;\theta) : \theta \in\Theta\}, L^2(\PP))}\,d \varepsilon \\
        &\lesssim \frac{1}{\sqrt n} \int_0^\infty \sqrt{d \log \frac{R_\Theta\,\norm G_{L^2}}{\varepsilon}}\,d\varepsilon
        \lesssim \frac{R_\Theta\sqrt d\,(R_\Theta + \sqrt{\lambda_{\max} d})}{\lambda_{\min}\sqrt n}\,.
    \end{align*}
    Moreover, we can note that $\abs{\ell(x;(\bar\mu,\dotsc,\bar\mu))} \lesssim \norm{x-\bar\mu}_2^2/\lambda_{\min}$, and
    \begin{align*}
        \abs{\ell(x;\theta)}
        &\le \abs{\ell(x; (\bar\mu,\dotsc,\bar\mu))} + R_\Theta\,G(x)
        \lesssim \frac{\norm{x-\bar\mu}_2^2}{\lambda_{\min}} + \frac{R_\Theta\,(R_\Theta + \norm{x-\bar\mu}_2)}{\lambda_{\min}}
        \lesssim \frac{R_\Theta^2 + \norm{x-\bar\mu}_2^2}{\lambda_{\min}}\,.
    \end{align*}
    It follows that $F: x\mapsto C\, (R_\Theta^2 + \norm{x-\bar\mu}_2^2)/\lambda_{\min}$ is an envelope function for our class.
    By~\cite[Theorem 2.14.23]{vdVWel23Empirical},
    \begin{align*}
        \bigl\lVert\sup_{\theta\in\Theta}{|\ell_n(\theta) - \ell^\star(\theta)|}\bigr\rVert_{\psi_1}
        &\lesssim \E\sup_{\theta\in\Theta}{|\ell_n(\theta) - \ell^\star(\theta)|} + \frac{\log n}{n}\,\norm F_{\psi_1} \\
        &\lesssim \frac{R_\Theta\sqrt d\,(R_\Theta + \sqrt{\lambda_{\max} d})}{\lambda_{\min}\sqrt n} + \frac{\log n}{n}\,\frac{R_\Theta^2 + \lambda_{\max} d}{\lambda_{\min}}\,.
    \end{align*}
    If $n \gg d\log^2 d$, then the term $(R_\Theta^2 \log n)/(\lambda_{\min} n)$ can be dropped.
\end{proof}

\subsection{Proof of Theorem~\ref{theorem:gaussian_mixture_model_check_assumptions}}\label{app:gmm_pf}
  
To prove \cref{assumption: random convexity of regular part at true parameter}, we first establish \cref{assumption:holder_hessian}.
\paragraph{Proof of \cref{assumption:holder_hessian}.} 
We take $\gamma= 1$.

   We will upper bound the operator norm of a matrix $A$ by taking the maximum of the operator norms of each column in $A$. Precisely, for any $A \in \R^{kd \times kd},$ where $A$ has a $k \times k$ block structure with each block a $d \times d$ matrix, i.e., $$A=\begin{pmatrix}
    A_{11} & A_{12} & \cdots & A_{1k} \\ 
    A_{21} & A_{22} & \cdots & A_{2k} \\ 
    \vdots & \vdots & \ddots & \vdots \\ 
    A_{k1} & A_{k2} & \cdots & A_{kk}
    \end{pmatrix}\,,$$ we have the following inequality:
   \begin{equation}\label{gmm_hessian_block_bound}\norm{A}_{\rm op} \leq \max_{j\in [k]}\sum_{i=1}^{k}{\norm{A_{ij}}_{\rm op}} \leq k \max_{i,j\in [k]}{\norm{A_{ij}}_{\rm op}}\,.\end{equation} Therefore, it suffices to bound the operator norm for each $d \times d$ block of the Hessian.

    We first show that $\gamma_{ij}$ is Lipschitz in $\theta$ for all $i, j$, with high probability. Since $\nabla_{\mu_j} \gamma_{ij} = \gamma_{ij}\,(1-\gamma_{ij})\,g_{ij}$, and by \cref{cor:gmm_concentrations}, $g_{ij}$ is bounded by
    $K_1 \lambda_{\min}^{-1}\sqrt{\log(n/\eta)}$
    up to a universal constant with probability at least $1-\eta/2$ uniformly over all $i \in [n]$, $j \in [k]$, and $\theta\in\Theta$. 
    Therefore, $\gamma_{ij}$ is Lipschitz in $\theta$ with constant $L_\gamma = O(K_1\lambda_{\min}^{-1} \sqrt{\log(n/\eta)})$ with probability at least $1-\eta/2$ uniformly over all $i \in [n]$, $j \in [k]$.
    Let $f(\gamma_{ij}) \deq \gamma_{ij}\,(1-\gamma_{ij})$. Since $f^\prime(\gamma_{ij}) = 1-2\gamma_{ij}$ is bounded by $1$ in absolute value, we have $$\abs{\gamma_{ij}(\theta_1)\,(1-\gamma_{ij}(\theta_1)) - \gamma_{ij}(\theta_2)\,(1-\gamma_{ij}(\theta_2))} \leq \abs{\gamma_{ij}(\theta_1) - \gamma_{ij}(\theta_2)} \lesssim L_\gamma\,\norm{\theta_1 - \theta_2}_2\,.$$
    With probability at least $1-O(\eta)$, for any $\theta_1, \theta_2 \in \Theta$, $j\in [k]$, we have
   \begin{align*} 
        &\norm{\grad^2_{\mu_j,\mu_j} \ell_n(\theta_1) - \grad^2_{\mu_j,\mu_j} \ell_n(\theta_2)}_{\rm op}\\
        &\qquad= \Bigl\|\frac{1}{n}\sum_{i=1}^{n} \bigl(\gamma_{ij}(\theta_1)\, (1-\gamma_{ij}(\theta_1))\, g_{ij}(\theta_1)\, g_{ij}(\theta_1)^\T - \gamma_{ij}(\theta_1)\, \Sigma_j^{-1} \\ 
        &\qquad\qquad\qquad\qquad{}- \gamma_{ij}(\theta_2)\, (1-\gamma_{ij}(\theta_2))\, g_{ij}(\theta_2)\, g_{ij}(\theta_2)^\T + \gamma_{ij}(\theta_2)\, \Sigma_j^{-1}\bigr)\Bigl\|_{\rm op} \\
        &\qquad\leq \frac{1}{n}\sum_{i=1}^{n} \frac{1}{4}\, \norm{g_{ij}(\theta_1)\,(g_{ij}(\theta_1)-g_{ij}(\theta_2))^\T + (g_{ij}(\theta_1)-g_{ij}(\theta_2))\,g_{ij}(\theta_2)^\T}_{\rm op} \\ 
        &\qquad\qquad + \Bigl\lVert \frac{1}{n}\sum_{i=1}^{n} \bigl(\gamma_{ij}(\theta_1)\, (1-\gamma_{ij}(\theta_1))-\gamma_{ij}(\theta_2)\, (1-\gamma_{ij}(\theta_2))\bigr)\, g_{ij}(\theta_1)\, g_{ij}(\theta_1)^\T\Bigr\rVert_{\rm op} \\ 
        &\qquad\qquad +
        \frac{1}{n}\sum_{i=1}^{n} {\norm{\Sigma_j^{-1}}_{\rm op}\,\abs{\gamma_{ij}(\theta_1) - \gamma_{ij}(\theta_2)}} \\
        &\qquad\lesssim \frac{L_\gamma}{ \lambda_{\min}}\,\norm{\theta_1-\theta_2}_2+  L_\gamma\,\norm{\theta_1 - \theta_2}_2\,\Bigl\lVert\frac{1}{n} \sum_{i=1}^n g_{ij}(\theta_1)\,g_{ij}(\theta_1)^\T\Bigr\rVert_{\rm op}
        + \frac{L_\gamma}{\lambda_{\min}}\,\norm{\theta_1-\theta_2}_2 \\ 
        &\qquad\lesssim \frac{L_\gamma}{ \lambda_{\min}}\,\norm{\theta_1-\theta_2}_2 + \frac{L_\gamma \,(R_\Theta^2 + \lambda_{\max})}{\lambda_{\min}^2}\,\norm{\theta_1-\theta_2}_2\,. 
    \end{align*}
    In the last line, we applied~\cref{lemma:result_gmm_covariance_concentration}.
    If we assume that
    \begin{align*}
        n \gg \Bigl( \frac{R_\Theta^2 + \lambda_{\max}}{\lambda_{\min}}\Bigr)^2\,\bigl(d+\log(1/\eta)\bigr)\,,
    \end{align*}
    then it yields
   \begin{align*} 
        \norm{\grad^2_{\mu_j,\mu_j} \ell_n(\theta_1) - \grad^2_{\mu_j,\mu_j} \ell_n(\theta_2)}_{\rm op}
        &\lesssim \frac{L_\gamma\,(R_\Theta^2 + \lambda_{\max})}{\lambda_{\min}^2}\,\norm{\theta_1-\theta_2}_2\,.
    \end{align*}
    
   The same argument can be applied to the mixed Hessian $\nabla^2_{\mu_{j}, \mu_{j^\prime}}\ell_n = -\frac{1}{n}\sum_{i=1}^{n} \gamma_{ij} \gamma_{ij^{\prime}} g_{ij} g_{ij^\prime}^\T$ for $j \neq j^\prime$. 
   We have that $$\abs{\gamma_{ij}(\theta_1)\, \gamma_{ij^\prime}(\theta_1) - \gamma_{ij}(\theta_2)\, \gamma_{ij^\prime}(\theta_2)} \leq \abs{\gamma_{ij}(\theta_1) - \gamma_{ij}(\theta_2)} + \abs{\gamma_{ij^\prime}(\theta_1) - \gamma_{ij^\prime}(\theta_2)} \lesssim L_\gamma\,\norm{\theta_1 - \theta_2}_2\,.$$
   Therefore, for any $\theta_1, \theta_2 \in \Theta$,
   \begin{align*} 
        &\norm{\grad^2_{\mu_j, \mu_{j^\prime}} \ell_n(\theta_1) - \grad^2_{\mu_j, \mu_{j^\prime}} \ell_n(\theta_2)}_{\rm op}\\
        &\qquad = \Bigl\|\frac{1}{n}\sum_{i=1}^{n} \bigl(-\gamma_{ij}(\theta_1)\, \gamma_{ij^\prime}(\theta_1)\, g_{ij}(\theta_1)\, g_{ij^\prime}(\theta_1)^\T + \gamma_{ij}(\theta_2)\, \gamma_{ij^\prime}(\theta_2)\, g_{ij}(\theta_2)\, g_{ij^\prime}(\theta_2)^\T\bigr)\Bigr\|_{\rm op} \\
        &\qquad \leq \frac{1}{n}\sum_{i=1}^{n} \norm{g_{ij}(\theta_1)\,(g_{ij^\prime}(\theta_1)-g_{ij^\prime}(\theta_2))^\T + (g_{ij}(\theta_1)-g_{ij}(\theta_2))\,g_{ij^\prime}(\theta_2)^\T}_{\rm op} \\ 
        &\qquad\qquad + \Bigl\lVert\frac{1}{n}\sum_{i=1}^{n} \bigl(\gamma_{ij}(\theta_1)\, \gamma_{ij^\prime}(\theta_1)-\gamma_{ij}(\theta_2)\, \gamma_{ij^\prime}(\theta_2)\bigr)\, g_{ij}(\theta_1)\, g_{ij^\prime}(\theta_1)^\T\Bigr\rVert_{\rm op} \\ 
        &\qquad \lesssim \frac{L_\gamma}{\lambda_{\min}}\,\norm{\theta_1 - \theta_2}_2 + \frac{1}{\lambda_{\min}^2}\,\Bigl\lVert\frac{1}{n}\sum_{i=1}^{n} \bigl(\gamma_{ij}(\theta_1)\, \gamma_{ij^\prime}(\theta_1)-\gamma_{ij}(\theta_2)\, \gamma_{ij^\prime}(\theta_2)\bigr)\, (X_i - \mu_j)(X_i - \mu_{j'})^\T\Bigr\rVert_{\rm op}\,.
    \end{align*}
    Using a similar decomposition as the proof of~\cref{lemma:result_gmm_covariance_concentration}, one can bound the latter term by $L_\gamma\,(R_\Theta^2+\lambda_{\max})/\lambda_{\min}^2$.
    Therefore, we conclude that~\cref{assumption:holder_hessian} holds with
    \begin{align}\label{eq:gmm_s3}
        s_3 \lesssim \frac{(R_\Theta^2 + \lambda_{\max})\,(R_\Theta + \sqrt{d\lambda_{\max}})\,k\sqrt{\log(n/\eta)}}{\lambda_{\min}^3}\,.
    \end{align}

\paragraph{Proof of \cref{assumption: random convexity of regular part at true parameter}.}

By \cref{lemma: l_n convexity}, we have $$\inf_{\theta \in B(\theta^\star, r_0, r_1)}\lambda_{\min}(-\grad^2_{S_0} \ell_n(\theta)) \geq \frac{c_{S_0}^\star}{2} - \max_{\theta \in N_\epsilon(\theta^\star, r_0, r_1)}\norm{\grad^2_{S_0}\ell_n(\theta) - \grad^2_{S_0}\ell^\star(\theta)}_{\rm op}$$ with probability at least $1-\eta$, provided that $N_\varepsilon(\theta^\star,r_0,r_1)$ is an $\varepsilon$-covering net of $B(\theta^\star,r_0,r_1)$ with $\varepsilon \le \frac{c_{S_0}^\star}{2s_{3}}$ and $s_3$ is bounded as in~\eqref{eq:gmm_s3}.
It remains to show concentration of the Hessian around its population counterpart over the covering net.

Let $v=(v_1, \dotsc, v_k)$ be any unit vector in $\R^{kd_0}$ with support in $S_0$, where each $v_j \in \R^{d_0}$ for $j \in [k]$. For the rest of the proof, we will omit the subscript $S_0$ in $\grad^2_{S_0} \ell_n(\theta)$ for simplicity, but it is understood that we are only considering the Hessian with respect to regular parameters.
Since
    \begin{align*} 
        \bigl\langle \grad^2 \ell(\theta; X_i)\,v, v\bigr\rangle 
        &= \sum_{j,l=1}^{k} \bigl\langle\grad^2_{\mu_j,\mu_l} \ell(\theta; X_i)\,v_j, v_l\bigr\rangle \\
        &= \sum_{j=1}^{k} \bigl\langle \grad^2_{\mu_j,\mu_j} \ell(\theta; X_i)\,v_j, v_j\bigr\rangle + 2\sum_{j<l} \bigl\langle\grad^2_{\mu_j,\mu_l} \ell(\theta; X_i)\,v_j, v_l\bigr\rangle\,,
    \end{align*} 
    we can bound the subexponential norm of $\langle\grad^2 \ell(\theta; X_i)\,v, v\rangle$ by the subexponential norms of each term in the sum:
    \begin{align*} 
        &\bigl\lVert\bigl\langle \grad^2 \ell(\theta; X_i)\,v, v\bigr\rangle\bigr\rVert_{\psi_1}\\
        &\qquad \le \sum_{j=1}^{k} {\bigl\lVert \bigl\langle \grad^2_{\mu_j,\mu_j} \ell(\theta; X_i)\,v_j, v_j\bigr\rangle\bigr\rVert_{\psi_1}} + 2\sum_{j<l} {\bigl\lVert\bigl\langle\grad^2_{\mu_j,\mu_l} \ell(\theta; X_i)\,v_j, v_l\bigr\rangle\bigr\rVert_{\psi_1}} \\
        &\qquad = \sum_{j=1}^{k}{\bigl\lVert\gamma_{ij} \,(1-\gamma_{ij})\, v_j^\T g_{ij} g_{ij}^\T v_j - \gamma_{ij} v_j^\T \Sigma^{-1}_j v_j\bigr\rVert_{\psi_1}} + 2\sum_{j<l}{\bigl\lVert \gamma_{ij} \gamma_{il}v_j^\T g_{ij} g_{il}^\T v_l\bigr\rVert_{\psi_1}} \\ 
        &\qquad \leq \sum_{j=1}^{k}\bigl(\bigl\lVert v_j^\T \Sigma^{-1}_j (X_i - \mu_j)(X_i - \mu_j)^\T \Sigma^{-1}_j v_j\bigr\rVert_{\psi_1}  + \norm{v_j^\T \Sigma^{-1}_j v_j}_{\psi_1}\bigr) \\
        &\qquad\qquad{} + 2\sum_{j<l}{\bigl\lVert v_j^\T \Sigma^{-1}_j (X_i - \mu_j)(X_i - \mu_l)^\T \Sigma^{-1}_l v_l\bigr\rVert_{\psi_1}} \\
        &\qquad \lesssim \sum_{j=1}^{k}\bigl(\norm{v_j^\T \Sigma^{-1}_j(X_i - \mu_j)}_{\psi_2}^2  + \frac{\norm{v_j}_2^2}{\lambda_{\min}}\bigr) + \sum_{j<l}{\norm{v_j^\T \Sigma^{-1}_j(X_i - \mu_j)}_{\psi_2}\,\norm{v_l^\T \Sigma^{-1}_l(X_i - \mu_l)}_{\psi_2}}\,.
    \end{align*}
    
    By~\eqref{eq:gmm_subg_in_dir},
\begin{equation}\label{gmm_subexp_norm}\norm{v_j^\T g_{ij}}_{\psi_2} \lesssim \frac{R_\Theta+\sqrt{\lambda_{\max}}}{\lambda_{\min}}\,\norm{v_j}_2\,.\end{equation}
    Therefore,
    \begin{align*} 
        \sum_{j=1}^{k}{\norm{v_j^\T \Sigma^{-1}_j(X_i - \mu_j)}_{\psi_2}^2}
        &\lesssim \frac{R_\Theta^2+\lambda_{\max}}{\lambda_{\min}^2}\sum_{j=1}^{k}{\norm{v_j}_2^2}
        = \frac{R_\Theta^2+\lambda_{\max}}{\lambda_{\min}^2}\,, \\
        \sum_{j=1}^{k}\frac{\norm{v_j}_2^2}{\lambda_{\min}} 
        &= \frac{1}{\lambda_{\min}}\,,
        \end{align*}
    and
        \begin{align*}
        &\sum_{j<l}{\norm{v_j^\T \Sigma^{-1}_j(X_i - \mu_j)}_{\psi_2}\,\norm{v_l^\T \Sigma^{-1}_l(X_i - \mu_l)}_{\psi_2}}\\ 
        &\qquad \lesssim \sum_{j<l}\lambda_{\min}(\Sigma_j)^{-1}\,\lambda_{\min}(\Sigma_l)^{-1}\,\norm{v_j}_2\,\norm{v_l}_2\,(R_\Theta + \sqrt{\lambda_{\max}})^2\\
        &\qquad \leq \sum_{j<l}(\norm{v_j}_2^2+\norm{v_l}_2^2)\,\lambda_{\min}^{-2}\,(R_\Theta^2 + \lambda_{\max})\\
        &\qquad \lesssim k\lambda_{\min}^{-2}\,(R_\Theta^2 + \lambda_{\max})\,.
    \end{align*}
    Therefore, $$\bigl\lVert \inangle{\grad^2_{S_0} \ell(\theta; X_i)\,v, v}\bigr\rVert_{\psi_1} \lesssim k\lambda_{\min}^{-2}\,(R_\Theta^2 + \lambda_{\max})\,,$$ which implies
    \begin{align*}
        &\PP\bigl(\bigl\lvert \inangle{\grad^2_{S_0} \ell_n(\theta)\,v, v} - \E\inangle{\grad^2_{S_0} \ell_n(\theta)\,v, v}\bigr\rvert \geq t\bigr) \lesssim \exp\Bigl\{-n\,\Omega\Bigl( \frac{t^2\lambda_{\min}^{4}}{k^2\,(R_\Theta^2 + \lambda_{\max})^2} \wedge \frac{t\lambda_{\min}^{2}}{k\,(R_\Theta^2 + \lambda_{\max})}\Bigr)\Bigr\}\,.    
    \end{align*}
    This implies that
    $$\big|\inangle{\grad^2_{S_0} \ell_n(\theta)\,v, v} - \E\inangle{\grad^2_{S_0} \ell^{\star}(\theta)\,v, v}\big| \leq \frac{c_{S_0}^\star}{4} $$ with probability at least $1-\eta$, provided $n \gg \phi(\frac{k\,(R_\Theta^2 + \lambda_{\max})}{c_{S_0}^\star \lambda_{\min}^2}) \log(1/\eta)$, with $\phi(x) \deq x \vee x^2$. 

 Taking a union bound over all $v$ in an $\epsilon$-covering net of the unit sphere in $\R^{kd_0}$ with $\epsilon \ll \frac{c_{S_0}^\star}{2s_3}$, we have
    \begin{align*}
        &\sup_{\mu \in N_\epsilon(\theta^\star, r_0, r_1)}\norm{\grad^2_{S_0} \ell_n(\theta) - \grad^2_{S_0} \ell^{\star}(\theta)}_{\rm op} \leq \frac{c_{S_0}^\star}{4}
    \end{align*}
    with probability at least $1-\eta$, provided
    \begin{align*}
        n \gg \phi\Bigl(\frac{k\,(R_\Theta^2+\lambda_{\max})}{c_{S_0}^\star \lambda_{\min}^2}\Bigr)\,\Bigl(kd_0\log\frac{r_0 s_3}{c_{S_0}^\star} + \log(1/\eta)\Bigr)\,.
    \end{align*}
    Thus, \cref{assumption: random convexity of regular part at true parameter} holds with $c_{S_0} = c_{S_0}^\star/4$.

\paragraph{Proof of \cref{ass:subG}.}
For any unit vector $v \in \R^{kd}$,
\begin{align*}
\norm{\langle v, \grad\ell(\theta; X_1) - \E\grad\ell(\theta; X_1)\rangle}_{\psi_1} &= \Bigl\lVert \sum_{j=1}^k \langle v_j, \gamma_{1j} g_{1j} - \E[\gamma_{1j} g_{1j}]\rangle\Bigr\rVert_{\psi_1} \\
&\lesssim \sum_{j=1}^k{\|\langle v_j, g_{1,j}\rangle\|_{\psi_1}}
\lesssim \lambda_{\min}^{-1}\,(R_\Theta+\sqrt{\lambda_{\max}}) \sum_{j=1}^k \|v_j\|_2 \\
&\leq k^{1/2}\lambda_{\min}^{-1}\, (R_{\Theta} + \sqrt{\lambda_{\max}}) \text{ by } \eqref{gmm_subexp_norm}\,.
\end{align*} 
Therefore,~\cref{ass:subG} holds with $\sigma \lesssim k^{1/2} \lambda_{\min}^{-1}\,(R_\Theta+\sqrt{\lambda_{\max}})$.

\paragraph{Proof of \cref{ass:bounded_mixed_derivatives_ell_n}.}
    We apply \eqref{gmm_hessian_block_bound} to bound the whole matrix $\grad^2 \ell_n(\theta)$, so it suffices to bound each block $\grad^2_{\mu_j,\mu_j}\ell_n(\theta)$ and $\grad^2_{\mu_j, \mu_{j'}} \ell_n(\theta)$.
    \begin{align*} 
        \norm{\grad^2_{\mu_j,\mu_j}\ell_n(\theta)}_{\rm op}
        &= \Bigl\lVert\frac{1}{n}\sum_{i=1}^{n} \bigl(\gamma_{ij}\, (1-\gamma_{ij})\, g_{ij} g_{ij}^\T - \gamma_{ij} \Sigma^{-1}_{j}\bigr)\Bigr\rVert_{\rm op}
        \le \Bigl\lVert \frac{1}{n}\sum_{i=1}^{n}  g_{ij} g_{ij}^\T\Bigr\rVert_{\rm op} + \Bigl\lVert\frac{1}{n}\sum_{i=1}^{n} \gamma_{ij} \Sigma^{-1}_j\Bigr\rVert_{\rm op}\,, \\
        \norm{\grad^2_{\mu_j,\mu_j'} \ell_n(\theta)}_{\rm op}
        &= \Bigl\lVert\frac{1}{n}\sum_{i=1}^{n} \gamma_{ij} \gamma_{ij^{\prime}} g_{ij} g_{ij^\prime}^\T\Bigr\rVert_{\rm op}\,.
    \end{align*} 
     Therefore, $\|\grad^2_{\mu_j,\mu_j}\ell_n(\theta)\|_{\rm op} + \|\grad^2_{\mu_j,\mu_{j'}} \ell_n(\theta)\|_{\rm op} \lesssim \frac{R_\Theta^2 + \lambda_{\max}}{\lambda_{\min}^2}$ with probability at least $1-\eta$ when $n\gg \frac{R_\Theta^4 + \lambda_{\max}^2}{\lambda_{\min}^2}\,\bigl(d+\log(1/\eta)\bigr)\,,$
    where the inequalities follow from \cref{lemma:result_gmm_covariance_concentration}.
    By \eqref{gmm_hessian_block_bound}, we see that \cref{ass:bounded_mixed_derivatives_ell_n} holds with
    $$s_2 \lesssim k\,\frac{R_\Theta^2 + \lambda_{\max}}{\lambda_{\min}^2}\,.$$

\paragraph{Proof of \cref{assumption_contraction:compact_parameter_space}.}
    This follows from \cref{assumption:gmm_dimension}, \cref{cor:non_reg_rate}, and \cref{lemma:uniform convergence regular part}.

\paragraph{Proof of \cref{assumption:exponential_decay}.}
We decompose, 
    for any $\theta$ outside $B(\hat\theta, r_0,r_1)$,
    \begin{align*} 
        \ell_n(\theta) - \ell_n(\hat\theta) &= \underbrace{\ell_n(\theta) - \ell^{\star}(\theta)}_{(i)} + \underbrace{\ell^{\star}(\theta) - \ell^{\star}(\theta^\star)}_{(ii)} + \underbrace{\ell^{\star}(\theta^\star) - \ell^{\star}(\hat\theta)}_{(iii)} + \underbrace{\ell^{\star}(\hat\theta) - \ell_n(\hat\theta)}_{(iv)}\,.
    \end{align*}
    We control these terms using the uniqueness of the maximizer of $\ell^{\star}$, consistency of the log-likelihood, and the uniform convergence of the log-likelihood.
    In particular, $(i)$ and $(iv)$ are differences between the empirical log-likelihood and the population log-likelihood, which converges to $0$ uniformly over $\Theta$ by \cref{lemma:gmm_log_likelihood_uniform_convergence}.
    Term $(ii)$ is the difference between the population log-likelihood at $\theta$ and the population log-likelihood at the mode, which is bounded by $-\zeta^\star$ by \cref{assumption:exponential_decay_l}. Finally, $(iii)$ is the difference between the population log-likelihood at the true parameter and the mode of empirical log-likelihood. Recall that $\theta_{S_1}^\star=\hat\theta_{S_1}$ by \cref{cor:non_reg_rate}. Therefore, by Taylor expansion to the second order and \cref{lemma:uniform convergence regular part}, 
    \begin{equation*}
        \ell^{\star}(\theta^\star) - \ell^{\star}(\hat\theta) \leq \frac{s_2}{2} \,\norm{\hat\theta_{S_0} - \theta^\star_{S_0}}^2_2
        \lesssim \frac{s_2\sigma^2}{c_{S_0}^2}\,\Bigl(\frac{d_0 + \log(1/\eta)}{n} + \frac{(d_0 +\log(1/\eta))^2}{n^2}\Bigr)\,.
    \end{equation*}
     Collecting the error terms, we conclude that with probability at least $1-O(\eta)$, if
     \begin{align*}
         n &\gg d\log^2 d + \frac{(R_\Theta^2 + d\lambda_{\max})\,R_\Theta^2 d \log^2(1/\eta)}{\lambda_{\min}^2\,(\zeta^\star)^2} + \frac{d\lambda_{\max}}{\lambda_{\min} \zeta^\star} \log(1/\eta) \log\Bigl[\frac{d\lambda_{\max}}{\lambda_{\min} \zeta^\star} \log(1/\eta)\Bigr] \\[0.25em]
         &\qquad{} + \phi\Bigl(\frac{k\,(R_\Theta^2 + \lambda_{\max})}{c_{S_0}^\star\,\lambda_{\min}^2\,(\zeta^\star)^{1/2}}\Bigr)\,\bigl(d_0+\log(1/\eta)\bigr)\,,
     \end{align*}
     where $\phi(x) \deq x \vee x^2$, then \cref{assumption:exponential_decay} holds with $\zeta = \zeta^\star/2$.

\end{document}